\documentclass[11pt ,twoside]{amsart}  
\usepackage{graphicx}
\usepackage{multirow}
\usepackage{hyperref}
\usepackage{bbm}
\usepackage{amsmath,amssymb,amsfonts, amsthm,epsfig}
\usepackage{verbatim,color}
\usepackage{enumerate}
\usepackage{mathtools}
\mathtoolsset{showonlyrefs}

\numberwithin{equation}{section}

\newtheorem{thma}{Theorem}[section]
\newtheorem{lemma}[thma]{Lemma}

\newtheorem{defi}[thma]{Definition}
\newtheorem{prop}[thma]{Proposition}

\newtheorem{claim}{Claim}
\newtheorem{remark}{Remark}

\newcommand{\eps}{\varepsilon}

\newtheorem*{thma*}{Theorem}

\renewcommand{\mod}{\;\text{mod}\;}

\begin{document}

\title[Local Geometry in the two-periodic Aztec diamond]{Local Geometry of the rough-smooth interface in the two-periodic Aztec diamond}

\author{Vincent Beffara, Sunil Chhita and Kurt Johansson} 


\maketitle
\begin{abstract}
Random tilings of the two-periodic Aztec diamond contain three macroscopic regions: frozen, where the tilings are deterministic; rough, where the correlations between dominoes decay polynomially; smooth, where the correlations between dominoes decay exponentially. In a previous paper, the authors found that a certain averaging of height function differences at the rough-smooth interface converged to the extended Airy kernel point process. In this paper, we augment the local geometrical picture at this interface by introducing well-defined lattice paths which are closely related to the level lines of the height function. We show, after suitable centering and rescaling, that a point process from these paths converge to the extended Airy kernel point process provided that the natural parameter associated to the two-periodic Aztec diamond is small enough.

\end{abstract}

\section{Introduction} \label{sec:intro}

Random tiling models have in recent years provided a rich source of stochastic processes related to random matrix theory statistics; see~\cite{Joh17} and references therein. In particular, restricted to certain domains, random tilings of large domains may separate into macroscopic regions featuring facets at the boundary.   In these facets, the random tiling appears ordered and the measure is said to be \emph{frozen} (or solid).  Away from these facets, the measure can be \emph{rough} (also known as liquid) or \emph{smooth} (also known as gas)\footnote{We adopt the nomenclature from statistical physics instead of solid, liquid and gas. These are not states of matter. } with the distinction depending on whether the correlations between the tiles decay polynomially or exponentially.  For some classes of these random tilings, the curves separating these regions can be analyzed~\cite{CKP:01, KO:07, ADPZ:20}.  This feature is mathematically established for random tilings but should hold for other similar models such as the six vertex model; see e.g.~\cite{CS:16, Agg:19}.

For many random tiling models containing just a frozen and a rough phase, there is a lattice path which separates the two phases.  It has been shown in some of these models that the fluctuations of this path, under suitable scaling and centering, is given by the Airy process\footnote{In this paper, by Airy process, we mean the Airy-2 process.}, and this feature is believed to be universal.  The question motivating a series of papers including this one is whether there is a similar path separating a rough and smooth phase and are its fluctuations, after suitable centering and rescalings, also described by the Airy process?

In this paper, we focus on a particular random tiling model, the \emph{two-periodic Aztec diamond} which is defined fully below. This model was introduced in~\cite{CY:13} and its correlation kernel\footnote{more precisely, a formula was found for entries of the inverse Kasteleyn matrix} was computed in a long-winded computation.  This formula was simplified in~\cite{CJ:16} into a more convenient form for asymptotic computations, and the asymptotics were computed along a diagonal of the Aztec diamond, including at the rough-smooth boundary.  In a later development, Duits and Kuijlaars~\cite{DK:17} gave a different and more systematic approach to compute a particle correlation kernel and analyze its asymptotics in the two-periodic Aztec diamond using multiple orthogonal polynomials. Yet another approach based on Wiener-Hopf factorization of matrix-valued symbols was given in ~\cite{BD:19}. Further developments of these approaches have been particularly fruitful in other models~\cite{Be:19, CDKL:19}.

However, these results did not give any significant insight into the geometry of the interface between the smooth and rough regions nor the limiting behavior.  The two main obstacles being that there was no clear definition of paths which could separate the two macroscopic regions, and the methods available at the time, only gave the asymptotics of the dominoes and did not directly connect the computations with the asymptotic picture which is overwhelmingly evident from simulations.  Put simply, the asymptotics of the inverse Kasteleyn matrix for the dimers/dominoes at the rough-smooth boundary involves a full-plane smooth term with a part of the Airy kernel as a correction term. Nevertheless, we introduced a random signed measure in~\cite{BCJ:16} built by taking specific averages of height function differences between faces. The height function gives a random surface interpretation of the random tiling model and is defined precisely below.
 After quite a subtle computation, we showed that this signed measure converged to the extended Airy kernel point process.  

However, this recent development did not specify any lattice paths which separate the rough and smooth regions as one would expect with the presence of the extended Airy kernel point process.    In this paper, we  find that there is a way to define a sequence of lattice paths such that the net (signed) number of lattice paths through appropriate intervals  converges to the extended Airy kernel point process, provided that the natural parameter associated to the two-periodic Aztec diamond is small enough.  This restriction is due to technical details of our proof and we do not believe there to be any different behavior outside this restriction.   The significance of our result is that it shows that there are paths separating the rough and smooth regions that are in a sense described by the Airy kernel point process in the limit. Thus we take a step towards understanding what is
apparent from our simulations.  Unfortunately, we fall short of proving the overall geometry as well as showing that there is a last path converging to the Airy process.  The rest of this introduction is devoted to giving the main definitions of the model, defining the extended Airy kernel, giving an informal version of the main theorem, which is stated precisely later in the paper.

\subsection{The two-periodic Aztec diamond} \label{sec:intro:overview}
An \emph{Aztec diamond graph of size $n$} is a bipartite graph which contains white vertices given by
\begin{equation}
\mathtt{W}= \{(i,j): i \mod 2=1, j \mod 2=0, 1 \leq i \leq 2n-1, 0 \leq j \leq 2n\}
\end{equation}
and black vertices given by
\begin{equation}
\mathtt{B}= \{(i,j): i \mod 2=0, j \mod 2=1, 0 \leq i \leq 2n, 1\leq j \leq 2n-1\}.
\end{equation}
The edges of the Aztec diamond graph are given by $\mathtt{b} -\mathtt{w}= \pm e_1,\pm e_2$ for $\mathtt{b} \in \mathtt{B}$ and $\mathtt{w} \in \mathtt{W}$, where $e_1=(1,1)$ and $e_2=(-1,1)$.  The coordinate of a face in the graph is defined to be the coordinate of its center.
For an Aztec diamond  graph of size $n=4m$ with $m \in \mathbb{N}$, define the \emph{two-periodic Aztec diamond}, $D_m$, to be an Aztec diamond graph with edge  weights $a$ for all edges incident to the faces $(i,j)$ with  $(i+j)\mod 4=2$ and edge weights $b$ for all the edges incident to the faces $(i,j)$ with  $(i+j) \mod 4=0$; see Fig.~\ref{fig:weights}. We call the faces $(i,j)$ with $(i+j)\mod 4=2$ to be the \emph{$a$-faces} and the faces $(i,j)$ with $(i+j)\mod 4=0$ to be the \emph{$b$-faces}. With this setup, one sees that there are  two types of white vertices and black vertices depending on the weights of the incident edges. These are given by
\begin{equation} \label{white:parity}
\mathtt{W}_i= \{(x,y) \in \mathtt{W}: x+y \mod4=2i+1\} \hspace{5mm}\mbox{for } i \in\{0,1\}
\end{equation}
and
\begin{equation}\label{black:parity}
\mathtt{B}_i= \{(x,y) \in \mathtt{B}: x+y \mod4=2i+1\} \hspace{5mm}\mbox{for } i \in\{0,1\}.
\end{equation}
\begin{figure}
\begin{center}
\includegraphics[height=5cm]{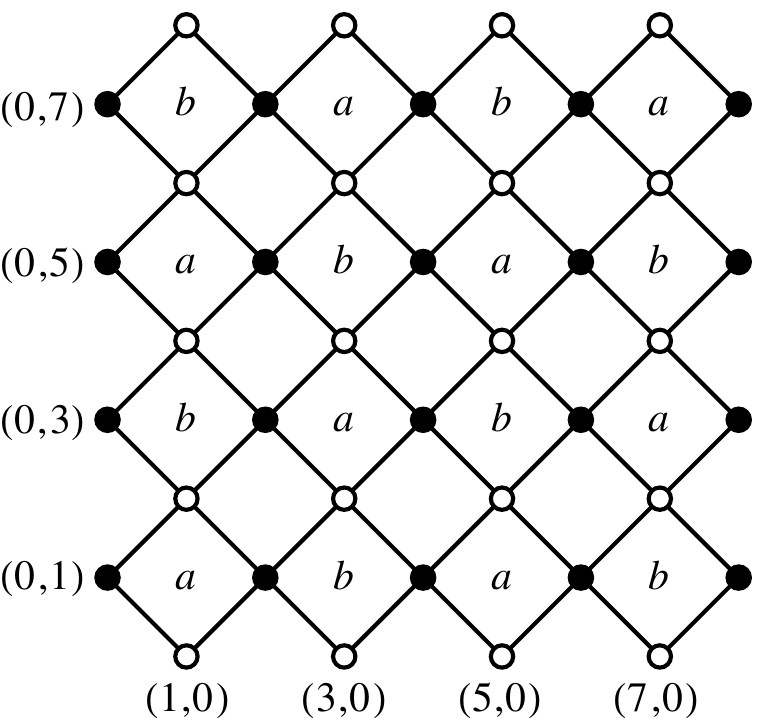}
\caption{The two-periodic Aztec diamond with $m=1$. Edges incident to the faces labelled $a$ have edge weight $a$ while edges incident to the faces labelled $b$ have edge weight $b=1$.}
\label{fig:weights}
\end{center}
\end{figure}

A dimer configuration on the Aztec diamond graph is a subset of edges so that each vertex is incident to exactly one edge. Such edges in a configuration are called dimers.  A probability measure is defined on the two-periodic Aztec diamond graph by picking each dimer configuration with probability proportional to the product of the edge weights of that dimer configuration.

Throughout the rest of the paper, we refer to an  $a$-dimer ($b$-dimer resp.) to be a dimer covering an $a$-edge ($b$-edge resp.).  We say that an $a$-dimer is \emph{incident to a particular $b$-face} if it shares a common vertex with that $b$-face.


\subsection{Squishing}
Assign an orientation to each edge of the Aztec diamond, by prescribing an arrow from each white vertex to its incident black vertices.  For the two-periodic Aztec diamond graph, define the \emph{squishing procedure} as the operation which contracts each $b$-face while simultaneously increasing the size of the $a$-faces so that the $a$-face coordinates remain unchanged and keeping the orientation; see Fig.~\ref{fig:squishAztec12} for an example. The resulting graph consists of only $a$-edges and $a$-faces while for the dimers, only the $a$-dimers are visible.  This operation was inspired by a similar operation for the honeycomb graph in~\cite{You:09} and we adopt the naming convention. Label $\tilde{D}_m$ to be the graph $D_m$ after the squishing procedure.

\begin{figure}
\begin{center}
\includegraphics[height=6cm]{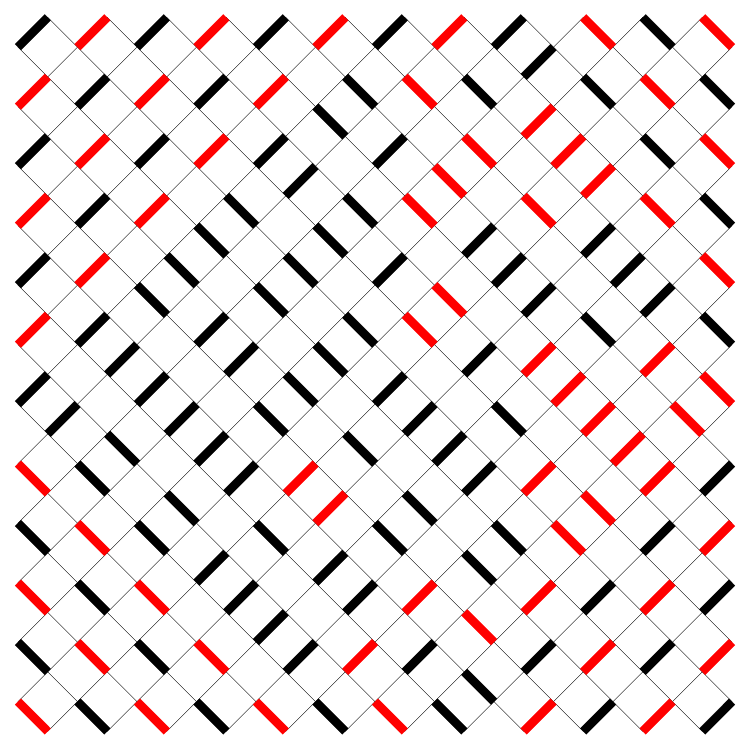}
\includegraphics[height=6cm]{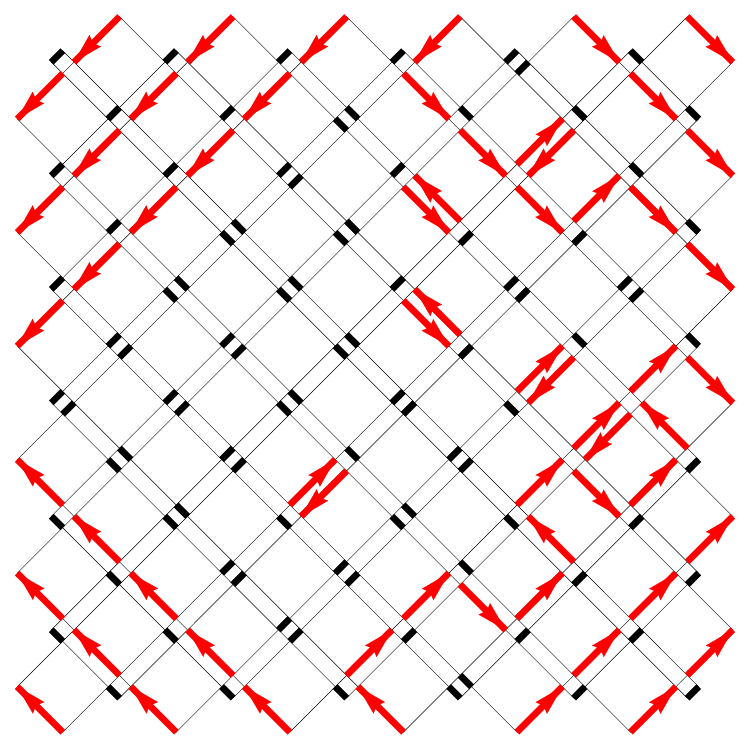}
\caption{The squishing procedure  for an Aztec diamond of size $12$. In each figure, the $a$-dimers are drawn in red while the $b$-dimers are drawn in black. The left figure shows the original dimer configuration while the right figure shows the same dimer configuration with a smaller size of $b$-face. We have only put the orientation on the $a$-dimers. }	
\label{fig:squishAztec12}
\end{center}
\end{figure}

 After this procedure, we call a \emph{double edge} to be the result of two $a$-dimers contracting to the same edge.  Observe that there is a parity condition for the number of incident $a$-dimers for each $b$-face. That is, the number of incident $a$-dimers for each $b$-face is either $0$, $2$ or $4$ since odd numbers invalidate the dimer covering.  A consquence of this parity condition, as we explain in Section~\ref{subsec:squish}, is that the $a$-dimers are either part of double edges, (oriented) \emph{loops} or \emph{paths}, where the precise definitions of loops and paths are given in Section~\ref{subsec:squish}\footnote{A careful choice needs to be made to make paths and loops well-defined. This choice is given in Section~\ref{subsec:squish}.}.  Heuristically, paths can be thought of as connected sequences of $a$-dimers which start at  either the top or bottom boundaries of the Aztec diamond and terminate at either the left or right boundaries of the Aztec diamond. Fig.~\ref{fig:squishAztec12} shows paths and double edges, while Fig.~\ref{fig:n300} shows double edges, loops and paths  in a larger simulation.

\begin{figure}
\begin{center}
\includegraphics[height=12cm]{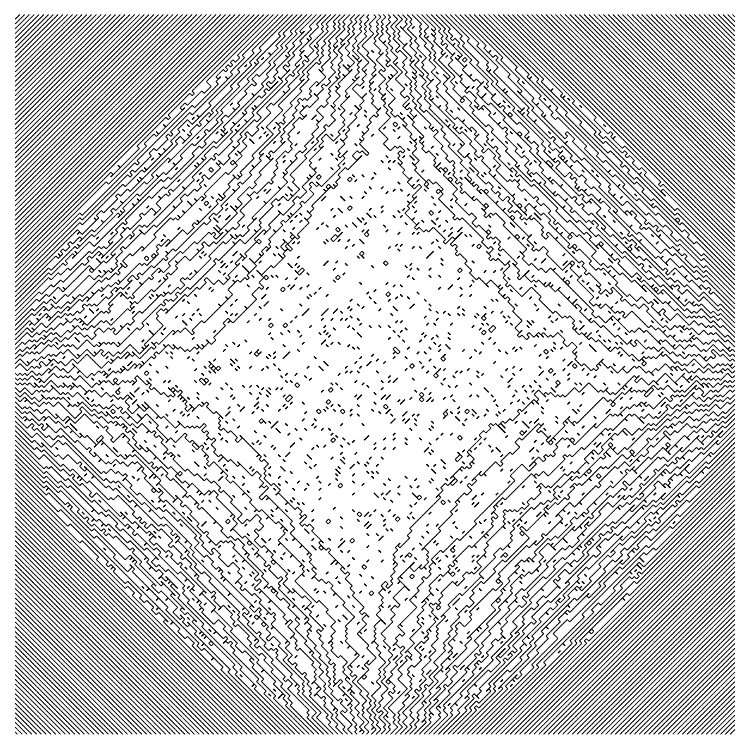}
	\caption{A random dimer configuration of a two-periodic Aztec diamond of size 300 with $a=0.5$ after the squishing procedure. We have suppressed the orientation.} 
	\label{fig:n300}
\end{center}
\end{figure}

\subsection{Extended Airy kernel point process}

Following~\cite{Joh:03}, let $\mathbbm{I}_{A}$ be the indicator function for some set $A$ and denote $\mathbbm{I}$ to be the identity matrix or operator. 
Let $\mathrm{Ai}(\cdot)$ denote the standard Airy function, and define
\begin{equation} \label{eq:Airymod}
\begin{split}
\tilde{\mathcal{A}}(\tau_1,\zeta_1;\tau_2,\zeta_2)&= 
\int_0^\infty e^{-\lambda (\tau_1-\tau_2) } \mathrm{Ai} (\zeta_1 +\lambda) \mathrm{Ai} (\zeta_2+\lambda) d\lambda.\\
\end{split}
\end{equation}
and
\begin{equation}\label{eq:Airyphi}
\phi_{\tau_1,\tau_2} (\zeta_1 ,\zeta_2) =
 \mathbbm{I}_{\tau_1<\tau_2} \frac{1}{\sqrt{4 \pi (\tau_2-\tau_1)}} e^{-\frac{(\zeta_1-\zeta_2)^2}{4(\tau_2-\tau_1)}-\frac{(\tau_2-\tau_1)(\zeta_1+\zeta_2)}{2}+\frac{(\tau_2-\tau_1)^3}{12}},
\end{equation}
which is referred to as the \emph{Gaussian part} of the extended Airy kernel; see~\cite{Joh:03}. The \emph{extended Airy kernel}, ${\mathcal{A}}(\tau_1,\zeta_1;\tau_2,\zeta_2)$, is defined by
\begin{equation}\label{eq:extendedAiry}
{\mathcal{A}}(\tau_1,\zeta_1;\tau_2,\zeta_2)=\tilde{\mathcal{A}}(\tau_1,\zeta_1;\tau_2,\zeta_2)-\phi_{\tau_1,\tau_2} (\zeta_1 ,\zeta_2).
\end{equation}
Let $\beta_1 < \dots < \beta_{L_1}$, $L_1 \geq 1$ be fixed given real numbers and let $A_p=[\alpha_p^l,\alpha_p^r]$ for $\alpha_p^l < \alpha_p^r$ and $1 \leq p \leq L_2$ be finite disjoint intervals in $\mathbb{R}$.  Write 
$$
\Psi(x) = \sum_{p=1}^{L_2} \sum_{q=1}^{L_1}w_{p,q}\mathbbm{I}_{\{\beta_q\} \times A_p}(x),
$$
where $w_{p,q}$ are some given complex numbers for $1 \leq p \leq L_2$, $1 \leq q \leq L_1$.  The \emph{extended Airy kernel point process}, $\mu_{\mathrm{Ai}}$, is a determinantal point process  on $L_1$ lines $\{\beta_1,\dots, \beta_{L_1}\} \times \mathbb{R}$ defined by  
\begin{equation}\label{eq:extendedAirydef}
\mathbb{E}\left[ \exp \left( \sum_{p=1}^{L_2} \sum_{q=1}^{L_1}w_{p,q} \mu_{\mathrm{Ai}} (\{\beta_q\} \times A_p) \right)\right] = \det ( \mathbbm{I}+(e^{\Psi}-1) \mathcal{A})_{L^2(\{\beta_1,\dots, \beta_{L_1}\} \times \mathbb{R}}
\end{equation}
for $w_{p,q} \in \mathbb{C}$.

\subsection{Informal Statement of Theorem} \label{intro:mainresult}

Here, we state informally our main theorem, using the informal definitions for paths.  The main theorem will be made precise below in Section~\ref{sec:mainthm}. 

In Fig.~\ref{fig:n300}, we see that there are paths starting and ending at the boundary, as well as double edges and loops, some of which may be attached to the paths.  These notions will be defined precisely below in Section~\ref{subsec:squish}. If we remove the loops and the double edges, we are left with just paths.  The regions between these paths will be called \emph{corridors}.  These corridors go all the way up to the boundary and the height at the boundary defines the corridor height for all faces in the corridor\footnote{{ Our convention for the height function is given in Section~\ref{subsec:squish}}}.  Differences between corridor heights on faces gives the (signed) number of paths between faces.    For $ 1 \leq p \leq L_2$, $1 \leq q \leq L_1$, the intervals $\{\beta_q\} \times A_p$, can be rescaled and put between faces at the rough-smooth boundary, and we can consider the corridor height differences between faces.  Dividing this by $4$ gives a quantity that we denote by $\kappa_m (\{\beta_q\} \times A_p)$.

\begin{thma*}[Informal version of Theorem~\ref{thm:main}]
	Assume that  $a \in (0,1/3)$. The random variables $\kappa_m (\{\beta_q\} \times A_p)$, $1 \leq q \leq L_1$, $1 \leq p \leq L_2$ converge jointly in law to the random variables $\mu_{\mathrm{Ai}} (\{ \beta_q\} \times A_p)$, $1 \leq q \leq L_1$, $1 \leq p \leq L_2$, as $m \to \infty$.   
\end{thma*}

A couple of remarks are in order. 

\begin{remark}
	\begin{enumerate}
		\item For the statement of the theorem, we require that $a \in (0,1/3)$, but there is a smooth phase for all $a \in (0,1)$.  This is a technical restriction and we believe that the theorem should hold for all values of $a \in (0,1)$. 
		\item  We cannot show that there actually is a last path in the third quadrant connecting the bottom and left boundaries as we move along the diagonal. Paths can in principle behave in strange ways but these strange behaviors should happen with very low probability. 
			
	\end{enumerate}

\end{remark}

These remarks are summarized into a conjecture after the statement of the main theorem below.

\subsection*{Acknowledgements}
SC acknowledges the support of the UK Engineering and Physical Sciences Research Council (EPSRC) grant EP/T004290/1.
KJ acknowledges the support of the Swedish Research Council (VR) and grant KAW
2015.0270 of the Knut and Alice Wallenberg Foundation. We would like to thank the referees for their careful readings and comments on an earlier version of this paper.

\section{Combinatorial definitions} 
\label{subsec:squish}

In this section, we expand on the squishing procedure introduced before, giving the concepts of loops and paths and their correspondence with the height function.

As mentioned in the introduction, we assign outgoing edges from each white vertex to its incident black vertices.  For a dimer covering $d$ on $D_m$, let $\tilde{d}$ denote the dimer covering after the squishing procedure, that is, $\tilde{d}$ records the collection of $a$-dimers present in a configuration $D_m$, with  prescribed arrows from white vertices to black vertices.

 The \emph{height function}, an idea usually attributed to Thurston~\cite{Thu:90}, is defined for the two-periodic Aztec diamond at the center of each face of the Aztec diamond graph. The height function is determined by the height differences as we traverse between each pair of adjacent faces influenced by whether a dimer covers the shared edge and the prescribed arrow in the following way:
\begin{itemize}
	\item a height change of $+3$ if the shared edge is covered by a dimer and the prescribed arrow points to the left (from the starting face), 
	\item a height change of $-3$ if the shared edge is covered by a dimer and the prescribed arrow points to the right (from the starting face), 
	\item a height change of $-1$ if the shared edge is not covered by a dimer and the prescribed arrow points to the left (from the starting face), and
	\item a height change of $+1$ if the shared edge is not covered by a dimer and the prescribed arrow points to the right (from the starting face). 
\end{itemize}
We assign the height at the face $(0,0)$ (outside of the Aztec diamond graph) to be equal to 1.  The height function on the faces bordering the Aztec diamond graph are deterministic and given by the above rule.   Notice that the height function is \emph{divergence free} around each white and black vertex in the Aztec diamond.

Define the \emph{$a$-height function}, denoted by $h^a(f)$ where $f$ is an $a$-face, to be the height function of the two-periodic Aztec diamond restricted to the $a$-faces. The $a$-height change between two $a$-faces which share an edge after the squishing procedure is in $\{-4,0,4\}$. It follows that the $a$-height function is completely determined after the squishing procedure but the $a$-height function does not recover the original dimer covering; see Fig.~\ref{fig:squishAztec12a} for an example.
\begin{figure}
\begin{center}
\includegraphics[height=6cm]{n12heights2b.pdf}
\includegraphics[height=6cm]{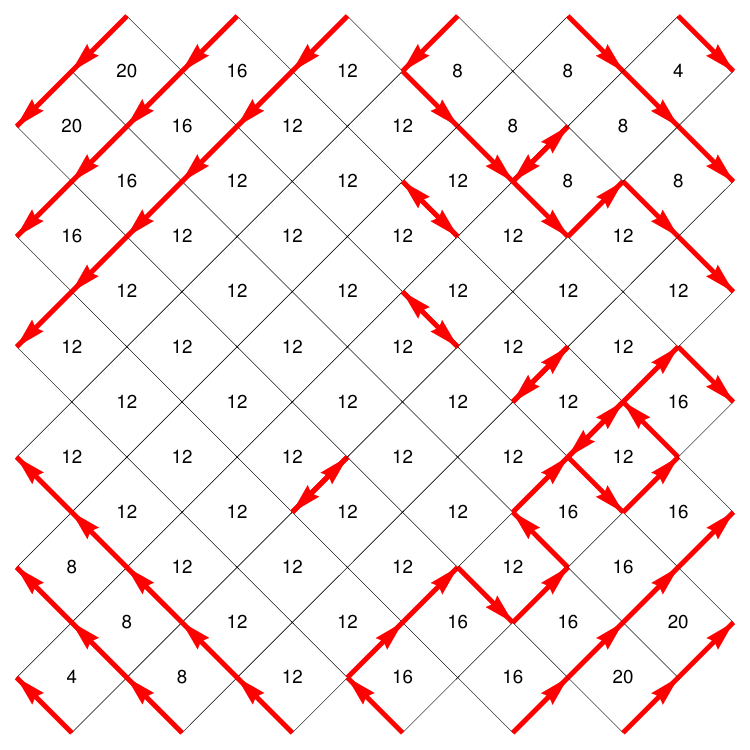}
	\caption{The squishing procedure from Fig.~\ref{fig:squishAztec12} and the right figure shows the $a$-height function with the size of $b$-face set to zero.   }	
\label{fig:squishAztec12a}
\end{center}
\end{figure}
An easy way to see this is that if there is no height change between two $a$-faces then this is either from a double edge or from no $a$-dimers on the shared edge between faces; we cannot distinguish between these configurations from the $a$-height function. The $a$-height function, by construction, is divergence free on the $a$-faces around each $b$-face. The possible $a$-height changes when traversing the $a$-faces around each $b$-face are no change; one $a$-height change of $\pm 4$ and another $a$-height change of $\mp 4$; a height change of $\pm 4$ followed $\mp 4$ followed by $\pm 4$ followed by $\mp 4$. These indicate that the maximum $a$-height change between $a$-faces that are incident to the same $b$-face but do not share an edge after the squishing procedure is 4.  These can easily be verified by considering all local configurations around each $b$-face; see Fig.~\ref{fig:localconfigs}.

\begin{figure}
\begin{center}
\includegraphics[height=3cm]{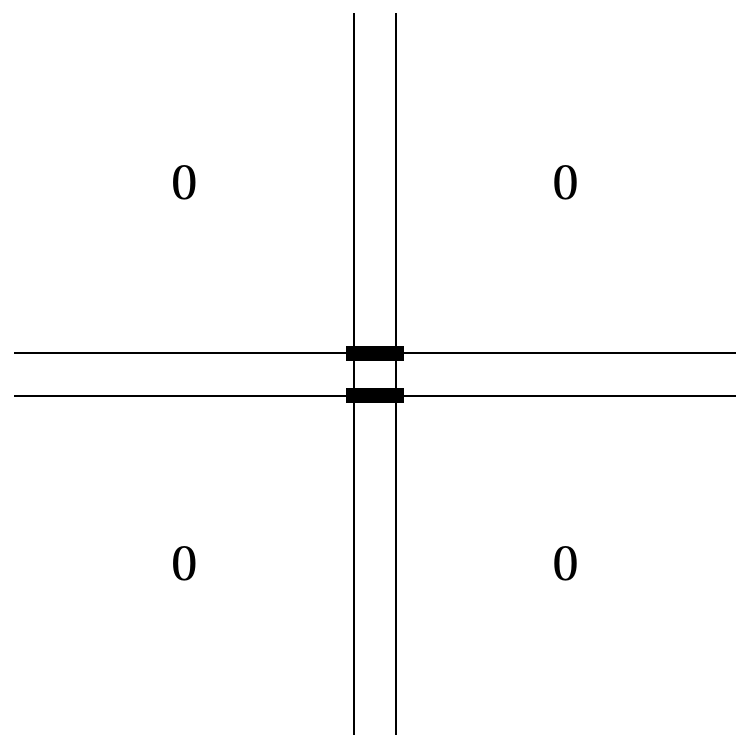}
\includegraphics[height=3cm]{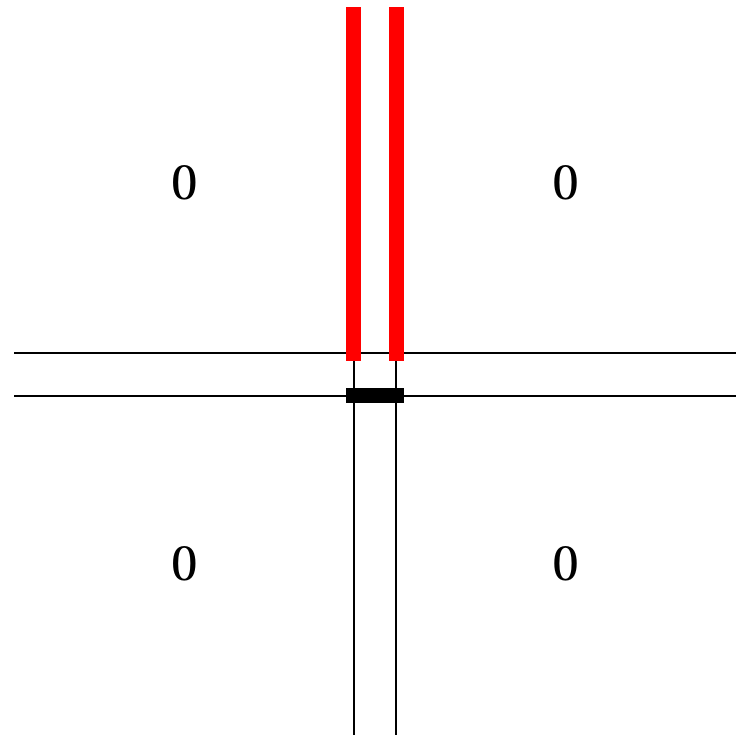}
\includegraphics[height=3cm]{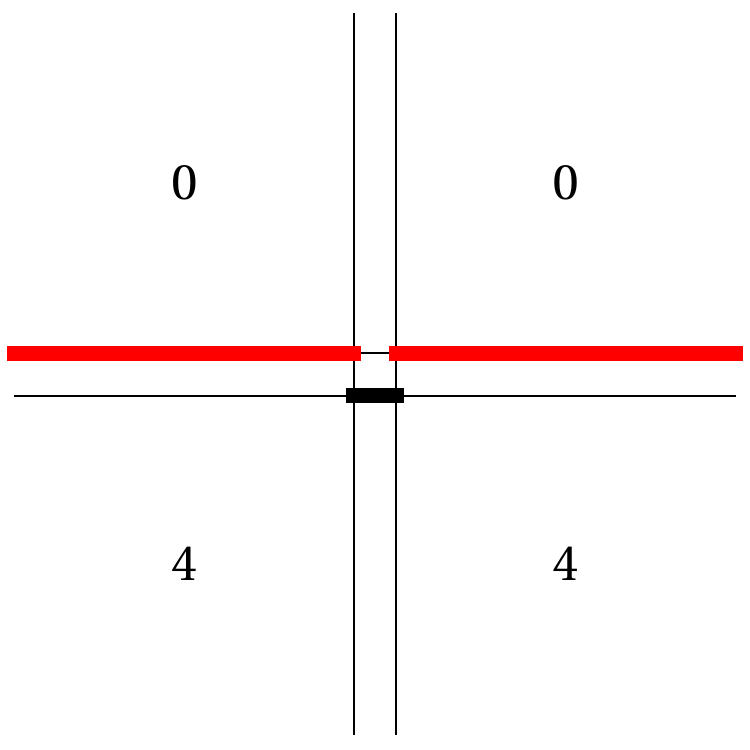}
\includegraphics[height=3cm]{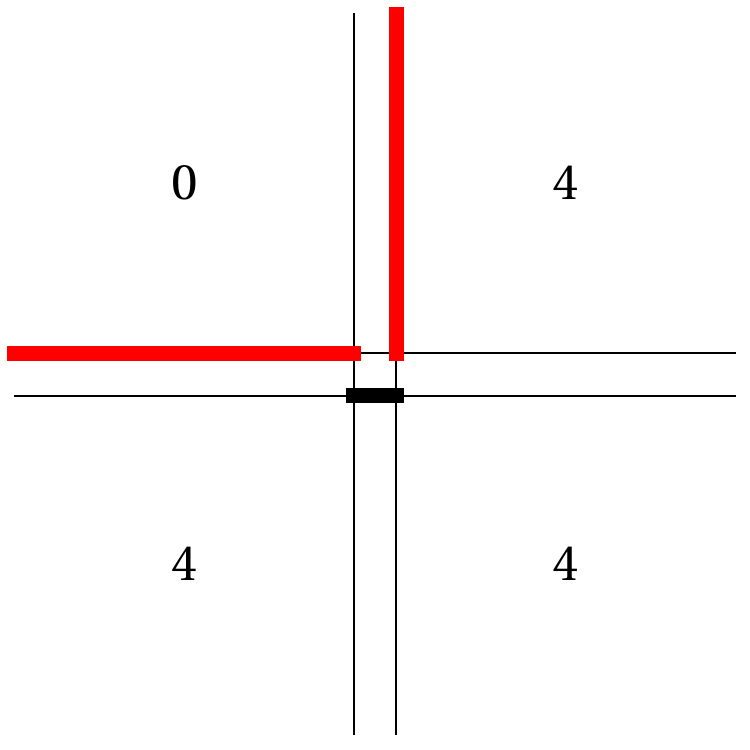}
\includegraphics[height=3cm]{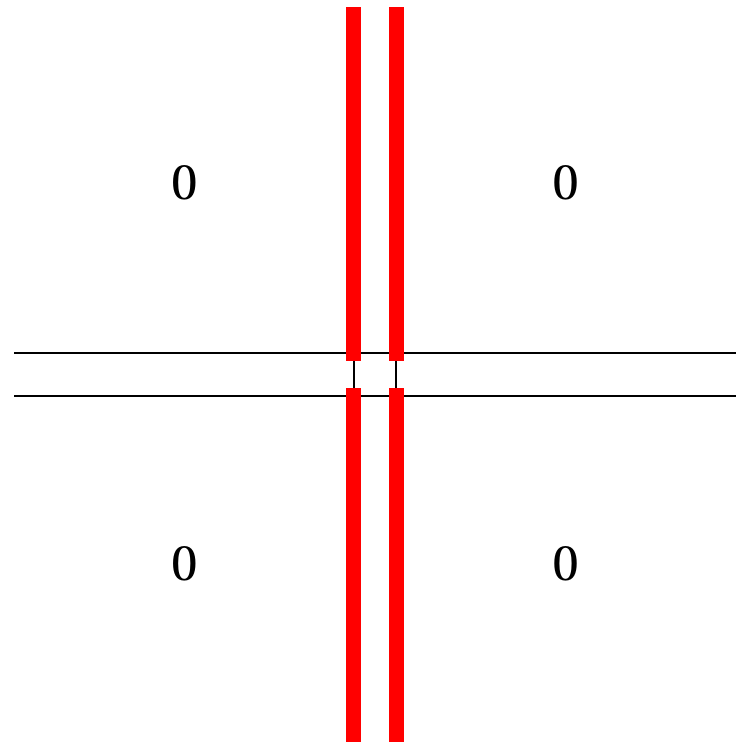}
\includegraphics[height=3cm]{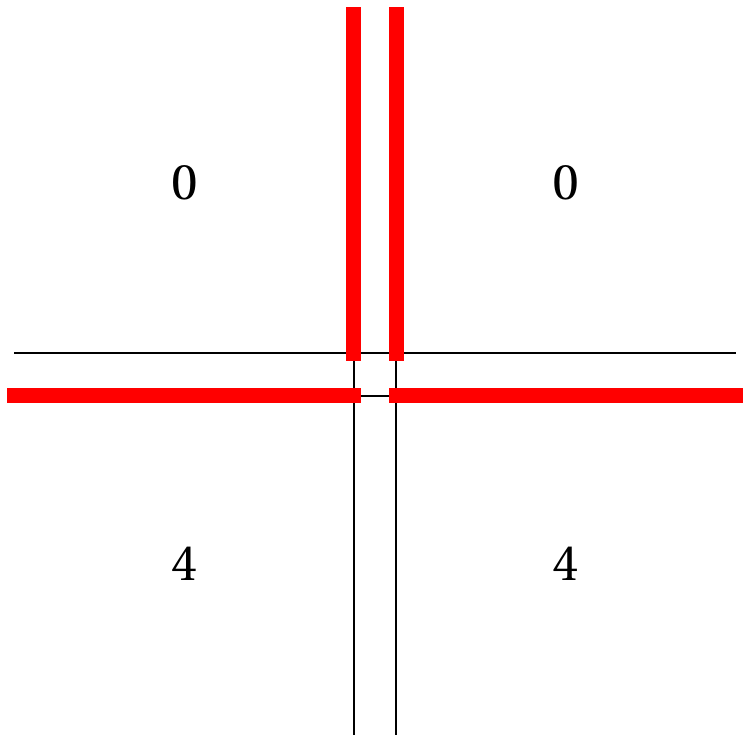}
\includegraphics[height=3cm]{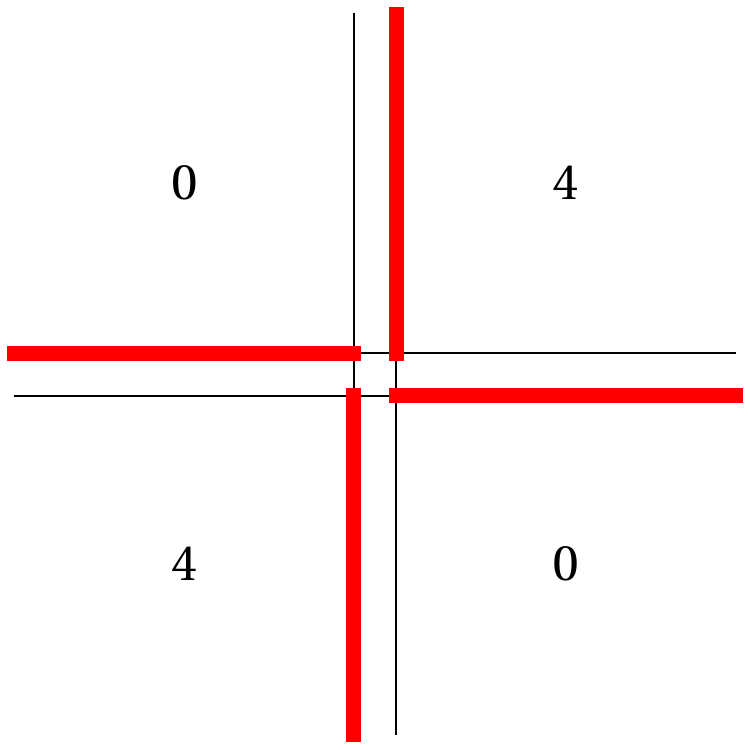}
\caption{All possible local configurations around each $b$-face (up to rotations and reflections) along with the possible $a$-height changes. The $a$-dimers are drawn in red while the $b$-dimers are drawn in black. For simplicity, we have suppresed the orientation.}
\label{fig:localconfigs}
\end{center}
\end{figure}


We define a \emph{loop} of length $k$, with $k\geq 4$, to be a sequence of distinct edges $(e_1,e_3,\dots,e_{2k-1})$  such that
\begin{enumerate}
\item $e_{2i+1}$ are $a$-edges and covered by dimers for all $0\leq i \leq k-1$, and none of these $a$-dimers are part of a double edge after the squishing procedure,
\item there are distinct $b$ edges $e_0,e_2,\dots,e_{2k}$ incident to distinct $b$-faces not covered by dimers such that $e_{2i}$ shares one endpoint with $e_{2i-1}$ and its other endpoint with $e_{2i+1}$ for all $0\leq i \leq 2k$ where $e_{-1}=e_{2k-1}$ and $e_{2k+1}=e_1$.
\end{enumerate}
It follows that after the squishing procedure, the sequence of edges in the loop is connected and visually forms a loop.  Each loop is in fact oriented thanks to the prescribed orientation.   We denote by $\ell(\gamma)$ to be the length of the loop $\gamma$.

	The above criterion of requiring distinct $b$-edges means that loops which appear to have one self-intersection after the squishing procedure, are in fact two separate loops.   However, there is an ambiguity in the definition when two (or more) loops intersect at more than one $b$-face; see Fig~\ref{fig:dichotomy}.  

To circumvent this ambiguity when two or more loops intersect or meet at more than one $b$-face, we introduce a \emph{mirror} at each vertex of $\tilde{D}_m$ where $\tilde{d}$ has four incident $a$-edges.  The mirror is a line between the centers of the $a$-faces of lowest $a$-height value, and on each side of the mirror, there is a different loop.  Once this choice is given, it is not hard to see that the loops are unique. From this convention, we call the vertices with mirrors \emph{meeting points} and say that two loops \emph{meet} at a vertex of $\tilde{D}_m$.

The height function definition means that there is an $a$-height change of $\pm4$ when traversing into or out of each loop with the sign depending on the orientation of the loop and that the $a$-height function along the inner boundary $a$-faces of the loop is constant.  From our conventions, stepping into a counterclockwise loop decreases the $a$-height function by 4 (a negative loop) while stepping into clockwise loop increases the $a$-height function by 4 (a positive loop) which leads to the following definition. 
\begin{defi}
Define $h^a_l(f)$ to be the contribution of the $a$-height function from \emph{only} the loops for the $a$-face $f$, that is $h^a_l(f)/4$  is given by the number of positive loops surrounding $f$ subtracted by the number of negative loops surrounding $f$.
\end{defi}
It follows that given a configuration of oriented loops, there is a well-defined $a$-height function on loops. From the above definition of mirrors, the converse is also true.

We define a \emph{path} of length $k$, with $k\geq 1$, to be a sequence of distinct edges $(e_1,e_3,\dots,e_{2k-1})$ such that
\begin{enumerate}
\item $e_{2i+1}$ are $a$-edges and covered by dimers for all $0\leq i \leq k-1$ and none of these $a$-dimers are part of a double edge after the squishing procedure,
\item there are distinct $b$-edges $e_2,\dots,e_{2k-2}$ not covered by dimers such that $e_{2i}$ shares an endpoint with $e_{2i-1}$ and $e_{2i+1}$ for all $1\leq i \leq k-1$,
\item $e_1$ and $e_{2k-1}$ are incident to the boundary face of $D_m$.
\end{enumerate}

As for loops, each path is in fact oriented thanks to the prescribed orientation.  Analogous to the ambiguity that is present for loops, paths are not well defined due to the possibility of intersecting multiple times with loops (or other paths). This ambiguity is removed using the mirrors, that is paths can meet with loops (or other paths) at meeting points and it is clear, by our convention, which sequence of $a$-dimers belongs to which object. 

\begin{figure}
\begin{center}
\includegraphics[height=5cm]{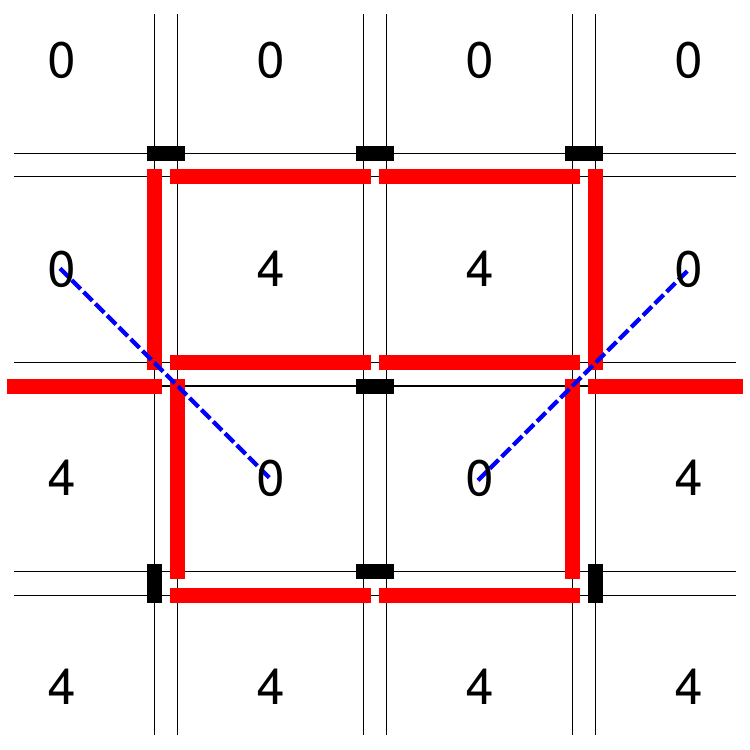}
	\caption{The $a$-dimers are drawn in red while the $b$-dimers are drawn in black, with the $a$-height function given at each $a$-face. The leftmost and rightmost $a$-dimers can be seen as part of a path or a loop. Between these leftmost and rightmost $a$-dimers, it is not clear whether the loop or path goes up or down, unless mirrors are used. The mirrors are drawn in blue (dashed). }
\label{fig:dichotomy}
\end{center}
\end{figure}

For each path, there is an $a$-height difference of $\pm 4$ for the $a$-faces on either side of the path, which depends on the orientation of the path which leads to the following lemma.  

\begin{lemma}
Each $a$-dimer on $D_m$ after squishing is either part of a double edge, a loop or a path. 

\end{lemma}
\begin{proof}

From the squishing procedure, there are either zero, two or four $a$-dimers incident to each $b$-face.  
If there are two incident $a$-dimers to a $b$-face, then there must be a dimer covering a $b$-edge on this face,  with the $a$-dimers forming either a double edge or the $a$-dimers sharing no common incident $a$-face of $D_m$. In the former case, a nearest neighboring $b$-face must have at least two incident $a$-dimers while in the latter case, there is an $a$-height difference which means the presence of a loop or a path. 
 If there are four incident $a$-dimers to a $b$-face, then there are three cases given by having
\begin{enumerate}
\item two adjacent double edges, 
\item  a double edge incident to a loop or a path, 
\item two loops or paths. 
\end{enumerate}
For each of these possibilities, see Fig.~\ref{fig:localconfigs}. The first case is immediate and notice that it is impossible for these double edges to have angles $\pm \pi/2$ from each other.    To see the second case, if there is a double edge incident to a $b$-face and the remaining two vertices of that $b$-face are not incident to another double edge, then the two remaining $a$-dimers are not incident to the same $a$-face and hence, a loop or a path is formed due to an $a$-height difference. When there are no double edges incident to a $b$-face, it follows that there is exactly one $a$-dimer in each direction protruding out of the $b$-face, which corresponds to a mirror. This means that there are $a$-height differences  giving two separate loops or paths.

\end{proof}

For the rest of this subsection, we suppose that we have applied the squishing procedure.
From the definition of the $a$-height function and the proof of the above lemma, we see that a path cannot meet itself (this type of self-intersection is a loop and a path), nor can it  meet another path, unless both paths separate the same $a$-height.  If two paths separate the same $a$-height, then these paths can meet at the meeting points (our convention using mirrors at the meeting points defines each path uniquely). 

Each boundary face on the top and bottom boundaries of $\tilde{D}_m$ induces an oriented path that terminates on either the left or right boundaries of $\tilde{D}_m$, with each path separating a different $a$-height. Since the paths on the bottom (resp. top) boundary separate different heights, it follows that the paths on the bottom (resp. top) boundary cannot meet at a vertex in $\tilde{D}_m$.   It is possible (combinatorially), that one path starting from the bottom boundary and one path starting from the top boundary meet at mulitple vertices in $\tilde{D}_m$, which only happens if they separate the same $a$-height. Note that due to the height increasing from left to right on the bottom boundary and decreasing from left to right on the top boundary, only one such pairing is combinatorially possible. 

For a dimer covering $d$ of $D_m$, label $\tilde{\Gamma}_{i,m}=\tilde{\Gamma}_{i,m} (\tilde{d})$ to be the oriented paths which separate the $a$-height $4i$ and $4i+4$ for $0 \leq i \leq 2m-1$. The trajectories of these paths naturally partition $\tilde{D}_m$ into sets of faces which we call \emph{corridors}, so that each face in the (squished) Aztec diamond belongs to a corridor, which is captured in the next definition. 
\begin{defi}
Let $\mathcal{C}_0=\mathcal{C}_0(\tilde{d})$ be all the $a$-faces in $\tilde{D}_m$ bounded between $\tilde{\Gamma}_{0,m}$ and the boundary of $\tilde{D}_m$, $\mathcal{C}_i=\mathcal{C}_i(\tilde{d})$ be all $a$-faces in $\tilde{D}_m$  bounded between $\tilde{\Gamma}_{i-1,m}$ and $\tilde{\Gamma}_{i,m}$ for $0<i\leq  2m-1$, and $\mathcal{C}_{2m}=\mathcal{C}_{2m}(\tilde{d})$ be all $a$-faces in $\tilde{D}_m$ bounded by $\tilde{\Gamma}_{2m-1,m}$ and the boundary of $\tilde{D}_m$.  

For an $a$-face $f \in \mathcal{C}_i$ and $0  \leq i\leq 2m$, we denote
	\begin{equation*}
		h^a_\mathtt{c}(f)=4 i,
	\end{equation*}
	which is called the \emph{corridor height} of the face $f \in \mathcal{C}_i$.  
\end{defi}
It follows from the above discussion that $a$-height function is the sum of the contribution from loops and the corridor height, that is, for
$f \in \tilde{D}_m $ we have
\begin{equation}\label{loopcorridor}
h^a(f)=h^a_\mathtt{c}(f)+h^a_l(f). 
\end{equation}

\section{Main Theorem}\label{sec:mainthm}

Before stating the main theorem, we introduce some notation.
We have the following constant from~\cite{CJ:16} 
\begin{equation} \label{eq:parameterc}
c=\frac{a}{(1+a^2) }.
\end{equation}
 Since we are interested in the rough-smooth boundary, we fix
 $\xi=-\frac{1}{2}\sqrt{1-2c}$ and set
\begin{equation} \label{eq:scalingparameters}
c_0=\frac{(1-2c)^{\frac{2}{3}}}{(2c(1+2c))^{\frac{1}{3}}}, \hspace{5mm} 
 \lambda_1 =  \frac{\sqrt{1-2c}}{2c_0} \hspace{5mm} \mbox{and} \hspace{5mm}
\lambda_2=\frac{(1-2c)^{\frac{3}{2}}}{2c c_0^2}.
\end{equation}
The term $\xi$ can be thought of as the asymptotic parameter which puts the analysis at the rough-smooth boundary after re-scaling (along the main diagonal in the third quadrant of the Aztec diamond). The terms $\lambda_1$ and $\lambda_2$ are scale parameters, as found in~\cite{CJ:16}.

For the rest of this paper, we introduce $M=M(m) \to \infty$ slowly as $m\to \infty$, but with $M^4 (\log m)^8/m^{1/3} \to 0$ as $m\to \infty$, for example, we could have $M=(\log m)^\gamma$ where $\gamma>0$.   Recall that $\beta_1< \dots < \beta_{L_1}$, $L_1 \geq 1$ are given fixed real numbers and $A_p=[\alpha_p^l,\alpha_p^r]$ for $\alpha_p^l < \alpha_p^r$ and $1 \leq p \leq L_2$ are finite disjoint intervals in $\mathbb{R}$.   We want to place scaled versions of the intervals $\{\beta_q \} \times A_p$  approximately at the rough-smooth boundary so that we get intervals between $a$-faces.  To be more precise, 
introduce
\begin{equation}
\beta_m(q,k)=2\lfloor\beta_q \lambda_2 (2m)^{2/3} + k \lambda_2 (\log m )^2 \rfloor, \hspace{5mm}\mbox{and} 
\end{equation}
\begin{equation}
\rho_m=4 \lfloor m(1+\xi)\rfloor, \hspace{5mm} \tau_m(q)=\lfloor\beta_q^2 \lambda_1 (2m)^{1/3}\rfloor.
\end{equation}
where $1 \leq k \leq M$. The additional parameter $k$ is for notational convenience later in the paper and is not needed (that is, set $k=1$) in order to state of the results of this paper.

 We also need the following notation for $a$-faces. Recall that $e_1=(1,1)$ and $e_2=(-1,1)$. 
Define the $a$-faces for $1 \leq p \leq L_2$, $1 \leq q \leq L_1$ and $1 \leq k_1,k_2 \leq M$
\begin{equation} \label{eq:facesJl}
J_{p,q,k_1,k_2}^l= (\rho_m+2 \lfloor \alpha_p^l \lambda_1 (2m)^{1/3}-\lambda_1 k_1 (\log m)^2 \rfloor -1 -2 \tau_m (q)) e_1 - \beta_m(q,k_2) e_2
\end{equation}
and
\begin{equation}\label{eq:facesJr}
J_{p,q,k_1,k_2}^r= (\rho_m+2 \lfloor \alpha_p^r \lambda_1 (2m)^{1/3} +\lambda_1 k_1 (\log m)^2 \rfloor +1 -2 \tau_m (q)) e_1 - \beta_m(q,k_2) e_2.
\end{equation}
Again, the additional parameters $k_1$ and $k_2$ are for convenience later in the paper. Let $\mathbb{P}_{\mathrm{Az}}$ denote the probability measure with respect to the two-periodic Aztec diamond and for $1\leq p \leq L_2$, $1 \leq q \leq L_1$, let 
\begin{equation} \label{eq:kappa}
	\kappa_m(\{\beta_q\} \times A_p) =\frac{1}{4} \left( h^a_\mathtt{c} (J^r_{p,q,1,1})-h^a_\mathtt{c} (J^l_{p,q,1,1}) \right).
\end{equation}
 We are now in a position to state precisely the main theorem. 
\begin{thma}\label{thm:main}
	Assume that $a < 1/3$ and that $\kappa_m$ and $\mu_{\mathrm{Ai}}$ are defined as above.  Then, the random variables $\kappa_m(\{\beta_q\}\times A_p)$, $1 \leq p \leq L_2, 1 \leq q \leq L_1$ converge jointly in distribution to the random variables $\mu_{\mathrm{Ai}}(\{\beta_q\}\times A_p)$ as $m \to \infty$. 
\end{thma}
Although $\kappa_m$ is in general a signed measure, we expect that with probability tending to 1, it is actually a positive measure.  

It is clear from the definition that the corridors are separated by paths. Hence, $\kappa_m$ in~\eqref{eq:kappa} counts the (signed) number of paths between two points at a distance of order $m^{1/3}$ at the rough-smooth boundary.  The theorem says that counting the number of paths this way defines a signed measure that converges to the extended Airy kernel point process.  We get a signed measure because the paths can backtrack.  However, we expect that the backtracks are small (like loops) and do not have any influence on the scales we are considering. 

The paths $\tilde{\Gamma}_{i,m}$ split into two parts $\tilde{\Gamma}_{i,m}^t$ and $\tilde{\Gamma}_{i,m}^b$ which start from the top and bottom boundaries respectively; see Fig.~\ref{fig:squishAztec12a}.  We expect that for all $a \in (0,1)$ there is an $i_0$, close to $m$, such that $\tilde{\Gamma}_{i_0,m}^b$ ends at the left boundary and $\tilde{\Gamma}_{i_0+1,m}^b$ ends at the right boundary.  Thus, we believe that the path $\tilde{\Gamma}_{i_0,m}^b$ is the last path in the third quadrant in the vicinity of the main diagonal, that is, between this path and the center of the Aztec diamond, there are no paths.  Moreover, we conjecture that for all $a \in (0,1)$ the path $\tilde{\Gamma}_{i_0,m}^b$ converges, after appropriate rescaling, to the Airy process.

The proof of the above theorem involves four main ingredients which are stated in Section~\ref{sec:Auxiliary}, with their proofs postponed until later in the article.   This allows us to give the proof of the main theorem in Section~\ref{sec:mainproof}.  In Section~\ref{sec:nuconverge} we give the proof of the first main ingredient which gives a refinement of the main result in~\cite{BCJ:16}.   In Section~\ref{section:couple}, we give the proof of the second main ingredient which gives couplings between configurations at the rough-smooth boundary with the smooth phase.  In Section~\ref{sec:geometry}, we give the proof of the third main ingredient which says that there are no (full-plane) paths in the smooth phase almost surely. 
In Section~\ref{sec:Peierls}, we give the proof of the final main ingredient, which gives control of the size of the loops provided that $a<1/3$.

\section{Auxiliary Results} \label{sec:Auxiliary}

Before we are in a position to prove Theorem~\ref{thm:main}, we require four ingredients  which are given in the following four subsections.  The proofs of these results are postponed to later in the paper.

\subsection{Multi-line to single line}\label{sec:singtomulti}

In this subsection, we give the asymptotics of the inverse Kasteleyn matrix at the rough-smooth boundary~\cite{CJ:16}, the definition of the random measure defined by taking (single) $a$-height differences on multiple lines used in~\cite{BCJ:16}, and state a result that this random measure is equivalent, at the rough-smooth boundary as $m\to \infty$, to the random measure defined by taking multiple $a$-height differences on a single line.

\subsubsection{The Kasteleyn matrix and its inverse} \label{App:Kast}

The Kasteleyn matrix for a finite planar bipartite graph is a type of signed weighted adjacency matrix whose rows are indexed by the black vertices of the graph and whose columns are indexed by the white vertices of the graph. More precisely, for a graph with white vertices $\tilde{\mathtt{W}}$ and black vertices $\tilde{\mathtt{B}}$ which admits dimer coverings, $K$ is a matrix with entries
\begin{equation}
	K_{\mathtt{b}\mathtt{w}} = \left\{\begin{array}{ll}
		0  & \mbox{if $(\mathtt{w},\mathtt{b})$ is not an edge in the graph}\\
		\mathrm{sgn}(e)w(e) & \mbox{if $e=(\mathtt{w},\mathtt{b})$ is an edge in the graph},
	\end{array} \right.
\end{equation}
where $\mathrm{sgn}(e)$ is chosen according to the \emph{Kasteleyn orientation}, a choice in signs ensures that the product of $K_{\mathtt{b}\mathtt{w}}$ for the edges around each face is negative, and $w(e)$ denotes the edge weight of $e$. 
For the significance of the Kasteleyn matrix for random tiling models, see for example~\cite{Ken:09}.  We will denote $K_{\mathtt{b}\mathtt{w}}=K(\mathtt{b},\mathtt{w})$ and stick to this convention throughout the paper.

For planar bipartite graphs, $G$, the dimers form a determinantal point process~\cite{Ken:97}.   More explicitly,  suppose that $E=\{\mathtt{e}_i\}_{i=1}^r$ is a collection of distinct edges with $\mathtt{e}_i=(\mathtt{b}_i,\mathtt{w}_i)$, where $\mathtt{b}_i$ and $\mathtt{w}_i$ denote black and white vertices. 
\begin{thma}[\cite{Ken:97,Joh17}]\label{localstatisticsthm}
    The dimers form a determinantal point process on the edges of $G$ with correlation kernel $L$ meaning that the probability of observing edges $e_1,\dots, e_r$ is given by 
	$\det L(\mathtt{e}_i,\mathtt{e}_j)_{1 \leq i,j \leq r}$
where
	    $L(\mathtt{e}_i,\mathtt{e}_j) = K({\mathtt{b}_i,\mathtt{w}_i}) K^{-1}({\mathtt{w}_j,\mathtt{b}_i}).$
\end{thma}
{ In what follows below, the graph $G$ will either denote a finite graph such as the Aztec diamond graph or the full-plane ($\mathbb{Z}^2$). }

The \emph{Kasteleyn matrix} for the two periodic Aztec diamond of size $n=4m$ with  parameters $a$ and $b$, denoted by $K_{a,b}$, is given by 
\begin{equation} \label{pf:K}
      K_{a,b}(x,y)=\left\{\begin{array}{ll}
		     a (1-j) + b j  & \mbox{if } y=x+e_1, x \in \mathtt{B}_j \\
		     (a j +b (1-j) ) \mathrm{i} & \mbox{if } y=x+e_2, x \in \mathtt{B}_j\\
		     a j + b (1-j)  & \mbox{if } y=x-e_1, x \in \mathtt{B}_j \\
		     (a (1-j) +b j ) \mathrm{i} & \mbox{if } y=x-e_2, x \in \mathtt{B}_j\\
			0 & \mbox{if $(x,y)$ is not an edge}
		     \end{array} \right.
\end{equation}
where $\mathrm{i}^2=-1$ and $j\in \{0,1\}$. A formula for the inverse Kasteleyn matrix for the two-periodic Aztec diamond was derived in~\cite{CY:13} and a simplification given in~\cite{CJ:16}.     Before giving the asymptotics of the inverse Kasteleyn matrix at the rough-smooth boundary, we give the full-plane smooth phase inverse Kasteleyn matrix (with the same edge weight conventions).   Denote $\mathbb{P}_{\mathrm{sm}}$ to be the probability measure in the full-plane smooth phase.  Define the white and black vertices on the plane by
{
\begin{equation} \label{white:parityplane}
	\mathtt{W}_i^*= \{(x,y) \in \mathbb{Z}^2: x \mod 2=1,y\mod=0, x+y \mod4=2i+1\} 
\end{equation}
and
\begin{equation}\label{black:parityplane}
	\mathtt{B}_i^*= \{(x,y) \in \mathbb{Z}^2: x \mod 2=0,y\mod=1,  x+y \mod4=2i+1\} 
\end{equation}
where $  i \in\{0,1\}$. Recall that $b=1$. }
For $j \in \{0,1\}$ and $w \in \mathtt{W}_j^*$, the weight of the edge $(w, w+(-1)^k e_i)$ is given by $a^{(1-k)(1-j)+kj}$ for $k \in \{0,1\}$, $i\in \{1,2\}$, which is the same convention as the two-periodic Aztec diamond.  Let
\begin{equation} \label{ctilde}
\tilde{c} (u_1,u_2)=2(1+a^2) + a(u_1+u_1^{-1})(u_2+u_2^{-1}),
\end{equation} 
which is related to the so-called \emph{characteristic polynomial} for the dimer model~\cite{KOS:06}; see~\cite[(4.11)]{CJ:16} for an explanation. Write
\begin{equation}
h(\eps_1,\eps_2)=\eps_1(1-\eps_2)+\eps_2(1-\eps_1),
\end{equation}
and for the rest of this paper, $\Gamma_R$ denotes a positively oriented circle of radius $R$ around the origin.  
The full-plane smooth phase inverse Kasteleyn matrix is given by
\begin{equation}\label{smoothphaseeqn}
\mathbb{K}^{-1}_{1,1}(x,y)=-\frac{\mathrm{i}^{1+h(\eps_x,\eps_y)}}{(2 \pi \mathrm{i})^2} 
\int_{\Gamma_1} \frac{du_1}{u_1}  \int_{\Gamma_1} \frac{du_2}{u_2} \frac{a^{\eps_y} u_2^{1-h(\eps_x,\eps_y)} +a^{1-\eps_y} u_1 u_2^{h(\eps_x,\eps_y)}}{\tilde{c}(u_1,u_2) u_1^{\frac{x_1-y_1+1}{2}} u_2^{\frac{x_2-y_2+1}{2}}},
\end{equation}
where 
$x=(x_1,x_2)\in \mathtt{W}_{\eps_x}^*$ and $y=(y_1,y_2)\in \mathtt{B}_{\eps_y}^*$ with $\eps_x,\eps_y \in \{0,1\}$; see~\cite[Section 4]{CJ:16} for details and connections with~\cite{KOS:06}.  Note that in the above formula, we can replace $x\in \mathtt{W}_{\eps_x}^*$ and $y\in \mathtt{B}_{\eps_y}^*$ by $x\in \mathtt{W}_{\eps_x}$ and $y\in \mathtt{B}_{\eps_y}$ for $\eps_x,\eps_y \in \{0,1\}$ and this gives the same formula.

We can now give our formulas for the asymptotics of the inverse Kasteleyn matrix at the rough-smooth boundary.  
We set 
\begin{equation} \label{G}
\mathcal{C}=\frac{1}{\sqrt{2c}} (1-\sqrt{1-2c} ).
\end{equation}


 From~\cite{CJ:16}, it is natural to write
\begin{equation}\label{eq:Ksplit}
K^{-1}_{a,1} (x,y)=\mathbb{K}^{-1}_{1,1}(x,y)-\mathbb{K}_{\mathrm{A}}(x,y)
\end{equation}
which  defines  $\mathbb{K}_{\mathrm{A}}$.  The full expression for $\mathbb{K}_{\mathrm{A}}$ is complicated and will not be given in full here; see~\cite[Theorem 2.3]{CJ:16} and~\cite[Proposition 6.2]{DK:17}.

Let $\alpha_x,\alpha_y,\beta_x,\beta_y \in \mathbb{R}$, $k_x^1,k_x^2,k_y^1,k_y^2 \in \mathbb{Z}$ and $f_x,f_y \in \mathbb{Z}^2$. We will use the following scaling of $x$ and $y$ at the rough-smooth boundary  
\begin{equation}	\begin{split} 
	x&=(\rho_m +2\lfloor \alpha_x \lambda_1 (2m)^{1/3}+k_x^1 \lambda_1 (\log m)^2 \rfloor)e_1\\& -(2\lfloor\beta_x \lambda_2 (2m)^{2/3}+k_x^2 \lambda_2 (\log m)^2 \rfloor)e_2 +f_x\\
		y&=(\rho_m +2\lfloor\alpha_y \lambda_1 (2m)^{1/3}+k_y^1 \lambda_1 (\log m)^2 \rfloor)e_1 \\&-(2\lfloor[\beta_y \lambda_2 (2m)^{2/3}+k_y^2 \lambda_2 (\log m)^2\rfloor)e_2 +f_y.
\end{split}\label{xyscaling}
\end{equation}

We introduce the notation 
\begin{equation}
\mathtt{g}_{\eps_1,\eps_2} =
\left\{
\begin{array}{ll}
\frac{\mathrm{i} \left(\sqrt{a^2+1}+a\right)}{1-a}
 &  \mbox{if } (\eps_1,\eps_2)=(0,0) \\
\frac{\sqrt{a^2+1}+a-1}{\sqrt{2a} (1-a) }
 & \mbox{if } (\eps_1,\eps_2)=(0,1) \\
-\frac{\sqrt{a^2+1}+a-1}{\sqrt{2a} (1-a) }
&\mbox{if } (\eps_1,\eps_2)=(1,0)\\
\frac{\mathrm{i}\left(\sqrt{a^2+1}-1\right)}{(1-a) a}
&\mbox{if } (\eps_1,\eps_2)=(1,1).
\end{array}
\right.
\end{equation}
From~\cite[Theorem 2.7]{CJ:16} and its proof, we have
\begin{thma}[\cite{CJ:16}]
\label{Airyasymptotics}
Assume that $x=(x_1,x_2) \in \mathtt{W}_{\eps_x}$ and $y=(y_1,y_2) \in \mathtt{B}_{\eps_y}$ are given by~\eqref{xyscaling} with $\eps_x,\eps_y \in\{0,1\}$.  Furthermore, assume that $|\alpha_x|$, $|\alpha_y|$, $|\beta_x|$, $|\beta_y|$, $|f_x|$, $|f_y|\leq C$ for some constant $C>0$ and that $|k_x^1|,|k_x^2|,|k_y^1|,|k_y^2|\leq M$ with $M$ as above.  Then, as $m \to \infty$
\begin{equation}
\begin{split}\label{As1}
\mathbb{K}_{\mathrm{A}}(x,y)&= \mathrm{i}^{y_1-x_1+1} \mathcal{C}^{\frac{-2-x_1+x_2+y_1-y_2}{2}} c_0 \mathtt{g}_{\eps_x,\eps_y} e^{\alpha_y \beta_y -  \alpha_x \beta_x -\frac{2}{3} (\beta_x^3-\beta_y^3)}
\\ &\times 
(2m)^{-\frac{1}{3}} (\tilde{\mathcal{A} }(\beta_x , \alpha_x+\beta_x^2; \beta_y,\alpha_y+\beta_y^2)+o(1)).
\end{split}
\end{equation}
Also, as $m\to \infty$,
\begin{equation}
\begin{split}\label{As2}
\mathbb{K}_{1,1}^{-1}(x,y)&= \mathrm{i}^{y_1-x_1+1} \mathcal{C}^{\frac{-2-x_1+x_2+y_1-y_2}{2}} c_0 \mathtt{g}_{\eps_x,\eps_y} e^{\alpha_y \beta_y -  \alpha_x \beta_x -\frac{2}{3} (\beta_x^3-\beta_y^3)}\\ &\times (2m)^{-\frac{1}{3}}  ( \phi_{\beta_x,\beta_y} ( \alpha_x+\beta_x^2 ;\alpha_y+\beta_y^2) +o(1)).
\end{split}
\end{equation}

\end{thma}
The formulation given for the above theorem is slightly different from that in~\cite{CJ:16}, however, this modification makes no difference.

\subsubsection{Definition of random measures } 
We give the definition of $\mu_m$ which is equivalent to the definition in~\cite{BCJ:16} but in a simpler form as well as random measure that will be used in the proof of the main theorem.  The reason for the simplification for $\mu_m$ is that in~\cite{BCJ:16},  we stated the formulas in terms of particles to coincide with determinantal point processes. 
Define
\begin{equation}
\begin{split}
\mu_m(\{ \beta_q \} \times A_p) &= \frac{1}{4 M} \sum_{k=1}^{M} h(J_{p,q,1,k}^r) - h(J_{p,q,1,k}^l).
\end{split}
\end{equation}
The random signed measure $\mu_m$ can be thought of as a `horizontal averaging' of the height function.  The following theorem from~\cite{BCJ:16} holds for our choice of $M$ in this paper, it is easy to see that our choice of $M$ in this paper is a restriction of the one given in~\cite{BCJ:16}. 
\begin{thma}[Theorem 1.1 in~\cite{BCJ:16}] \label{thm:previousBCJ}
	As $m \to \infty$ and for all $a \in (0,1)$, $\mu_m$ converges to the extended Airy kernel point process in the sense that  there exists $R>0$ such that
\begin{equation}
	\begin{split}
		&\lim_{m \to \infty}\mathbb{E}\bigg[\exp \bigg( {\sum_{p=1}^{L_2} \sum_{q=1}^{L_1} w_{p,q} \mu_m (\{ \beta_q \} \times A_p)}\bigg)\bigg]\\ &= \mathbb{E}\bigg[\exp \bigg( {\sum_{p=1}^{L_2} \sum_{q=1}^{L_1} w_{p,q} \mu_{\mathrm{Ai}} (\{ \beta_q \} \times A_p)} \bigg) \bigg]
	\end{split}
\end{equation}
with $w_{p,q} \in \mathbb{C}$ such that  $|w_{p,q}| <R $ for all  $1 \leq p \leq L_2$, $1 \leq q \leq L_1$. 
\end{thma}
   Roughly speaking, the above theorem says that  the horizontal averaging of the height function converges (in the above sense) to the Airy kernel point process.

Introduce the random signed measure
\begin{equation}\label{def:nu_m}
\begin{split}
\nu_m(\{ \beta_q \} \times A_p) &= \frac{1}{4 M} \sum_{k=1}^{M} h^a(J_{p,q,k,1}^r) - h^a(J_{p,q,k,1}^l).
\end{split}
\end{equation}
The measure $\nu_m$ can be thought of as a `vertical averaging' of the height function.   
\begin{prop} \label{prop:nu_mconverge}
As $m \to \infty$ and for all $a \in (0,1)$, $\nu_m$ converges to the extended Airy kernel point process, where convergence  is in the same sense  as that given in Theorem~\ref{thm:previousBCJ}.

\end{prop}
It follows from the proposition that the random variables $\nu_m(\{ \beta_q\} \times A_p)$ converge jointly in law to the random variables $\mu_{\mathrm{Ai}} (\{\beta_q\} \times A_p)$, $1 \leq p \leq L_2$, $1 \leq q \leq L_1$ as $m\to \infty$.  The proof is given in Section~\ref{sec:nuconverge}.

\subsection{Smooth Couplings} \label{subsec:smoothcouple}

Here, we state a result for coupling the dimer configurations at the rough-smooth boundary to the restriction of the full-plane smooth phase to a single box and also show that two distant configurations in the full-plane smooth phase are almost independent. 

Let $(u,v)$ be an $a$-face and let $\Lambda_L^{(u,v)}=\Lambda ((u,v),L)$ be a box with corners $(u+L-1,v)$, $(u-L+1,v)$,  $(u,v+L-1)$ and $(u,v-L+1)$ for $L \in 2\mathbb{Z}$, chosen so that the box is inside the Aztec diamond. 
 Let $\partial \Lambda_L^{(u,v)}$ denote vertices which share edges that cross boundary of the box $\Lambda_L^{(u,v)}$. Let $\mathtt{W}^{\Lambda_L^{(u,v)}}$ denote all white vertices in $\Lambda_L^{(u,v)} \cup \partial \Lambda_L^{(u,v)}$, then we can write $\mathtt{W}^{\Lambda_L^{(u,v)}}=\{{w}_1,\dots, {w}_R\}$ where $R=L^2/2+2L$. Set $f_1=e_1, f_2=e_2, f_3=-e_1$, and $f_4=-e_2$ and $[N]=\{1,\dots, N\}$ for a positive integer $N$. A \emph{configuration} in $\Lambda_L^{(u,v)}$ is a  set of edges:
\begin{equation} \label{gascouple:edges}
	({w}_1,{w}_1+f_{s_1}),\dots, ({w}_R,{w}_R+f_{s_R})
\end{equation}
where $s_j \in[4], 1 \leq j \leq R$.  We think of~\eqref{gascouple:edges} as the event that all these edges are covered by dimers. For $\overline{s} \in [4]^R$, we let $\overline{s}$ denote the configuration~\eqref{gascouple:edges}. Note that for certain choices of $\overline{s}$, two edges will meet at a black vertex, but these configurations will have probability zero. We have two probability measures on $\Omega=[4]^R$ coming from $\mathbb{P}_{\mathrm{Az}}$ and $\mathbb{P}_{\mathrm{sm}}$. 

Write 
\begin{equation}
	A_{ij}(\overline{s})=K_{a,1}({w}_i,{w}_i)\mathbb{K}^{-1}_{1,1}({w}_i,{w}_j+f_{s_j})
\end{equation}
and 
\begin{equation}
	C_{ij}(\overline{s})=-K_{a,1}({w}_i,{w}_i)\mathbb{K}_{\mathrm{A}}({w}_i,{w}_j+f_{s_j});
\end{equation}
see Section~\ref{App:Kast}. It follows that from Theorem~\ref{Airyasymptotics} and~\cite[Theorem 2.6]{CJ:16} that there is a constant $C_0$ such that 
\begin{equation} \label{constC0}
	|C_{ij}(\overline{s})| \leq C_0 
\end{equation}
for all $i,j$ and for all $\overline{s}$ chosen so that $(u,v)=x$, where $x$ is of the form given by~\eqref{xyscaling}. 
Then, $\mathbb{P}_{\mathrm{Az}}$ induces the following measure on $\Omega$:
\begin{equation}
	p_{\mathrm{Az}}\big(\overline{s}|\Lambda_L^{(u,v)}\big)=\det\big( A_{ij}(\overline{s})+ C_{ij}(\overline{s} )\big)_{1\leq i,j\leq R} \label{gascouple:pAz} 
\end{equation}
and $\mathbb{P}_{\mathrm{sm}}$ gives
\begin{equation}	
	p_{\mathrm{sm}}\big(\overline{s}|\Lambda_L^{(u,v)}\big)=\det( A_{ij}(\overline{s}))_{1\leq i,j\leq R}.\label{gascouple:S}
\end{equation}
Note that if $w_j+f_{s_j}=w_k+f_{s_k}$ for $j\not = k$, then~\eqref{gascouple:pAz} and~\eqref{gascouple:S}  both give zero, so configurations with overlaps have probability zero.

\begin{prop}\label{prop:couple}
	Let $C_0$ be the constant in~\eqref{constC0}. Then, we have the following estimate of the total variation distance
	\begin{equation}
		d_{TV}(p_{\mathrm{Az}},p_{\mathrm{sm}}) = \frac{1}{2} \sum_{\overline{s} \in \Omega} \big|p_{\mathrm{Az}}\big(\overline{s}|\Lambda_L^{(u,v)}\big) - p_{\mathrm{sm}}\big(\overline{s}|\Lambda_L^{(u,v)}\big)\big| \leq \frac{ eC_0(2L+L^2)^2}{m^{1/3}},
	\end{equation}
	where $	p_{\mathrm{Az}}=p_{\mathrm{Az}}\big(\cdot|\Lambda_L^{(u,v)}\big)$, $p_{\mathrm{sm}}=p_{\mathrm{sm}}\big(\cdot|\Lambda_L^{(u,v)}\big)$,
	provided $L$ satisifes $ C_0(2L+L^2)^2 \leq m^{1/3}$ and  $(u,v)=x$ is of the form given in~\eqref{xyscaling}.
\end{prop}
The proof is given in Section~\ref{section:couple}.

In order to show that two distant configurations in the smooth phase are almost independent, we need to augment the notation given above.  Consider $\Lambda^1 = \Lambda(J^r_{p,q,k_1,1},L)$ and $\Lambda^2 = \Lambda(J^r_{p,q,k_2,1},L)$ for $k_1 \not = k_2$ with $1 \leq k_1 ,k_2 \leq M$, $1 \leq p \leq L_2$  and $1 \leq q \leq L_1$, {where $L<\lambda_1 (\log m)^2/\sqrt{2}$. The condition on $L$ ensures that $\Lambda^1$ and $\Lambda^2$ are disjoint.} We extend the above conventions for $\Lambda^1 \cup \Lambda^2$. That is, for $1 \leq i \leq 2$,
\begin{itemize}
	\item $\partial \Lambda^i$ denotes vertices which share edges that cross boundary of the box $\Lambda^i$ 
	\item $\mathtt{W}^{\Lambda^i}=\{{w}_{1+(i-1)R},\dots, {w}_{R+(i-1)R}\} $ denotes all white vertices in $\Lambda^i \cup \partial \Lambda^i$.
\end{itemize}
 A \emph{configuration} in $\Lambda^1 \cup \Lambda^2$ is a  set of edges:
\begin{equation} \label{gascouple:edgesgas}
	({w}_1,{w}_1+f_{s_1}),\dots, ({w}_{2R},{w}_{2R}+f_{s_{2R}})
\end{equation}
where $s_j \in[4], 1 \leq j \leq 2R$ and $R=\frac{L^2}{2}+2L$ as before.   We consider the smooth phase on the set of configurations $\Omega^1\cup\Omega^2=[4]^{2R}$, where each $\Omega^i$ is responsible for the configuration in $\Lambda^i$ for $1 \leq i \leq 2$.  
Write for $1 \leq i , j \leq 2R$
\begin{equation}
	D_{ij}(\overline{s})=K_{a,1}({w}_i,{w}_i)\mathbb{K}^{-1}_{1,1}({w}_i,{w}_j+f_{s_j}),
\end{equation}
\begin{equation}
	E_{ij}(\overline{s})=\left\{ \begin{array}{ll} 
		K_{a,1}({w}_i,{w}_i)\mathbb{K}^{-1}_{1,1}({w}_i,{w}_j+f_{s_j}) & \mbox{if }1 \leq i,j \leq R \\
		K_{a,1}({w}_i,{w}_i)\mathbb{K}^{-1}_{1,1}({w}_i,{w}_j+f_{s_j}) &\mbox{if }R+1 \leq i,j \leq 2R \\ 
		0 & \mbox{otherwise}
	\end{array} \right.
\end{equation}
and $F_{ij}(\overline{s})=D_{ij}(\overline{s})-E_{ij}(\overline{s}).$ As above, $\mathbb{P}_{\mathrm{sm}}$ induces the probability measure on $\Omega^1 \cup \Omega^2$
\begin{equation}
	p_{\mathrm{sm}}\big(\overline{s}|\Lambda^1 \cup \Lambda^2\big) =\det(D_{ij}(\overline{s}))_{1 \leq i,j \leq 2R}
\end{equation}
and the marginals on each $\Lambda^i$ 
$$
p_{\mathrm{sm}}\big(\overline{s}|\Lambda^k \big)=\det(E_{ij}(\overline{s}))_{1+(k-1)R \leq i,j \leq Rk}\hspace{5mm}  \mbox{for }1 \leq k \leq 2.$$

\begin{prop}\label{prop:couple2}
	There exists constants $C,c_0>0$ such that 
	\begin{equation}
		  \sum_{\overline{s} \in \Omega} \big|p_{\mathrm{sm}}\big(\overline{s}|\Lambda^1 \cup \Lambda^2\big) - p_{\mathrm{sm}}\big(\overline{s}|\Lambda^1 \big)p_{\mathrm{sm}}\big(\overline{s}|\Lambda^2 \big)\big| \leq CR^2 e^{-c_0(\log m)^2},
	\end{equation}
provided that $$R<\frac{\lambda_1^2(\log m)^4}{4}+\lambda_1(\log m)^2.$$  \end{prop}
The proof is given in Section~\ref{section:couple}.

\subsection{Biinfinite paths in the full-plane smooth phase}

In Section~\ref{subsec:squish}, we defined corridors, paths and loops for a dimer covering on ${D}_m$. 
Analogous to the Aztec diamond, there is a height function for the dimer model on the plane defined through height differences between faces, with the same convention given for the Aztec diamond. This height function on the plane  is unique up to height level.  There is also a corresponding notion of $a$-height function, $h^a$. Both the prescribed orientation and the squishing procedure generalize to the full plane by assigning an arrow to each edge of the plane from its white vertex to its black vertex, and by contracting the size of each $b$-face while simultaneously increasing the size of each $a$-face.  A \emph{biinfinite path} is a biinfinite sequence of distinct edges $\{e_{2k+1}, k\in \mathbb{Z}\}$ such that 
\begin{enumerate}
\item the sequence of $a$-edges $\{e_{2k+1}\}_{k \in\mathbb{Z}}$ are all by  covered by $a$-dimers and none of these $a$-edges are part of a double edge after the squishing procedure, and  
\item there exists a sequence $\{e_{2k}\}_{k \in \mathbb{Z}}$ of distinct $b$-edges not covered by dimers such that the edge $e_{2k}$ shares an endpoint with the edges $e_{2k-1}$ and $e_{2k}$ for all $k \in \mathbb{Z}$.  
\end{enumerate}
Similar to the construction for  $\tilde{D}_m$, we introduce mirrors to vertices of the plane (after squishing) where the dimer covering has four incident $a$-dimers, that is, a mirror is a line between the centers of $a$-faces of lowest $a$-height value around each vertex which has four incident $a$-dimers after the squishing procedure.  Analogous to the case of $\tilde{D}_m$, equipped with mirrors,  biinfinite paths are  well-defined. Following the arguments in Section~\ref{subsec:squish}, each $a$-dimer in the full-plane smooth phase is either part of a double edge, an oriented loop or a biinfinite path.

\begin{thma}\label{thm:onecorridor}
For $a \in [0,1)$, in the full-plane smooth phase there are no biinfinite paths in the full-plane smooth phase almost surely.
\end{thma}
The proof of this theorem is given in Section~\ref{sec:geometry}.

\subsection{Control of loops}\label{sec:Auxiliaryingred1} 
For a dimer covering on $D_m$ or the full-plane, let $\mathcal{D}_l$ to be the set of loops in the covering.  Let $S$ be a set of $a$-edges.  We say that a loop $\gamma$ in $\mathcal{D}_l$ intersects $S$ if $\gamma$ has an $a$-dimer that \emph{covers} an edge in $S$.   Recall that $\ell(\gamma)$ denotes the number of $a$-dimers in a loop, that is, the length of a loop. 
\begin{lemma} \label{lem:loopsalmostsure}
	Let $S$ be a set of $a$-edges in $D_m$ or the full-plane, and assume that $a \in (0,1/3)$. Then,
\begin{equation}
	\mathbb{P}[ \exists \gamma \in \mathcal{D}_l\mbox{  that intersects $S$ and has length $\ell(\gamma)$ at least  $d$}	] \leq \frac{|S|}{1-3a}(3a)^d
	\end{equation}
	where $|S|$ is the size of $S$, and $\mathbb{P}$ is either $\mathbb{P}_{\mathrm{Az}}$ or $\mathbb{P}_{\mathrm{sm}}$.
\end{lemma}
We also prove a similar result for double edges, which holds for all $a \in[0,1)$, but this is not needed for the proof of our main result.  The proof is given in Section~\ref{sec:Peierls}.

\section{Proof of Theorem~\ref{thm:main}} \label{sec:mainproof}

Before giving the proof of Theorem~\ref{thm:main}, we first state and prove two lemmas.  We recall the notation from Section~\ref{subsec:smoothcouple} that for an $a$-face $(u,v)$, we use $\Lambda ((u,v),L)$ to denote a box with corners $(u+L-1,v)$, $(u-L+1,v)$,  $(u,v+L-1)$ and $(u,v-L+1)$ for $L \in 2\mathbb{Z}$. Throughout this section, we fix
\begin{equation}\label{eq:L}
L=2\lfloor  \lambda_1M (\log m)^2 \rfloor
\end{equation}
and for $1 \leq p \leq L_2, 1 \leq q \leq L_1$ we let $\Lambda_L^{p,q,r}=\Lambda(J_{p,q,\lfloor \frac{M}{2}\rfloor,1}^r, L  )$ so that $J_{p,q,k,1}^r \in \Lambda_L^{p,q,r}$ for all $1 \leq k \leq M$; similarly for $\Lambda_L^{p,q,l}$.  Note that the choice of $L$ above satisfies the condition in Proposition~\ref{prop:couple}.

\begin{lemma}\label{lem:loopssmall}
	Under $\mathbb{P}_{\mathrm{Az}}$ and for $a \in (0,1/3)$,  $$\frac{1}{4 M} \sum_{k=1}^{M}h^a_l(J_{p,q,k,1}^r) - h^a_l(J_{p,q,k,1}^l) \to 0$$ in probability as $m \to \infty$. 
\end{lemma}
A similar computation shows that $\frac{1}{4 M} \sum_{k=1}^{M}h^a_l(J_{p,q,1,k}^r) - h^a_l(J_{p,q,1,k}^l) \to 0$ in probability as $m \to \infty$.

\begin{proof}
Let $D_m^a$ denote the set of all $a$-edges in the Aztec diamond and let $\delta>0$ be given.  Take $S= D_m^a$ and let $d= \delta (\log m)^2$ in Lemma~\ref{lem:loopsalmostsure}.  Write
	\begin{equation}
		A= \{ \mbox{all loops in the Aztec diamond that have length} \leq  d \}.
	\end{equation}
	Lemma~\ref{lem:loopsalmostsure} gives 
	\begin{equation}
		\mathbb{P}_{\mathrm{Az}} [A^c] \leq \frac{|D_m^a|}{1-3 a} (3a)^{\delta (\log m)^2} \leq Cm^2 (3a)^{\delta (\log m)^2}=o(1)
	\end{equation}
	as $m\to \infty$.  Define 
	\begin{equation}
		h_l^{a,d}(f)=\mbox{the loop height given by loops of length less than or equal to } d.
	\end{equation}
	If $d_m =\delta(\log m)^{2+\epsilon}$ for any $\epsilon>0$, then in the Aztec diamond $h_l^a(f)= h_l^{a,d_m}(f)$ in the set $A$.  Thus, we have
	\begin{equation}
		\begin{split}
			&\mathbb{P}_{\mathrm{Az}}\bigg[ \frac{1}{4M} \sum_{k=1}^M h_l^a (J_{p,q,k,1}^r) > \varepsilon\bigg]\\
			&\leq \mathbb{P}_{\mathrm{Az}}\bigg[\bigg\{ \frac{1}{4M} \sum_{k=1}^M h_l^{a,d_m} (J_{p,q,k,1}^r) > \varepsilon \bigg\} \cap A\bigg] + \mathbb{P}_{\mathrm{Az}}[A^c] \\ 
			&\leq \mathbb{P}_{\mathrm{Az}}\bigg[ \frac{1}{4M} \sum_{k=1}^M h_l^{a,d_m} (J_{p,q,k,1}^r) > \varepsilon \bigg]+o(1).
		\end{split}
	\end{equation}
	By the choice of $L$ in~\eqref{eq:L}, the random variable $\frac{1}{4M} \sum_{k=1}^M h_l^{a,d_m} (J_{p,q,k,1}^r)$ only depends on dimer configurations in $\Lambda_L^{p,q,r}$.  Hence, by Proposition~\ref{prop:couple} 
	\begin{equation} \label{eq:lemproof:intermediatebound}
		\mathbb{P}_{\mathrm{Az}}\left[ \frac{1}{4M} \sum_{k=1}^M h_l^a (J_{p,q,k,1}^r) > \varepsilon\right] \leq \mathbb{P}_{\mathrm{sm}}\left[ \frac{1}{4M} \sum_{k=1}^M h_l^{a,d_m} (J_{p,q,k,1}^r) > \varepsilon\right]+o(1).
	\end{equation}
		Write	
\begin{equation}
	B= \{ \mbox{all loops in $\Lambda^{p,q,r}_L$ that have length} \leq  \frac{(2+\delta) \log L}{\log (1/3a)} \}.
\end{equation}
Then, we have 
\begin{equation}
	\mathbb{P}_{\mathrm{sm}} [B^c] \leq \frac{|\Lambda^{p,q,r}_L| }{1-3 a} (3a)^{-\frac{(2+\delta) \log L}{\log 3a} } = \frac{4L^2 }{1-3 a} L^{-2-\delta}=o(1) 
\end{equation}
	as $m\to \infty$.  If $\tilde{d}_m =  -\frac{3 \log L}{\log 3a} $, we have that $h_l^{a,d}(f)= h_l^{a,\tilde{d}_m}(f)$ in the set $B$.  We have that 
\begin{equation}
		\begin{split}
			&\mathbb{P}_{\mathrm{sm}}\bigg[ \frac{1}{4M} \sum_{k=1}^M h_l^{a,d_m} (J_{p,q,k,1}^r) > \varepsilon\bigg]\\
			&\leq \mathbb{P}_{\mathrm{sm}}\bigg[\bigg\{ \frac{1}{4M} \sum_{k=1}^M h_l^{a,\tilde{d}_m} (J_{p,q,k,1}^r) > \varepsilon \bigg\} \cap B\bigg] + \mathbb{P}_{\mathrm{sm}}[B^c] \\ 
			&\leq \mathbb{P}_{\mathrm{sm}}\bigg[ \frac{1}{4M} \sum_{k=1}^M h_l^{a,\tilde{d}_m} (J_{p,q,k,1}^r) > \varepsilon \bigg]+o(1).
		\end{split}
	\end{equation}
Then, we have reduced~\eqref{eq:lemproof:intermediatebound} to 
\begin{equation} \label{eq:lemproof:intermediatebound2}
	\mathbb{P}_{\mathrm{Az}}\left[ \frac{1}{4M} \sum_{k=1}^M h_l^a (J_{p,q,k,1}^r) > \varepsilon\right] \leq \mathbb{P}_{\mathrm{sm}}\left[ \frac{1}{4M} \sum_{k=1}^M h_l^{a,\tilde{d}_m} (J_{p,q,k,1}^r) > \varepsilon\right]+o(1).
	\end{equation}
We focus on the right side of~\eqref{eq:lemproof:intermediatebound2}. {  We have that $\mathbb{E}_{\mathrm{sm}}\left[   h_l^{a} (J_{p,q,k,1}^r) \right]=\mathbb{E}_{\mathrm{sm}}\left[   h_l^{a,\tilde{d}_m} (J_{p,q,k,1}^r) \right] =0$, which follows immediately since the distribution of the loops is symmetric (that is, the probability of a configuration of loops is invariant under flipping the sign of all the loops due to the form of the correlation kernel of the full-plane smooth phase in~\eqref{smoothphaseeqn} and that there are only loops and double edges in the full-plane smooth phase almost surely from Theorem~\ref{thm:onecorridor}). By using Chebychev's inequality, we have}
\begin{equation}
\label{eq:chebychev}
	\mathbb{P}_{\mathrm{sm}}\left[ \frac{1}{4M} \sum_{k=1}^M h_l^{a,\tilde{d}_m} (J_{p,q,k,1}^r) > \varepsilon\right] \leq \frac{1}{(4M\varepsilon)^2}\mathrm{Var}_{\mathrm{sm}}\left[\sum_{k=1}^M h_l^{a,\tilde{d}_m} (J_{p,q,k,1}^r) \right]
\end{equation}
and the right side expands to
	\begin{equation}
		\begin{split}
			&\mathrm{Var}_{\mathrm{sm}}\left[\sum_{k=1}^M h_l^{a,\tilde{d}_m} (J_{p,q,k,1}^r) \right] = \sum_{k=1}^M \mathrm{Var}_{\mathrm{sm}}\left[h_l^{a,\tilde{d}_m} (J_{p,q,k,1}^r) \right] \\& + 2\sum_{1\leq k_1<k_2 \leq M}  \mathrm{Cov}_{\mathrm{sm}}\left[ h_l^{a,\tilde{d}_m} (J_{p,q,k_1,1}^r)h_l^{a,\tilde{d}_m} (J_{p,q,k_2,1}^r) \right].
		\end{split}
	\end{equation}
	We first show that 
	\begin{equation}\label{eq:covbound}
		\mathrm{Cov}_{\mathrm{sm}}\left[ h_l^{a,\tilde{d}_m} (J_{p,q,k_1,1}^r)h_l^{a,\tilde{d}_m} (J_{p,q,k_2,1}^r) \right]  =o(1)
	\end{equation}
		as $m\to \infty$. Each random variable $h_l^{a,\tilde{d}_m} (J_{p,q,k,1}^r)$ only depends on at most $\tilde{d}_m$ loops, since the loops are bounded by $\tilde{d}_m$, which means that $h_l^{a,\tilde{d}_m} (J_{p,q,k,1}^r)$ only depends on dimer configurations inside $\Lambda(J_{p,q,k,1}^r, 2 [\tilde{d}_m])$.  This means that Proposition~\ref{prop:couple2} applies with $R$ chosen to be $2\tilde{d}_m^2+\tilde{d}_m$, which means that for $|r_1|,|r_2| < \tilde{d}_m$  
\begin{equation}
	\begin{split}\label{eq:errorterm}
		&C (2d_m^2+d_m)^2 e^{-c_0 (\log m)^2} \geq \bigg|	\mathbb{P}_{\mathrm{sm}}\left[   h^{a,\tilde{d}_m}_l (J_{p,q,k_1,1}^r)=r_1 ,h^{a,\tilde{d}_m} (J_{p,q,k_2,1}^r)=r_2 \right]\\&- \mathbb{P}_{\mathrm{sm}}\left[   h^{a,\tilde{d}_m} (J_{p,q,k_1,1}^r)=r_1 \right] \mathbb{P}_{\mathrm{sm}}\left[h^{a,\tilde{d}_m} (J_{p,q,k_2,1}^r)=r_2 \right]\bigg| 
	\end{split}
\end{equation}
	for $k_1 \not = k_2$ and $C,c_0$ as given in Proposition~\ref{prop:couple2}.  Using the above equation, we have that
		\begin{equation}
		\begin{split}
			&\mathbb{E}_{\mathrm{sm}}\left[   h_l^{a,\tilde{d}_m} (J_{p,q,k_1,1}^r) h_l^{a,\tilde{d}_m} (J_{p,q,k_2,1}^r) \right]\\& = \sum_{|r_1|,|r_2| <\tilde{d}_m}r_1r_2 \mathbb{P}_{\mathrm{sm}}\left[   h_l^{a,\tilde{d}_m} (J_{p,q,k_1,1}^r)=r_1 ,h_l^{a,\tilde{d}_m} (J_{p,q,k_2,1}^r)=r_2 \right]. \\
			&= o(1)+\sum_{|r_1|,|r_2| <\tilde{d}_m}r_1r_2 \mathbb{P}_{\mathrm{sm}}\left[   h^{a,\tilde{d}_m}_l (J_{p,q,k_1,1}^r)=r_1]\mathbb{P}_{\mathrm{sm}} [h^{a,\tilde{d}_m}_l (J_{p,q,k_2,1}^r)=r_2 \right]  \\
			&=\mathbb{E}_{\mathrm{sm}}\left[   h_l^{a,\tilde{d}_m} (J_{p,q,k_1,1}^r) \right]\mathbb{E}_{\mathrm{sm}}\left[   h_l^{a,\tilde{d}_m} (J_{p,q,k_2,1}^r) \right]+o(1) 
		\end{split}
	\end{equation}
	where the $o(1)$ error term after the second equality comes from the bound in~\eqref{eq:errorterm} multiplied by the number of terms in the sum. { By recalling that $\mathbb{E}_{\mathrm{sm}}\left[   h_l^{a} (J_{p,q,k,1}^r) \right]=\mathbb{E}_{\mathrm{sm}}\left[   h_l^{a,\tilde{d}_m} (J_{p,q,k,1}^r) \right] =0$, this means we have verified~\eqref{eq:covbound}}. By noting that $\mathbb{E}_{\mathrm{sm}}\left[   h_l^{a,\tilde{d}_m} (J_{p,q,k,1}^r)^2 \right] < \tilde{d}_m^2$ and $\tilde{d}_m^2/M \to 0$ as $m\to \infty$, 
we have shown that
$$
\mathbb{P}_{\mathrm{sm}}\left[ \frac{1}{4M} \sum_{k=1}^M h_l^{a,\tilde{d}_m} (J_{p,q,k,1}^r) > \varepsilon\right] \to 0 \hspace{5mm}\mbox{ as } m \to \infty
$$
as required.

\end{proof}
For the next lemma, let $\mathtt{Path}(p,q,r)$ be the event in the Aztec diamond that there is a path that intersects the set $S^{p,q,r}$ of $a$-edges between $J^r_{p,q,1,1}$ and $J^r_{p,q,M,1}$. 
\begin{lemma} \label{lem:nocorridors}
	For $a \in (0,1/3)$, and all $1 \leq p \leq L_2$, $1 \leq q \leq L_1$, 
	\begin{equation}
		\lim_{m \to \infty} \mathbb{P}_{\mathrm{Az}}[\mathtt{Path}(p,q,r)]=0.
	\end{equation}

\end{lemma}
\begin{proof}
	Let us call a consecutive set of $a$-edges that are part of a loop or a path a \emph{sequence} of $a$-edges. Since a path starts and ends at the boundary, any path that intersects $S^{p,q,r}$ has to have a sequence of $a$-edges in $\Lambda_L^{p,q,r}$ of length greater than or equal to $d_m = [M \lambda_1 (\log m)^2]$ by definition of $L$ in~\eqref{eq:L}. Then, we have
	\begin{equation}
		\begin{split} \label{eq:pathsAztectosmooth}
			&\mathbb{P}_{\mathrm{Az}}[\mathtt{Path}(p,q,r)] \\
			&\leq   \mathbb{P}_{\mathrm{Az}} [ \exists \mbox{ a seq. of $a$-edges in $\Lambda^{p,q,r}_L$ intersecting $S^{p,q,r}$ with length} \geq d_m] \\
			&= \mathbb{P}_{\mathrm{sm}} [ \exists \mbox{ a seq. of $a$-edges in $\Lambda^{p,q,r}_L$ intersecting $S^{p,q,r}$ with length} \geq d_m] +o(1)
		\end{split}
	\end{equation}
	where the last step follows by applying Proposition~\ref{prop:couple}. The last probability is the probability in the full-plane smooth phase of the restriction of an event in the full-plane.   By Theorem~\ref{thm:onecorridor}, there are no paths almost surely in the full-plane smooth phase, so the sequence of $a$-edges has to be part of a loop.  Thus, we have 
\begin{equation}
		\begin{split}
			&\mathbb{P}_{\mathrm{sm}} [ \exists \mbox{ a seq. of $a$-edges in $\Lambda^{p,q,r}_L$ intersecting $S^{p,q,r}$ with length} \geq d_m]  \\
			&\leq  \mathbb{P}_{\mathrm{sm}} [ \mbox{there is a loop intersecting $S^{p,q,r}$ of length} \geq d_m]  \\
			&\leq\frac{|S^{p,q,r}|}{1-3a} (3a)^{d_m} \leq \frac{M}{1-3a} (3a)^{[M \lambda_1 (\log m)^2]} =o(1)
		\end{split}
	\end{equation}
as $m\to \infty$. Combining this with~\eqref{eq:pathsAztectosmooth}, we have proved the lemma.  

\end{proof}

We now give the proof of Theorem~\ref{thm:main}.

\begin{proof}[Proof of Theorem~\ref{thm:main}]
	From the formulas for $\kappa_m$ and $\nu_m$ in~\eqref{eq:kappa} and~\eqref{def:nu_m}, we have that
	\begin{equation}
		\begin{split} \label{eq:thmproofbigsumsplit}
			&\kappa_m(\{\beta_q \} \times A_p) - \nu_m(\{\beta_q \} \times A_p)= \frac{1}{4M} \sum_{k=1}^M h^a_\mathtt{c}  (J^r_{p,q,1,1}) - h^a_\mathtt{c} ( J^l_{p,q,1,1}) \\&- h^a_\mathtt{c} ( J^r_{p,q,k,1})- h^a_l ( J^r_{p,q,k,1})+ h^a_\mathtt{c} ( J^l_{p,q,k,1})+ h^a_l ( J^l_{p,q,k,1})\\
			&=\frac{1}{4M} \sum_{k=1}^M h^a_l ( J^l_{p,q,k,1})- h^a_l ( J^r_{p,q,k,1}) +\frac{1}{4M} \sum_{k=2}^M h^a_\mathtt{c}  (J^r_{p,q,1,1})- h^a_\mathtt{c}  (J^r_{p,q,k,1})\\
			&-\frac{1}{4M} \sum_{k=2}^M h^a_\mathtt{c}  (J^l_{p,q,1,1})- h^a_\mathtt{c}  (J^l_{p,q,k,1}) 
		\end{split} 
	\end{equation}
where the first equality follows from using~\eqref{loopcorridor}. We have that 
$$
\frac{1}{4M} \sum_{k=1}^M h^a_l ( J^l_{p,q,k,1})- h^a_l ( J^r_{p,q,k,1}) \to 0
$$
	as $m\to \infty$ with probability tending to one by Lemma~\ref{lem:loopssmall}. From Lemma~\ref{lem:nocorridors}, no paths separate $J^r_{p,q,1,1}$ and $J^r_{p,q,k,1}$ for $1 \leq k \leq M$ with probability tending to one, and so we conclude that 
$$	
	\frac{1}{4M} \sum_{k=2}^M h^a_\mathtt{c}  (J^r_{p,q,1,1})- h^a_\mathtt{c}  (J^r_{p,q,k,1}) \to 0
	$$
	as $m \to \infty$ with probability tending to one. { A similar argument shows that 
$$	
	\frac{1}{4M} \sum_{k=2}^M h^a_\mathtt{c}  (J^l_{p,q,1,1})- h^a_\mathtt{c}  (J^l_{p,q,k,1}) \to 0
	$$
with probability tending to one.  }
 We have shown that  
	$$
\kappa_m(\{\beta_q \} \times A_p) - \nu_m(\{\beta_q \} \times A_p) \to 0
	$$
	as $m\to \infty$ with probability tending to 1.  Since $\nu_m$ converges to the Airy kernel point process weakly, we conclude that so does $\kappa_m$.  
\end{proof}

\section{Proof of Proposition~\ref{prop:nu_mconverge}}  \label{sec:nuconverge}
{	
Before giving the proof of Proposition~\ref{prop:nu_mconverge}, we need the following lemma.
	\begin{lemma} \label{lem:subtlecomps} There exists $R>0$ such that 
\begin{equation}
\lim_{m \to \infty} \mathbb{E}_{\mathrm{Az}}\bigg[\exp\bigg[\frac{1}{M}\sum_{p=1}^{L_2} \sum_{q=1}^{L_1} \sum_{k=1}^M w_{p,q}(h^a(J_{p,q,k,1}^r)-h^a(J_{p,q,k,k}^r))\bigg]=1
\end{equation}
and 
\begin{equation}
\lim_{m \to \infty} \mathbb{E}_{\mathrm{Az}}\bigg[\exp\bigg[\frac{1}{M}\sum_{p=1}^{L_2} \sum_{q=1}^{L_1} \sum_{k=1}^M w_{p,q}(h^a(J_{p,q,k,k}^r)-h^a(J_{p,q,1,k}^r))\bigg]=1
\end{equation}
		for all $|w_{p,q}|<R$ where $1 \leq p \leq L_2$ and $1 \leq q \leq L_1$. 
	\end{lemma}
This is proved in Appendix~\ref{Appendix:proofsubtlecomps}.
We can now give the proof of Proposition~\ref{prop:nu_mconverge}.

\begin{proof}[Proof of Proposition~\ref{prop:nu_mconverge}]

Let $w_{p,q}=u_{p,q} + \mathrm{i}v_{p,q}$ where $u_{p,q},v_{p,q} \in \mathbb{R}$ for $1 \leq p \leq L_2$, $1 \leq q \leq L_1$.  Define
\begin{equation}
\mathtt{R}_m= \sum_{p=1}^{L_2}\sum_{q=1}^{L_1}u_{p,q} \nu_{m}(\{\beta_q\} \times A_p), \hspace{4mm} 
\mathtt{R}'_m= \sum_{p=1}^{L_2}\sum_{q=1}^{L_1}v_{p,q} \nu_{m}(\{\beta_q\} \times A_p),
\end{equation}
\begin{equation}
\mathtt{S}_m= \sum_{p=1}^{L_2}\sum_{q=1}^{L_1}u_{p,q} \mu_{m}(\{\beta_q\} \times A_p), \hspace{4mm} 
\mathtt{S}'_m= \sum_{p=1}^{L_2}\sum_{q=1}^{L_1}v_{p,q} \mu_{m}(\{\beta_q\} \times A_p),
\end{equation}
\begin{equation}
\mathtt{T}_m= \sum_{p=1}^{L_2}\sum_{q=1}^{L_1}u_{p,q} \mu_{\mathrm{Ai}}(\{\beta_q\} \times A_p), \hspace{4mm} \mbox{and}  \hspace{4mm}
\mathtt{T}'_m= \sum_{p=1}^{L_2}\sum_{q=1}^{L_1}v_{p,q} \mu_{\mathrm{Ai}}(\{\beta_q\} \times A_p).
\end{equation}
We want to prove that there exists $r_0>0$ so that if $|u_{p,q}| \leq r_0$, $|v_{p,q}| \leq r_0$ for all $1 \leq p \leq L_2$, $1 \leq q \leq L_1$, then 
\begin{equation}\label{eq:prop_numconvergeproofeq1}
\lim_{m \to \infty} \mathbb{E}_{\mathrm{Az}}[e^{\mathtt{R}_m-\mathtt{T}_m+\mathrm{i} (\mathtt{R}'_m-\mathtt{T}'_m)}]=1.
\end{equation}
Define for $\zeta \in \mathbb{C}$, $|\mathrm{Re}\zeta|<1$
\begin{equation}
\mathtt{F}_m(\zeta) = \mathbb{E}_{\mathrm{Az}}[e^{\mathtt{R}_m-\mathtt{T}_m+\zeta (\mathtt{R}'_m-\mathtt{T}'_m)}].
\end{equation}
This is an analytic function in $\zeta$.  We need the following claim whose proof is postponed. 

\begin{claim}\label{claim:nu_mconverge}
There is an $r_0>0$ so that if $|u_{p,q}| \leq r_0/2$, $|v_{p,q}| \leq r_0/2$ for all $1 \leq p \leq L_2$, $1 \leq q \leq L_1$ then $\lim_{m \to \infty} \mathtt{F}_m(t)=1$ for all $t \in \mathbb{R}$, $|t| \leq 1$. 
\end{claim} 

We first show that this implies that  $\lim_{m \to \infty} \mathtt{F}_m(\mathrm{i})=1$ for $|u_{p,q}| \leq r_0/2$, $|v_{p,q}| \leq r_0/2$ for all $1 \leq p \leq L_2$, $1 \leq q \leq L_1$, which is exactly what we want to prove.  This follows if we show that the family of functions $\{\mathtt{F}_m(\zeta) \}_{m \geq 1}$ is a normal family for $|\mathrm{Re} \zeta|<1$, since this fact combined with Claim~\ref{claim:nu_mconverge} implies that $\mathtt{F}_m(\zeta) \to 1$ uniformly on compact subsets of  $|\mathrm{Re} \zeta|<1$, which implies~\eqref{eq:prop_numconvergeproofeq1}. We next show that $\{\mathtt{F}_m(\zeta) \}_{m \geq 1}$ is a normal family for $|\mathrm{Re} \zeta|<1$. If $t$ and $x$ are real and $|t| \leq r$, then 
\begin{equation}
e^{tx} \leq e^{|tx|} \leq e^{r|x|} \leq e^{rx}+e^{-rx}.
\end{equation}
If $|\mathrm{Re} \zeta|<1$, this inequality gives 
\begin{equation}
\begin{split}
|\mathtt{F}_m(\zeta)|& \leq \mathbb{E}_{\mathrm{Az}}[e^{\mathtt{R}_m-\mathtt{T}_m+ (\mathrm{Re} \zeta)(\mathtt{R}'_m-\mathtt{T}'_m)}] \\& \leq 
\mathbb{E}_{\mathrm{Az}}[e^{\mathtt{R}_m-\mathtt{T}_m+ \mathtt{R}'_m-\mathtt{T}'_m}]+\mathbb{E}_{\mathrm{Az}}[e^{\mathtt{R}_m-\mathtt{T}_m-(\mathtt{R}'_m-\mathtt{T}'_m)}]
\end{split}
\end{equation}
By Claim~\ref{claim:nu_mconverge}, the right side converges to 2 as $m$ tends to infinity and is bounded by 4 for sufficiently large $m$. Thus, we have $|\mathtt{F}_m(\zeta)| \leq 4$ for all $|\mathrm{Re} \zeta| \leq 1$ and for sufficiently large $m$. From Montel's theorem, we have that $\{\mathtt{F}_m(\zeta) \}_{m \geq 1}$ is a normal family for $|\mathrm{Re} \zeta|<1$.

It remains to prove Claim~\ref{claim:nu_mconverge}. We need the following claim whose proof is postponed until after the proof of Claim~\ref{claim:nu_mconverge}.
\begin{claim}\label{claim:nu_mconverge2}
There is an $r_1>0$ so that if $|u_{p,q}| \leq r_1/2$, $|v_{p,q}| \leq r_1/2$ for all $1 \leq p \leq L_2$, $1 \leq q \leq L_1$
\begin{equation}
\limsup_{m \to \infty} \mathbb{E}_{\mathrm{Az}} [e^{2(\mathtt{R}_m-\mathtt{S}_m)+2t(\mathtt{R}'_m-\mathtt{S}'_m)}] \leq 1 \label{claim:nu_mconverge2A}
\end{equation}
and
\begin{equation}
\liminf_{m \to \infty} \mathbb{E}_{\mathrm{Az}} [\mathtt{R}_m-\mathtt{S}_m +t (\mathtt{R}'_m-\mathtt{S}'_m)] =0\label{claim:nu_mconverge2B}
\end{equation}
for all $t\in \mathbb{R}$, $|t|\leq 1$.
\end{claim}

\begin{proof}[Proof of Claim~\ref{claim:nu_mconverge}]
From Theorem~\ref{thm:previousBCJ}, we have that there exists $r_0$ such that 
\begin{equation}\label{thm:previousBCJeqform}
\lim_{m \to \infty} \mathbb{E}_{\mathrm{Az}}[e^{\mathtt{S}_m-\mathtt{T}_m+t (\mathtt{S}_m'-\mathtt{T}_m')}]=1
\end{equation}
for $|u_{p,q}|\leq r_0$, $|v_{p,q}|\leq r_0$ for all $1 \leq p \leq L_2$, $1 \leq q \leq L_1$ and $|t| \leq 1$.  The Cauchy-Schwarz inequality gives 
\begin{equation}
\begin{split}
&\mathbb{E}_{\mathrm{Az}}[e^{\mathtt{R}_m-\mathtt{T}_m+t (\mathtt{R}_m'-\mathtt{T}_m')}]=\mathbb{E}_{\mathrm{Az}}[e^{\mathtt{R}_m-\mathtt{S}_m+\mathtt{S}_m-\mathtt{T}_m+t (\mathtt{R}_m'-\mathtt{S}'_m+\mathtt{S}'_m-\mathtt{T}_m')}] \\
&\leq \mathbb{E}_{\mathrm{Az}}[e^{2(\mathtt{R}_m-\mathtt{S}_m)+2t (\mathtt{R}_m'-\mathtt{S}_m')}]^{1/2}\mathbb{E}_{\mathrm{Az}}[e^{2(\mathtt{S}_m-\mathtt{T}_m)+2t (\mathtt{S}_m'-\mathtt{T}_m')}]^{1/2}
\end{split}
\end{equation}
It follows from~\eqref{claim:nu_mconverge2A} and~\eqref{thm:previousBCJeqform} that 
\begin{equation}
\limsup_{m \to \infty} \mathbb{E}_{\mathrm{Az}}[e^{\mathtt{R}_m-\mathtt{T}_m+t (\mathtt{R}_m'-\mathtt{T}_m')}]\leq 1.
\end{equation}

Conversely, by Jensen's inequality we have
\begin{equation}
\begin{split}
&\mathbb{E}_{\mathrm{Az}}[e^{\mathtt{R}_m-\mathtt{S}_m+\mathtt{S}_m-\mathtt{T}_m+t (\mathtt{R}_m'-\mathtt{S}'_m+\mathtt{S}'_m-\mathtt{T}_m')}] \\
&\geq \exp\big( \mathbb{E}_{\mathrm{Az}}[\mathtt{R}_m-\mathtt{S}_m+t(\mathtt{R}'_m-\mathtt{S}'_m)] \big)\exp\big( \mathbb{E}_{\mathrm{Az}}[\mathtt{R}_m-\mathtt{T}_m+t(\mathtt{R}'_m-\mathtt{T}'_m)]\big).
\end{split}
\end{equation}
It follows from~\eqref{claim:nu_mconverge2B} and~\eqref{thm:previousBCJeqform} that
\begin{equation}
\liminf_{m \to \infty} \mathbb{E}_{\mathrm{Az}}[e^{\mathtt{R}_m-\mathtt{T}_m+t (\mathtt{R}_m'-\mathtt{T}_m')}]\geq 1
\end{equation}
which proves the claim with $r_0=r_1/2$. 

\end{proof}

\begin{proof}[Proof of Claim~\ref{claim:nu_mconverge2}]
To prove~\eqref{claim:nu_mconverge2A}, we have by expanding out the definitions of $\mathtt{R}_m,\mathtt{S}_m, \mu_m$, and $\nu_m$
\begin{equation}
\begin{split}
&\mathbb{E}_{\mathrm{Az}} \left[ e^{2(\mathtt{R}_m-\mathtt{S}_m)+2t(\mathtt{R}'_m-\mathtt{S}'_m)} \right] =\mathbb{E}_{\mathrm{Az}}\bigg[ \exp \bigg[ \frac{2}{M} \sum_{p=1}^{L_2} \sum_{q=1}^{L_1} \sum_{k=1}^M (u_{p,q}+t v_{p,q}) \\
& \times \big( h^a(J_{p,q,k,1}^r) -h^a(J_{p,q,1,k}^r)-(h^a(J_{p,q,k,1}^l)-h^a(J_{p,q,1,k}^l)) \big) \bigg] \bigg]
\end{split}
\end{equation}
Applying Cauchy-Schwarz gives
\begin{equation}
\begin{split}\label{eq:claimproof:nu_mconverge2Aeq}
&\mathbb{E}_{\mathrm{Az}} \left[ e^{2(\mathtt{R}_m-\mathtt{S}_m)+2t(\mathtt{R}'_m-\mathtt{S}'_m)} \right] \\
& \leq \mathbb{E}_{\mathrm{Az}} \bigg[ \exp \bigg[ \frac{4}{M} \sum_{p=1}^{L_2} \sum_{q=1}^{L_1} \sum_{k=1}^M (u_{p,q}+t v_{p,q})  \big( h^a(J_{p,q,k,1}^r) -h^a(J_{p,q,1,k}^r) \big) \bigg] \bigg]^{1/2} \\
&\times\mathbb{E}_{\mathrm{Az}} \bigg[ \exp \bigg[ \frac{4}{M} \sum_{p=1}^{L_2} \sum_{q=1}^{L_1} \sum_{k=1}^M (u_{p,q}+t v_{p,q})  \big(-h^a(J_{p,q,k,1}^l) +h^a(J_{p,q,1,k}^l) \big) \bigg] \bigg]^{1/2}.
\end{split}
\end{equation}
For the first term on the right side of~\eqref{eq:claimproof:nu_mconverge2Aeq}, we use Cauchy-Schwarz again
\begin{equation}
\begin{split}
& \mathbb{E}_{\mathrm{Az}} \bigg[ \exp \bigg[ \frac{4}{M} \sum_{p=1}^{L_2} \sum_{q=1}^{L_1} \sum_{k=1}^M (u_{p,q}+t v_{p,q})  \big( h^a(J_{p,q,k,1}^r) -h^a(J_{p,q,1,k}^r) \big) \bigg] \bigg]  \\
&\leq \mathbb{E}_{\mathrm{Az}} \bigg[ \exp \bigg[ \frac{8}{M} \sum_{p=1}^{L_2} \sum_{q=1}^{L_1} \sum_{k=1}^M (u_{p,q}+t v_{p,q})  \big( h^a(J_{p,q,k,1}^r) -h^a(J_{p,q,k,k}^r) \big) \bigg] \bigg]^{\frac{1}{2}} \\ 
&\times \mathbb{E}_{\mathrm{Az}} \bigg[ \exp \bigg[ \frac{8}{M} \sum_{p=1}^{L_2} \sum_{q=1}^{L_1} \sum_{k=1}^M (u_{p,q}+t v_{p,q})  \big( h^a(J_{p,q,k,k}^r) -h^a(J_{p,q,1,k}^r) \big) \bigg] \bigg]^{\frac{1}{2}} 
\end{split}
\end{equation}
and conclude using Lemma~\ref{lem:subtlecomps} that the right side tends to 1 as $m$ tends to infinity. A similar computation holds for the second term on the right side of~\eqref{eq:claimproof:nu_mconverge2Aeq} using an analogous version of Lemma~\ref{lem:subtlecomps}.

To prove~\eqref{claim:nu_mconverge2B}, we expand out the definitions of $\mathtt{R}_m,\mathtt{S}_m, \mu_m$, and $\nu_m$ which gives 
\begin{equation}
\begin{split}
&\mathbb{E}_{\mathrm{Az}}[\mathtt{R}_m-\mathtt{S}_m+t(\mathtt{R}'_m-\mathtt{S}_m)] =\frac{1}{M} \sum_{p=1}^{L_2} \sum_{q=1}^{L_1} \sum_{k=1}^M  (u_{p,q}+t v_{p,q}) \\
&\times\mathbb{E}_{\mathrm{Az}} \big[ h^a(J_{p,q,k,1}^r) -h^a(J_{p,q,1,k}^r)\big]-\mathbb{E}_{\mathrm{Az}} \big[ h^a(J_{p,q,k,1}^l) -h^a(J_{p,q,1,k}^r) \big].
\end{split}
\end{equation}
We only focus on the first expectation on the right side; the second is analagous. The expectation of  height differences is the signed sum of dimer probabililites, which can be evaluated by Theorem~\ref{localstatisticsthm} using the asymptotic entries of $K^{-1}_{a,1}$ at the rough-smooth boundary.  As the distance between $J_{p,q,k,1}^r$ and $J_{p,q,1,k}^r$ is atmost $CM(\log m)^2$ which bounds the number of dimer probabilities involved, the contributions from $\mathbb{K}_A$ are negligible as $m\to \infty$, see Theorem~\ref{Airyasymptotics}, and so only contributions from the $\mathbb{K}^{-1}_{1,1}$ are relevant. This means we have
\begin{equation}
\begin{split}
&\mathbb{E}_{\mathrm{Az}} \big[ h^a(J_{p,q,k,1}^r) -h^a(J_{p,q,1,k}^r)\big]\\
&=\mathbb{E}_{\mathrm{sm}} \big[ h^a(J_{p,q,k,1}^r) -h^a(J_{p,q,k,k}^r)\big]- \mathbb{E}_{\mathrm{sm}} \big[h^a(J_{p,q,k,1}^r) -h^a(J_{p,q,k,k}^r) \big]+o(1)=o(1)
\end{split}
\end{equation}
as $m \to \infty$, where we have used the fact that the smooth phase is flat (so the expected height change between $a$-faces in directions parallel to $e_1$ or $e_2$ is zero - we omit the computation).
\end{proof}

\end{proof}
}

\section{Proofs of Proposition~\ref{prop:couple} and Proposition~\ref{prop:couple2}}\label{section:couple}

\begin{proof}[Proof of Proposition~\ref{prop:couple}]
	In the proof below, we write  $	p_{\mathrm{Az}}(\cdot)=p_{\mathrm{Az}}(\cdot| \Lambda^{(u,v)}_L)$, $p_{\mathrm{sm}}(\cdot)=p_{\mathrm{sm}}(\cdot |\Lambda^{(u,v)}_L)$. 

	Let $1 \leq i_1' < \dots < i_{R-r}' \leq R$ and $1 \leq j_1' < \dots < j_{R-r}' \leq R$ where $1 \leq r \leq R$ be given.  Observe that, 
\begin{equation}
	\begin{split}
		&| \det (A_{i_p',j_q'}(\overline{s}))_{1 \leq p,q \leq R-r} | = |\det(K_{a,1}(w_{i_l'},w_{i_p'}) \mathbb{K}_{1,1}^{-1} (w_{i_p'},w_{j_q'}+f_{s_{j_q'}}))_{1 \leq p,q \leq R-r}| \\
		&= \mathbb{P}_{\mathrm{sm}} [ \mbox{all edges }(w_{i_p'},w_{j_p'}+f_{s_{j_p'}}), 1 \leq p \leq R-r \mbox{ are covered}]
	\end{split}
\end{equation}
Consequently,
\begin{equation}
	\begin{split}
		&\sum_{\overline{s} \in [4]^R} | \det (A_{i_p',j_q'}(\overline{s}))_{1 \leq p,q \leq R-r} | \\
		&= 4^r \sum_{s_{j_1'},\dots,s_{j_{R-r}'} \in [4]} \mathbb{P}_{\mathrm{sm}} [ \mbox{edges }(w_{i_p'},w_{j_p'}+f_{s_{j_p'}}), 1 \leq p \leq r \mbox{ are covered}] \leq 4^r \label{gascouple:six}
	\end{split}
\end{equation}
since all the events in the sum are disjoint, they give different dimer configurations.

	Write 
	\begin{equation}
		A(\overline{s}) = ( A_{ij}(\overline{s}) )_{1 \leq i,j\leq R} =  (\overline{A}_1(\overline{s}) \dots  \overline{A}_R(\overline{s}))
	\end{equation}
where 
	\begin{equation}\label{gascouple:expandC1}
		\overline{A}_j(\overline{s})= \left( \begin{array}{c}
			A_{1j}(\overline{s}) \\ 
			\vdots \\
		A_{Rj}(\overline{s}) \\  \end{array} \right)
	\end{equation}
	and similarly for $C(\overline{s})$. Let $\overline{e}_i$ be the standard basis column vectors, $1 \leq i \leq R$, so that 
	\begin{equation}\label{gascouple:expandC}
		\overline{C}_j(\overline{s}) = \sum_{i=1}^R C_{ij}(\overline{s}) \overline{e}_i.
	\end{equation}
	Then, 
	\begin{equation}
		\begin{split}
			&|\det (A_{ij}(\overline{s})+m^{-1/3} C_{ij}(\overline{s}))_{1\leq i,j \leq R} - \det (A_{ij}(\overline{s}))_{1\leq i,j \leq R}|\\ 
			&\leq \sum_{r=1}^R \sum_{1\leq j_1 < \dots < j_r \leq R} \frac{1}{m^{r/3}} |\det (\overline{C}_{j_1}(\overline{s}) \dots \overline{C}_{j_r}(\overline{s})\overline{A}_{j_1'}(\overline{s})\dots \overline{A}_{j_{R-r}'}(\overline{s}) )|  \label{gascouple:nine}
		\end{split}
	\end{equation}
	where $[R] \backslash \{j_1, \dots, j_r \} = \{ j_1' < \dots< j_{R-r}' \}$.  
	Now, by~\eqref{gascouple:expandC}, 

	\begin{equation}
	\begin{split}
		&|\det(\overline{C}_{j_1}(\overline{s}) \dots \overline{C}_{j_r}(\overline{s})\overline{A}_{j_1'}(\overline{s})\dots \overline{A}_{j_{R-r}'}(\overline{s}) ) \\
		&= \left|\sum_{i_1,\dots, i_r=1}^R C_{i_1j_1}(\overline{s})\dots C_{i_rj_r}(\overline{s}) \det (\overline{e}_{i_1} \dots \overline{e}_{i_r} \overline{A}_{j_1'} (\overline{s}) \dots \overline{A}_{j_{R-r}'} (\overline{s}) )\right| \\
		&\leq C_0^r\sum_{i_1,\dots, i_r=1}^R |\det (\overline{e}_{i_1} \dots \overline{e}_{i_r} \overline{A}_{j_1'} (\overline{s}) \dots \overline{A}_{j_{R-r}'} (\overline{s}))|, \label{gascouple:ten}
	\end{split}
	\end{equation}
	by~\eqref{constC0}.  Note that 
	\begin{equation}
|\det (\overline{e}_{i_1} \dots \overline{e}_{i_r} \overline{A}_{j_1'} (\overline{s}) \dots \overline{A}_{j_{R-r}'} (\overline{s}))|=0
	\end{equation}
	if $i_p = i_q$ for some $p \not = q$. Thus, 
	\begin{equation}
		\begin{split}
			&\sum_{i_1,\dots, i_r=1}^R | \det (\overline{e}_{i_1} \dots \overline{e}_{i_r} \overline{A}_{j_1'} (\overline{s}) \dots \overline{A}_{j_{R-r}'} (\overline{s}))| \\
			&= r! \sum_{1\leq i_1<\dots< i_r<R} | \det (\overline{e}_{i_1} \dots \overline{e}_{i_r} \overline{A}_{j_1'} (\overline{s}) \dots \overline{A}_{j_{R-r}'} (\overline{s}))|\\
			&=r! \sum_{1\leq i_1<\dots< i_r<R} | \det(A_{i_p'j_q'}(\overline{s}))_{1 \leq p,q \leq R-r}| \\ \label{gascouple:twelve}
		\end{split}
	\end{equation}
	which can be seen by expanding the determinant along the first $r$ columns, where $[R] \backslash \{i_1,\dots, i_r \} = \{i_1'<\dots<  i_{R-r}' \}$. Combining~\eqref{gascouple:nine},~\eqref{gascouple:ten} and~\eqref{gascouple:twelve}, we see that 
	\begin{equation}
		\begin{split}
			&|\det (A_{ij}(\overline{s})+m^{-1/3} C_{ij}(\overline{s}))_{1 \leq i,j \leq R} - \det (A_{ij}(\overline{s}))_{1 \leq i,j \leq R}|\\
			&\leq \sum_{r=1}^R \left( \frac{C_0}{m^{1/3}} \right)^r r! \sum_{1 \leq i_1 < \dots < i_r \leq R} \sum_{1 \leq j_1 < \dots < j_r \leq R} | \det (A_{i_p'j_q'} (\overline{s}))_{1 \leq p,q \leq R-r}|. \label{gascouple:thirteen}
		\end{split}
	\end{equation}
Thus by~\eqref{gascouple:pAz},~\eqref{gascouple:S} and~\eqref{gascouple:thirteen}
	\begin{equation}
		\begin{split}
			&\sum_{\overline{s} \in \Omega } \big|p_{\mathrm{Az}} \big(\overline{s}|\Lambda_L^{(x,y)}\big) - p_{\mathrm{sm}}\big(\overline{s}|\Lambda_L^{(x,y)}\big)\big|\\
			&= \sum_{\overline{s} \in \Omega } |\det (A_{ij}(\overline{s})+m^{-1/3} C_{ij}(\overline{s}))_{1\leq i,j\leq R} - \det (A_{ij}(\overline{s}))_{1\leq i,j\leq R}| \\ 
			&\leq \sum_{r=1}^R \left( \frac{C_0}{m^{1/3}} \right)^r r! \sum_{\substack{ 1 \leq i_1 < \dots i_r \leq R  \\ 1 \leq j_1 < \dots< j_r \leq R  }} \sum_{\overline{s} \in \Omega } | \det (A_{i_p'j_q'} (\overline{s}))|_{1\leq p, q \leq R-r} \\
			& \leq \sum_{r=1}^R \left( \frac{C_0}{m^{1/3}} \right)^r r!\sum_{\substack{ 1 \leq i_1 < \dots i_r \leq R  \\ 1 \leq j_1 < \dots< j_r \leq R  }} 4^r,
		\end{split}
	\end{equation}
where we also used~\eqref{gascouple:six} in the last inequality.  Thus, using $r! \left( \begin{array}{cc} R \\ r \end{array}\right) \leq R^r$, we have
	\begin{equation}
		\begin{split}
			&\sum_{\overline{s} \in \Omega } |p_{\mathrm{Az}} (\overline{s}) - p_{\mathrm{sm}}(\overline{s}) | \leq \sum_{r=1}^R \left( \frac{4C_0}{m^{1/3}} \right)^r r! \left( \begin{array}{cc} R \\ r \end{array}\right)^2 \\
				&\leq \sum_{r=1}^R \left( \frac{4C_0R}{m^{1/3}} \right)^r  \left( \begin{array}{cc} R \\ r \end{array}\right) 
					= \left(1+4C_0Rm^{-1/3} \right)^R -1\\& = e^{R \log (1+4C_0Rm^{-1/3})}-1 \leq e^{4C_0R^2m^{-1/3}}-1 \leq 4C_0R^2 m^{-1/3}e
		\end{split}
	\end{equation}
	provided that $4 C_0 R^2 m^{-1/3} \leq 1$, as required.  
\end{proof}

We need the following lemma whose proof is in the Appendix~\ref{appendix}.
\begin{lemma} \label{anothersmoothbound}
	For $1 \leq i \leq R$ and $R+1 \leq j \leq 2R$ or $1 \leq j \leq R$ and $R+1 \leq i \leq 2R$ with $R<\lambda_1^2(\log m)^4/4+ \lambda_1 (\log m)^2$, there exists constants $c_0,D>0$ such that 
		\begin{equation}
|F_{ij}(\overline{s})| \leq  De^{-c_0(\log m)^2}
		\end{equation}
	\end{lemma}

\begin{proof}[Proof of Proposition~\ref{prop:couple2}]
	The computation is very similar to the one for Proposition~\ref{prop:couple} and so we give a shortened computation. We have that 
	\begin{equation}
		\begin{split}
			&	|\det (D_{ij}(\overline{s}))_{1\leq i,j \leq 2R} - \det (E_{ij}(\overline{s}))_{1\leq i,j \leq 2R}|\\ 
			&\leq \sum_{r=1}^R \sum_{1\leq j_1 < \dots < j_r \leq 2R}  |\det (\overline{F}_{j_1}(\overline{s}) \dots \overline{F}_{j_r}(\overline{s})\overline{E}_{j_1'}(\overline{s})\dots \overline{E}_{j_{2R-r}'}(\overline{s}) )|  
		\end{split}
	\end{equation}
	where $[2R] \backslash \{j_1, \dots, j_r \} = \{ j_1' < \dots< j_{2R-r}' \}$ and we use the same notation as given in~\eqref{gascouple:expandC1}.  Using the notation given in~\eqref{gascouple:expandC} and following the same steps given in Proposition~\ref{prop:couple2}, we have that the left side of the above equation is bounded above by 
	\begin{equation}
		\begin{split}
			&\sum_{r=1}^R \sum_{1\leq j_1 < \dots < j_r \leq 2R}  \left|\sum_{i_1, \dots, i_r=1}^{2R} F_{i_1j_1}(\overline{s}) \dots F_{i_rj_r}(\overline{s})  \det (\overline{e}_{i_1} \dots \overline{e}_{i_r} \overline{E}_{j_1'} (\overline{s}) \dots \overline{E}_{j_{2R-r}'} (\overline{s}) )\right| \\
			&=\sum_{r=1}^R r! \sum_{\substack{ 1\leq j_1 < \dots < j_r \leq 2R \\ 1\leq i_1 < \dots < i_r \leq 2R}} 	\left| F_{i_1j_1}(\overline{s}) \dots F_{i_rj_r}(\overline{s}) \det(E_{i_p'j_q'}(\overline{s}))_{1 \leq p,q \leq 2R-r} \right|
		\end{split}
	\end{equation}
	by the same argument given in~\eqref{gascouple:twelve}. We use the bound from Lemma~\ref{anothersmoothbound} for each $F_{i_lj_l}$ $1 \leq l \leq r$ to get
	\begin{equation}
		\begin{split}
			&\sum_{\overline{s} \in \Omega} |\det (D_{ij}(\overline{s}))_{1\leq i,j \leq 2R} - \det (E_{ij}(\overline{s}))_{1\leq i,j \leq 2R}|\\ 
			& \leq  \sum_{r=1}^{2R} D^r e^{-r c_0 (\log m)^2 }  \sum_{\substack{ 1\leq j_1 < \dots < j_r \leq 2R \\ 1\leq i_1 < \dots < i_r \leq 2R}} \sum_{\overline{s} \in \Omega} \left|\det(E_{i_p'j_q'}(\overline{s}))_{1 \leq p,q \leq 2R-r} \right| \\
			&\leq  16 R^2 D e^{-c_0(\log m)^2}
\end{split}
	\end{equation}
by following the same steps given in the last two equations in the proof of Proposition~\ref{prop:couple2}. 

\end{proof}

\section{Geometry of the full-plane smooth phase} \label{sec:geometry} 

In this section, we introduce directed random spanning trees and give three differently weighted graphs $\mathbb{L}_R$, $\mathbb{L}_R^{\mathtt{w}}$ and $\mathbb{L}_R^{\mathtt{f}}$, which are equivalent in dimer model measure.  We give the explicit gauge transformations between the measures. We show that the dimer model on $\mathbb{L}_R$ converges weakly to the full-plane smooth phase when $R \to \infty$.    Using this and extending the notion of corridors to the full-plane smooth phase, we show that there is only one corridor almost surely. 

\subsection{Directed spanning tree}\label{subsec:spanning}

In this subsection, among introducing directed spanning trees, we also give  three different weightings for a dimer model (which will eventually be shown to be \emph{gauge equivalent}) and describe the spanning tree correspondence for two of these weightings.

Consider a finite connected directed graph embedded in the plane. Assign weights to each directed edge of the graph. Note that the weight of the edge from $u$ to $v$ is not necessarily equal to the weight of the edge from $v$ to $u$.   A \emph{directed spanning tree} with root $\mathbf{r}$  (also known as an \emph{arborescence}) $T$ is a connected union of edges of $G$ such that each vertex of the graph has exactly one outgoing edge in $T$ except for the root $\mathbf{r}$ which has only incoming edges.  The weight of a directed spanning tree $T$ is the product of the weights of the directed edges of $T$. The random directed spanning tree is a probability measure on the set of directed spanning trees with the probability of picking a directed spanning tree being proportional to the weight of the directed spanning tree.  Random spanning tree is a rich subject but we will not review this here;~\cite{BLPS:01}.

Temperley~\cite{Tem:74} found a  bijection between random spanning tree of an $n \times m$ rectangle in $\mathbb{Z}^2$ and dimer covers on $(2m-1) \times (2n-1)$ with a corner vertex removed. This bijection was generalized in~\cite{KPW:00}, providing a bijection between directed weighted spanning trees on a connected planar graph and dimer coverings on a related graph.  Rather than describe this bijection in its full setting, we restrict to the setting relevant for this paper.

Introduce a bipartite graph (for the dimer model) which has white vertices given by
\[
	\begin{split}
		\bar{\mathtt{W}}&= \{ (2i+1-2R , 2j+2-2R): 0 \leq i \leq 2R-2, 0\leq j \leq 2R-2 \}  \\
		&\cup\{ (4i+1-2R , -2R): 1 \leq i \leq R-1 \}\\
		& \cup\{ (-1-2R , 4j+2-2R): 0 \leq j \leq R-1 \}  \\
		&\cup\{ (4i+1-2R , 2R): 0 \leq i \leq R-1 \}\\
		& \cup\{ (2R-1 , 4j+2-2R): 0 \leq j \leq R-1 \}  \\
	\end{split}
\]
and black vertices given by 
\[
	\bar{\mathtt{B}}= \{ (2i-2R , 2j+1-2R): 0 \leq i \leq 2R-1, 0\leq j \leq 2R-1 \}  \\
	\]
where $R >1$.  The edges between the white and black vertices are parallel to $e_1=(1,1)$ and $e_2=(-1,1)$; see Fig~\ref{fig:graphLR}. 
\begin{figure}
\begin{center}
\includegraphics[height=5cm]{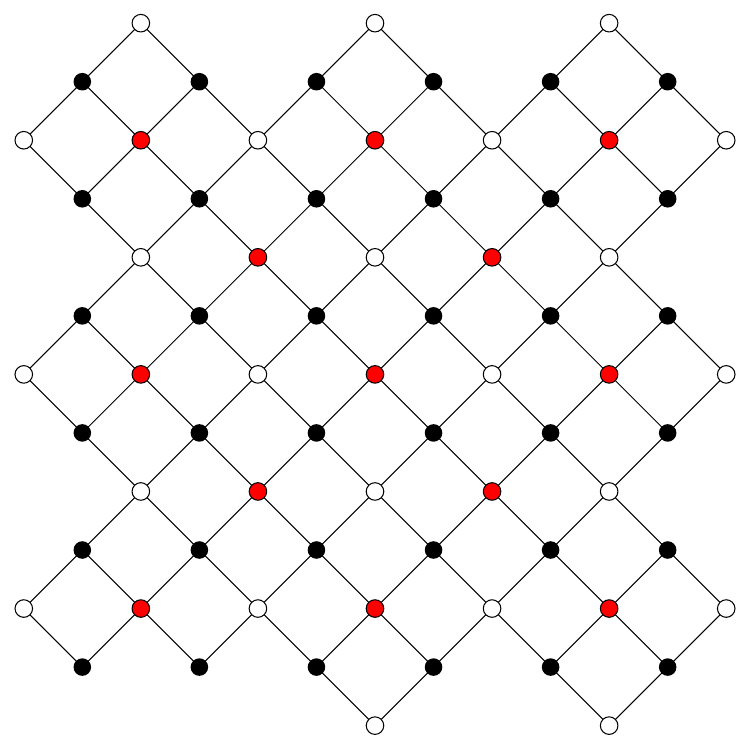}
\includegraphics[height=5cm]{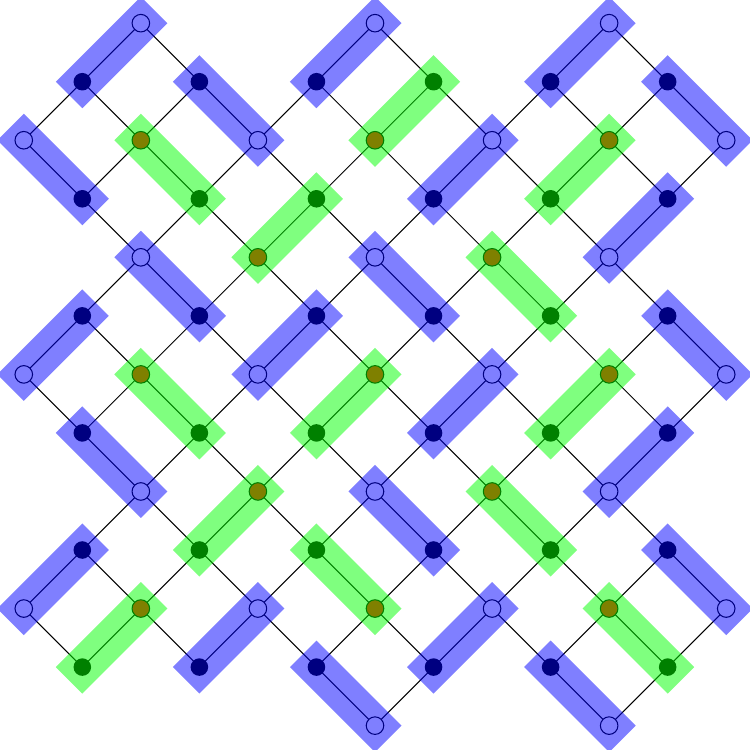}
	\caption{The left figure shows the graph $\mathbb{L}_R$ for $R=3$. The graphs $\mathbb{L}_R^{\mathtt{w}}$ and $\mathbb{L}_R^{\mathtt{w}}$ have the same vertex and edge sets as $\mathbb{L}_R$, but different edgeweights.  The vertices in  $\overline{\mathtt{W}}_1$ are colored in red while the vertices in  $\overline{\mathtt{W}}_0$ are colored in white. The right figure shows a dimer covering on $\overline{\mathbb{L}}_R$ with the dimers responsible for the tree on  $\mathbb{T}_R^{\mathtt{w},p}$ (and  $\mathbb{T}_R^{\mathtt{f},d}$) colored green and the dimers responsible for the tree on  $\mathbb{T}_R^{\mathtt{w},d}$ (and  $\mathbb{T}_R^{\mathtt{f},p}$) colored blue.} 
\label{fig:graphLR}
\end{center}
\end{figure}
As before, we have the same convention of $\bar{\mathtt{W}}_0,\bar{\mathtt{W}}_1, \bar{\mathtt{B}}_0$ and $\bar{\mathtt{B}}_1$, that is 
\begin{equation}
	\bar{\mathtt{W}}_i= \{(x,y) \in \bar{\mathtt{W}}: x+y \mod4=2i+1\} \hspace{5mm}\mbox{for } i \in\{0,1\}
\end{equation}
and
\begin{equation}
	\bar{\mathtt{B}}_i= \{(x,y) \in \bar{\mathtt{B}}: x+y \mod4=2i+1\} \hspace{5mm}\mbox{for } i \in\{0,1\}.
\end{equation}
We introduce three different weightings for this bipartite graph and label them accordingly. For $j,k \in \{0,1\}$, $i \in \{1,2\}$, and $w \in \overline{\mathtt{W}}_j$, if the edges $(w,w+(-1)^k e_i)$ have weight
\begin{itemize}
\item $a^{(1-j)(1-k)+kj}$, then label the graph $\mathbb{L}_R$;
\item $a^{2kj}$, then label the graph $\mathbb{L}_R^{\mathtt{w}}$;
\item $a^{2(1-k)(1-j)}$, then label the graph $\mathbb{L}_R^{\mathtt{f}}$,
\end{itemize}
that is the graph labels above are  sets of vertices, edges, as well as their edge weights. 
The first weighting above is the two-periodic weighting for this graph, the second has its edge weights that are not equal to one on edges incident to vertices in $\mathtt{W}_1$ while the third has its  edge weights that are not equal to one on edges incidenct to vertices in $\mathtt{W}_0$.   

Recall that dimer model is uniquely parameterized by its \emph{face weights}, that is the measure is uniquely determined by the alternating product of the edge weights around each face. It is easy to see that the dimer models on $\mathbb{L}_R$, $\mathbb{L}_R^{\mathtt{w}}$ and $\mathbb{L}_R^{\mathtt{f}}$ have the same face weights and hence the measures are the same, that is, they are gauge equivalent. We show below the explicit \emph{gauge transformations} between the dimer models.

We now describe the tree correspondence for the dimer model on $\mathbb{L}_R^\mathtt{w}$. We use the same convention as above that the graph label includes the graph's vertices, edges as well as the edge weights. 
 The graph for the \emph{primal} directed spanning tree, $\mathbb{T}_R^{\mathtt{w},p}$, has vertex set given by $\bar{\mathtt{W}}_1$ while the graph of the \emph{dual} directed spanning tree, $\mathbb{T}_R^{\mathtt{w},d}$, has  vertex set given by $\bar{\mathtt{W}}_0 \cup(-1-2R,-2R)$.  The edges in $\mathbb{T}_R^{\mathtt{w},d}$ and $\mathbb{T}_R^{\mathtt{w},p}$ are parallel to $\pm e_1$ and $\pm e_2$. For each dimer $(w,w\pm e_i)$ with $i \in \{1,2\}$  and $w \in \bar{\mathtt{W}}_1$, there is a directed edge in the directed spanning tree from $w$ to $w\pm2 e_i$ with the directed edge having the same weight as its corresponding dimer.  That is, the directed edges of $\mathbb{T}_R^{\mathtt{w},d}$ of the form $(v,v+2(-1)^ke_i)$ have weights $a^{2k}$ for $i \in \{1,2\}$ and $k \in \{0,1\}$.   The same correspondence holds for dimers incident to vertices in $\bar{\mathtt{W}}_0$ but these give the dual directed spanning tree and so all directed edges in $\mathbb{T}_R^{\mathtt{w},d}$ have weight 1. The choice in boundary conditions of $\mathbb{L}_R^{\mathtt{w}}$ means that all boundary vertices $\mathbb{T}_R^{\mathtt{w},p}$ are connected to a single vertex, that is, a \emph{wired directed spanning tree}. The dual spanning tree is rooted at the vertex $(-1-2R,-2R)$.  
It is immediate that once the primal tree has been found, the dual tree is fully determined and deterministic.  Moreover, the above correspondence between dimers to directed edges can be simply reversed, so that given a primal tree with the above weights, the dual tree and the resulting dimer configuration are completely determined, with each dimer configuration having weight given by the product of its edge weights. As a consequence, each pair of directed spanning trees in the above construction only depends on the  primal directed spanning tree, and so it follows that the dimer model $\mathbb{L}_R^{\mathtt{w}}$ is equivalent to the primal random directed spanning tree $\mathbb{T}_R^{\mathtt{w},p}$.   

Next we describe the tree correspondence for the dimer model on $\mathbb{L}_R^\mathtt{f}$. This time, the graph for the \emph{primal} directed spanning tree, $\mathbb{T}_R^{\mathtt{f},p}$, has vertex set given by $\bar{\mathtt{W}}_0\cup(-1-2R,-2R)$ while the graph of the \emph{dual} directed spanning tree, $\mathbb{T}_R^{\mathtt{f},d}$, has  vertex set given by $\bar{\mathtt{W}}_1$. The same correspondence between dimers and edges in the tree given in the correspondence on $\mathbb{L}_R^{\mathtt{w}}$ holds in this case.  Here, the primal tree $\mathbb{T}_R^{\mathtt{f},p}$ is rooted at $(-1-2R,-2R)$, the dual tree is wired, and the dimer configuration on $\mathbb{L}_R^{\mathtt{f}}$ is completely determined by the primal tree on $\mathbb{T}_R^{\mathtt{f},p}$.

As noted above for the Aztec diamond, there is a height function defined on faces $\mathbb{L}_R$ in one-to-one correspondence (up to height level) and dimer configurations on $\mathbb{L}_R$. Due to the bijection between dimers on $\mathbb{L}_R^\mathtt{w}$ and trees on $\mathbb{T}^{\mathtt{w},p}_R$, the height function is in correspondence with trees on $\mathbb{T}^{\mathtt{w},p}_R$~\cite{KPW:00}.   { In particular, each directed edge on $\mathbb{T}^{\mathtt{w},p}_R$ corresponds to two incident edges (which are in the same direction) on $\mathbb{L}_R^{\mathtt{w}}$, exactly one of which is covered by a dimer.  Due to the correspondence between trees, dimers and heights, there are four heights around each directed edge on $\mathbb{T}^{\mathtt{w},p}_R$~\cite{KPW:00} (since there are four faces incident to each directed edge).  }  
  The main observation we need from~\cite{KPW:00} is that two directed edges of the same type are only able to join the same tree if after unwinding\footnote{the winding number is defined as the number of right turns minus left turns},  their heights match.

Random directed drifted spanning tree can be generated using Wilson's algorithm~\cite{Wil:96}, which gives a convenient tool for infinite limits.  Wilson's algorithm is briefly described as follows: define the loop erasure of a finite path $\mathcal{P}$, denoted by $\mathrm{LE}(\mathcal{P})$ to be the path after chronologically removing the loops of $\mathcal{P}$. This is well-defined when $\mathcal{P}$ does not visit any vertex infinitely often.  Consider any ordering of the vertices $\{v_1, \dots ,v_{4R^2}\}$ of $\mathbb{T}_R^{\mathtt{w},p}$ and set $\mathcal{F}_0=\emptyset$.  Let $\mathcal{P}_i$ denote the path  generated by a random walk with weights $(1,1,a^{2},a^{2})$ in the directions $(e_1,e_2,-e_1,-e_2)$ started from $v_i$ which terminates if it exits $\mathbb{T}_R^{\mathtt{w},p}$ (i.e. it hits the single vertex connected to all the boundary vertices of $\mathbb{T}_R^{\mathtt{w},p}$) or hits $\mathcal{F}_{i-1}$ (if $v_i \in\mathcal{F}_{i-1}$, then the random walk has already hit $\mathcal{F}_{i-1}$). Then set $\mathcal{F}_i=\mathcal{F}_{i-1} \cup \mathrm{LE}(\mathcal{P}_i)$. The tree $\mathcal{F}_{4R^2}$ is a random drifted directed spanning tree. Note that the distribution of the tree is independent on the choice of ordering of the vertices~\cite{Wil:96}.

Finally, we mention that we denote the infinite graph of $\mathbb{T}_R^{\mathtt{w},p}$, that is in the limit as $R \to \infty$, by $\mathbb{T}^{\mathtt{w},p}$.  For the wired directed spanning tree on $\mathbb{T}_R^{\mathtt{w},p}$ one can take the limit as $R \to \infty$ without considering weak limits using \emph{Wilson's algorithm rooted at infinity}~\cite{BLPS:01} giving a wired directed spanning forest on $\mathbb{T}^{\mathtt{w},p}$~\cite{BLPS:01}, where the underlying directed edges have weights $a^{2k}$ for $(v,v+2(-1)^ke_i)$ for $k \in \{0,1\}$ and $i\in\{1,2\}$ and $v \in \mathtt{W}_1^*$. Indeed, the algorithm relies on the underlying random walk to be transient, which is the case for this directed spanning tree, and can be described as follows: let $\mathcal{F}_0=\emptyset $ and let $v_1, v_2\dots$ be an enumeration of the vertices in ${\mathtt{W}}_1^*$. Inductively, pick a vertex $v_n$ and run the drifted random walk from $v_n$. Stop the walk when it hits $\mathcal{F}_{n-1}$, otherwise let it run indefinitely.  Call this walk $\mathcal{P}_n$.  Set $\mathcal{F}_n=\mathcal{F}_{n-1} \cup \mathcal{P}_n$ and $\mathcal{F}=\bigcup_n \mathcal{F}_n$. Then, from~\cite{BLPS:01}, $\mathcal{F}$ has the same distribution as the wired directed forest on $\mathbb{T}^{\mathtt{w},p}$.  Moreover, we have the following.

\begin{prop} \label{prop:tree}
	The wired directed spanning forest on $\mathbb{T}^{\mathtt{w},p}$ is a single tree almost surely. 
\end{prop}
The original statement for uniform spanning trees was due to Pemantle~\cite{Pem:91}. The above result follows from the formulation in~\cite{BLPS:01}. Indeed, one only needs to show that two independent drifted random walks intersect with probability one when started from two different points in $\mathbb{Z}^2$~\cite{LPS:03}; see for example~\cite[Theorem 10.22]{LP:16}.  This is shown in~\cite[Theorem 1.3]{SSSWX:18}, so the proof of the result is complete.

\subsection{Gauge Transformation}

The act of multiplying all the edges incident to a vertex of a graph by a constant is called a \emph{gauge transformation}. This transformation does not change the dimer model measure.  We consider each of the three dimer models defined in Section~\ref{subsec:spanning} and give the explicit gauge transformations.

\begin{prop}\label{prop:gaugeequivalent}
The gauge transformation to get from the dimer model on $\mathbb{L}_R$ to the dimer model on $\mathbb{L}_R^{\mathtt{w}}$ is given by
\begin{itemize}
	\item muliplying each vertex $x=(x_1,x_2) \in \bar{\mathtt{W}}_j$ with $j\in \{0,1\}$ by $a^{j+\frac{1}{2}(x_2-2+2R)}$,
\item muliplying each vertex $y=(y_1,y_2) \in \bar{\mathtt{B}}$ by $a^{-\frac{1}{2}(y_2-1+2R)}$.
\end{itemize}
The gauge transformation to get from the dimer model on $\mathbb{L}_R$ to the dimer model on $\mathbb{L}_R^{\mathtt{f}}$ is given by
\begin{itemize}
	\item muliplying each vertex $x=(x_1,x_2) \in \bar{\mathtt{W}}_j$ with $j\in \{0,1\}$ by $a^{-j-\frac{1}{2}(x_2-2+2R)}$,
\item muliplying each vertex $y=(y_1,y_2) \in \bar{\mathtt{B}}$ by $a^{\frac{1}{2}(y_2-1+2R)}$.
\end{itemize}
\end{prop}
\begin{proof}

We apply the first gauge transformation to $\mathbb{L}_R$. By doing so, around $x=(x_1,x_2) \in \bar{\mathtt{W}}_j$  the edges $(x,x+(-1)^k e_i)$ for $k \in \{0,1\}$, $i \in \{1,2\}$ have weight 
$$
a^{(1-j)(1-k)+kj}a^{j+\frac{1}{2}(x_2-2+2R)}a^{-\frac{1}{2}(x_2+1-2k-1+2R)},
$$	
where the first factor is the weight of the edge while the second and third factors are from the multiplications assigned to the white and black vertices respectively.  
Simplifying the above formula gives
	$$
a^{(1-j)(1-k)+kj+j+k-1}=a^{2kj},
$$
which are the edge weights of $\mathbb{L}_R^{\mathtt{w}}$.  

Next, we apply the second gauge transformation to $\mathbb{L}_R$.  Then, around $x=(x_1,x_2) \in \bar{\mathtt{W}}_j$  the edges $(x,x+(-1)^k e_i)$ for $k \in \{0,1\}$, $i \in \{1,2\}$ have weight 
$$
a^{(1-j)(1-k)+kj}a^{-j-\frac{1}{2}(x_2-2+2R)}a^{\frac{1}{2}(x_2+1-2k-1+2R)},
$$
where the first factor is the weight of the edge while the second and third factors are from the multiplications assigned to the white and black vertices respectively.  Simplifying the above formula gives
$$
	a^{(1-j)(1-k)+kj-j-k+1}=a^{2(1-j)(1-k)  }
$$
which are the edge weights of $\mathbb{L}_R^{\mathtt{f}}$.

\end{proof}
\begin{remark}
{As a consequence of Proposition~\ref{prop:gaugeequivalent},}
the dimer model on $\mathbb{L}_R$ is equivalent to the directed random spanning tree on $\mathbb{T}_R^{\mathtt{w},p}$ and to the directed random spanning tree on $\mathbb{T}_R^{\mathtt{f},p}$.
\end{remark}

\subsection{Convergence to the full-plane smooth phase}

The Kasteleyn matrix on $\mathbb{L}_R$ reads for $(x,y) \in \bar{\mathtt{B}} \times \bar{\mathtt{W}}$
\begin{equation} \label{pf:KLambda}
       K(x,y)=\left\{\begin{array}{ll}
		     a (1-j) + b j  & \mbox{if } y=x+e_1, x \in \bar{\mathtt{B}}_j \\
		     (a j +b (1-j) ) \mathrm{i} & \mbox{if } y=x+e_2, x \in \bar{\mathtt{B}}_j\\
		     a j + b (1-j)  & \mbox{if } y=x-e_1, x \in \bar{\mathtt{B}}_j \\
		     (a (1-j) +b j ) \mathrm{i} & \mbox{if } y=x-e_2, x \in \bar{\mathtt{B}}_j\\
			0 & \mbox{if $(x,y)$ is not an edge}.
		     \end{array} \right.
\end{equation}
The following proposition shows that entries of $K^{-1}$ converge to their full-plane smooth phase counterpart which indicates that as $R\to \infty$, the dimer model on $\mathbb{L}_R$ converges weakly to the full-plane smooth phase.  

\begin{prop}\label{prop:fintetoinfinite}
	For ${x}\in \bar{\mathtt{W}}_{\varepsilon_1}$, ${y}\in \bar{\mathtt{B}}_{\varepsilon_2}$ fixed in terms of $R$,  with $\eps_1,\eps_2 \in \{0,1\}$ and all $a \in (0,1)$  we have
\begin{equation}
|K^{-1}(x,y)- \mathbb{K}_{1,1}^{-1}(x,y) | \leq CRe^{-c_0R}
\end{equation}
where $c_0, C>0$ are constants.
\end{prop}
A few remarks are in order.

\begin{remark}
	\begin{enumerate}
		\item Although ${x}\in \bar{\mathtt{W}}_{\varepsilon_1}$, ${y}\in \bar{\mathtt{B}}_{\varepsilon_2}$ in the above proposition, the choice in coordinate system for the graph $\mathbb{L}_R$ has the same parity as the Aztec diamond and the full-plane as well. 
\item We expect that the bound in Proposition~\ref{prop:fintetoinfinite} could be sharpened to $e^{-c_0R}$, by using more precise estimates, e.g.those estimates from~\cite[Section 4]{CJ:16}.  However, this requires a much more delicate computation than the one given here.  
	\end{enumerate}
\end{remark}

The approach taken is partly based from a computation in~\cite{CY:13} with a useful simplification valid for this setting. The above result and the local statistics theorem, Theorem~\ref{localstatisticsthm}, guarantees the measure on $\mathbb{L}_R$ converges weakly as $R \to \infty$ to the full-plane smooth phase measure.   Below, we use that $x=(x_1,x_2)$ and $y=(y_1,y_2)$ without further mention.  

We let for $i,j \in \{0,1\}$
\begin{equation}
	G^{i,j} = G^{i,j}(w_1,w_2,b_1,b_2) = \sum_{{x \in \bar{\mathtt{W}}_i}}\sum_{ y \in \bar{ \mathtt{B}}_j} K^{-1} (x,y) w_1^{x_1}w_2^{x_2} b_1^{y_1} b_2^{y_2},
\end{equation}
that is, the generating function of the inverse Kasteleyn matrix on $\mathbb{L}_R$ with the variables $(w_1,w_2)$ marking the white vertex coordinate and variables $(b_1,b_2)$ marking the black vertex coordinate.  We also need restrictions on the generating function.  Here, we will abuse notation and denote 
\begin{equation}
	G^{i,j}\bigg|_{\substack{x \in A \\ y \in B}}=\sum_{{x \in \bar{\mathtt{W}}_i}}\sum_{ y \in \bar{ \mathtt{B}}_j} K^{-1} (x,y) w_1^{x_1}w_2^{x_2} b_1^{y_1} b_2^{y_2} \mathbb{I}_{x \in A}\mathbb{I}_{y \in B}.
\end{equation}
We will also use the notation that $f_r(w)=(1-w^r)/(1-w)$.

\begin{proof}[Proof of Proposition~\ref{prop:fintetoinfinite}]
We give the computation in full for vertices in $\bar{\mathtt{W}}_1 \times \bar{\mathtt{B}}_0$ and the other computations follow from the same method.  For space reasons, we omit these additional computations but highlight the main differences.

Consider the matrix $\Delta_a= K^* K$, where $K^*$ is the conjugate transpose of $K$. For $x=(x_1,x_2) \in \bar{\mathtt{W}}_1$ and $y=(y_1,y_2) \in \bar{\mathtt{B}}_0$ and since $K K^{-1}=\mathbb{I}$, we have
\begin{equation}
\Delta_a K^{-1} (x,y)= \sum_{b \in \bar{\mathtt{B}}} \sum_{w \in \bar{\mathtt{W}}} K^*(x,b) K(b,w) K^{-1}(w,y)=\sum_{b \sim x }K^*(x,b)\mathbb{I}_{b=y}
\end{equation} 
where $\sum_{b \sim x }$ denotes the sum over vertices $b$ that are nearest neighbored vertices to $x$.  Notice that we can instead expand out $K^* K$ first in $\Delta_a K^{-1} (x,y)$ which gives 
\begin{equation}
\begin{split} \label{eq:gfwhite}
&\sum_{b \sim x }K^*(x,b)\mathbb{I}_{b=y}=\Delta_a K^{-1} (x,y)= a \bigg( K^{-1}(x+2e_1,y) \mathbb{I}_{x_1<2R-3} \mathbb{I}_{x_2<2R-2} \\
&+K^{-1}(x+2e_2,y) \mathbb{I}_{x_1>1-2R} \mathbb{I}_{x_2<2R-2}  
+K^{-1}(x-2e_1,y) \mathbb{I}_{x_1>1-2R} \mathbb{I}_{x_2>2-2R} \\
&+K^{-1}(x-2e_2,y) \mathbb{I}_{x_1<2R-3} \mathbb{I}_{x_2>2-2R}  \bigg)
	+2(1+a^2) K^{-1}(x,y) \hspace{7mm}\mbox{for } x \in \bar{\mathtt{W}}_1.
\end{split}
\end{equation}
Here, the indicator functions keep track of the boundary of the box. 
We multiply the above equation by $w_1^{x_1}w_2^{x_2}b_1^{y_1}b_2^{y_2}$ and sum over $x \in \bar{\mathtt{W}}_1$ and $y\in \bar{\mathtt{B}}_0$, simplifying each term into generating function formulas. For example, under this procedure we have
\begin{equation}
\begin{split}
	&\sum_{ x \in \bar{\mathtt{W}}_1} \sum_{ y \in \bar{\mathtt{B}}_0} K^{-1} (x+2 e_1,y) \mathbb{I}_{x_1<2R-3} \mathbb{I}_{x_2<2R-2} w_1^{x_1}w_2^{x_2}b_1^{y_1}b_2^{y_2} 
	=\frac{1}{w_1^2 w_2^2}\sum_{ x \in \bar{\mathtt{W}}_1} \sum_{ y \in \bar{\mathtt{B}}_0} \\ & \times\big(
1- \mathbb{I}_{x_1=1-2R} -\mathbb{I}_{x_2=2-2R} +\mathbb{I}_{x=(1-2R,2-2R)} 
\big)K^{-1}(x,y)  w_1^{x_1}w_2^{x_2}b_1^{y_1}b_2^{y_2}\\
	&=\frac{1}{w_1^2 w_2^2}\big( G^{1,0}- G^{1,0}\big|_{ x_1=1-2R  }- G^{1,0}\big|_{ x_2=2-2R}+ G^{1,0}\big|_{ x=(1-2R,2-2R)} \big) \\
	&=\frac{1}{w_1^2 w_2^2}\big( G^{1,0}- G^{1,0}\big|_{ x_1=1-2R  }- G^{1,0}\big|_{ \substack{ x_2=2-2R \\ x_1 \not=1-2R}} \big) .
\end{split}
\end{equation}
By applying this procedure to all terms in~\eqref{eq:gfwhite}, and after collecting terms we arrive at
\begin{equation} \label{eq:gfwhite2}
\begin{split}
	& \big( a (w_1^{-2}+w_1^2)(w_2^{-2}+w_2^2)+2(1+a^2) \big) G^{1,0}\\
	&-a ( w_2^{-2}+w_2^2) \big( w_1^{-2}G^{1,0}\big|_{x_1=1-2R } +w_1^{2}G^{1,0}\big|_{x_1=2R-3 } \big) \\ 
	&-a ( w_1^{-2}+w_1^2) \bigg( w_2^{-2}G^{1,0}\bigg|_{\substack{ x_2=2-2R \\ x_1 \not = 1-2R,2R-3}} +w_2^{2}G^{1,0}\bigg|_{\substack{ x_2=2R-2 \\ x_1\not = 1-2R,2R-3}} \bigg) \\ 
	&= \sum_{ x \in \bar{\mathtt{W}}_1}\sum_{ y \in \bar{\mathtt{B}}_0}\bigg(\sum_{b \sim x }K^*(x,b)\mathbb{I}_{b=y} \bigg)  w_1^{x_1} w_2^{x_2} b_1^{y_1} b_2^{y_2} 
\end{split}
\end{equation} 
Notice that the first term on the left side in the above expression is $\tilde{c}(w_1^2,w_2^2)G^{1,0}$.  
We set 
\begin{equation}
	d^1_{10}(w_1,w_2,b_1,b_2)=\sum_{ x \in \bar{\mathtt{W}}_1} \sum_{ y \in \bar{\mathtt{B}}_0}\bigg(\sum_{b \sim x }K^*(x,b)\mathbb{I}_{b=y} \bigg)  w_1^{x_1} w_2^{x_2} b_1^{y_1} b_2^{y_2},
\end{equation}
which is the right side of~\eqref{eq:gfwhite2}.  For $d^1_{10}(w_1,w_2,b_1,b_2)$, 
we expand out the right side of the above equation by using the definition of $K^*$, use the indicator function and the fact that the black vertices are in $\bar{\mathtt{B}}_0$. This gives
\begin{equation}
\begin{split}
&d^1_{10}(w_1,w_2,b_1,b_2)\\
&=\sum_{ x \in \bar{\mathtt{W}}_1} \sum_{ y \in \bar{\mathtt{B}}_0}
(\mathbb{I}_{x+e_1=y} -\mathrm{i} \mathbb{I}_{x+e_2=y}+a \mathbb{I}_{x-e_1=y}-a \mathrm{i} \mathbb{I}_{x-e_2=y}) w_1^{x_1} w_2^{x_2} b_1^{y_1} b_2^{y_2} \\
&=\sum_{ x \in \bar{\mathtt{W}}_1} \sum_{ y \in \bar{\mathtt{B}}_0}(\mathbb{I}_{x+e_1=y} +a \mathbb{I}_{x-e_1=y}) w_1^{x_1} w_2^{x_2} b_1^{y_1} b_2^{y_2} \\
&=\sum_{ x \in \bar{\mathtt{W}}_1} w_1^{x_1} w_2^{x_2} b_1^{x_1+1} b_2^{x_2+1}+ a w_1^{x_1} w_2^{x_2} b_1^{x_1-1} b_2^{x_2-1} \\  
&= \frac{( b_1 b_2 +a b_1^{-1} b_2^{-1})}{ w_1^{2R} w_2^{2R} b_1^{2R} b_2^{2R}}\\
&\times \bigg( w_1b_1 w_2^2 b_2^2 f_R(w_1^4 b_1^4) f_R(w_2^4 b_2^4) +w_1^3b_1^3 w_2^4 b_2^4 f_{R-1}(w_1^4 b_1^4) f_{R-1}(w_2^4 b_2^4)   \bigg) ,
\end{split}
\end{equation}
	where the two terms in parenthesis in the above equation are from vertices in $\overline{\mathtt{W}}_1$ whose coordinates are either of the form $(4i+1-2R, 4j-2R+2)$ with $0 \leq i ,j \leq  R-1$ or of the form $(4i+3-2R, 4j+4-2R)$ with $0 \leq i,j \leq  R-2$.
We also set
\begin{equation}
\begin{split}
&d^2_{10}(w_1,w_2,b_1,b_2)= 
	a ( w_2^{-2}+w_2^2) \big( w_1^{-2}G^{1,0}\big|_{ x_1=1-2R } +w_1^{2}G^{1,0}\big|_{ x_1=2R-3} \big) \\ 
	&+a ( w_1^{-2}+w_1^2) \bigg( w_2^{-2}G^{1,0}\bigg|_{\substack{   x_2=2-2R \\x_1 \not = 1-2R,2R-3}} +w_2^{2}G^{1,0}\bigg|_{\substack{  x_2=2R-2 \\ x_1\not = 1-2R,2R-3}} \bigg) .\\ 
\end{split}
\end{equation}
Then, we have that~\eqref{eq:gfwhite2} can be rewritten as 
\begin{equation}\label{eq:gfwhite3}
\begin{split}
	G^{1,0} =\frac{d^1_{10}}{\tilde{c}(w_1^2,w_2^2)}   + \frac{d^2_{10}}{\tilde{c}(w_1^2,w_2^2)}.
\end{split}
\end{equation}
Extracting coefficients of the $G^{1,0}$ in the above equation gives formulas for $K^{-1}(x,y)$.    We consider, for $\varepsilon>0$
\begin{equation}
	\frac{1}{(2\pi \mathrm{i})^4} \int_{\Gamma_{1-\varepsilon}}\frac{d w_1}{w_1} \int_{\Gamma_{1-\varepsilon}}\frac{d w_2}{w_2} \int_{\Gamma_{1-\varepsilon}}\frac{d b_1}{b_1} \int_{\Gamma_{1-\varepsilon}}\frac{d b_2}{b_2}  \frac{	G^{1,0}}{w_1^{\tilde{x}_1} w_2^{\tilde{x}_2} b_1^{\tilde{y}_1}b_2^{\tilde{y}_2}}
\end{equation} 
for each term in~\eqref{eq:gfwhite3}. These are given in the following two lemmas whose proofs are postponed until after completing the proof of the proposition.

\begin{lemma} \label{Claim1}
For $x \in \bar{\mathtt{W}}_1$ and $y \in \bar{\mathtt{B}}_0$, 
\begin{equation}
\begin{split}
&\frac{1}{(2\pi \mathrm{i})^4} \int_{\Gamma_{1-\varepsilon}}\frac{d w_1}{w_1} \int_{\Gamma_{1-\varepsilon}}\frac{d w_2}{w_2} \int_{\Gamma_{1-\varepsilon}}\frac{d b_1}{b_1} \int_{\Gamma_{1-\varepsilon}}\frac{d b_2}{b_2}\\ &\times \frac{ d_{10}^1(w_1,w_2,b_1,b_2)  }{ w_1^{x_1} w_2^{x_2} b_1^{y_1} b_2^{y_2}\tilde{c}(w_1^2,w_2^2)}= \mathbb{K}^{-1}_{1,1}(x,y)
\end{split}
\end{equation}
\end{lemma}

\begin{lemma}\label{Claim2}
For $C,c_0>0$ constants and $x_1,x_2,y_1,y_2$ fixed in terms of $R$,
\begin{equation}
\begin{split}
&\left|\frac{1}{(2\pi \mathrm{i})^4} \int_{\Gamma_{1-\varepsilon}}\frac{d w_1}{w_1} \int_{\Gamma_{1-\varepsilon}}\frac{d w_2}{w_2} \int_{\Gamma_{1-\varepsilon}}\frac{d b_1}{b_1} \int_{\Gamma_{1-\varepsilon}}\frac{d b_2}{b_2}\frac{d^2_{10}(w_1,w_2,b_1,b_2)}{\tilde{c}(w_1^2,w_2^2)w_1^{x_1} w_2^{x_2} b_1^{y_1}b_2^{y_2}}  \right| \\
&\leq CRe^{-c_0R}
\end{split}
\end{equation}

\end{lemma} 
We now proceed with the rest of the proof of the proposition.  From Lemmas~\ref{Claim1} and~\ref{Claim2}, it follows that only the first term on the right side of~\eqref{eq:gfwhite3} gives a contribution when extracting out the coefficient of $(x_1,x_2)$ and $(y_1,y_2)$ for the white and black vertices respectively while the other term tends to zero exponentially fast.  This verifies the proposition for the case when $x \in \bar{\mathtt{W}}_1$ and $y \in \bar{\mathtt{B}}_0$.  For the case $x \in \bar{\mathtt{W}}_1$ and $y \in \bar{\mathtt{B}}_1$, the difference is that to equation~\eqref{eq:gfwhite}, we multiply by $w_1^{x_1}w_2^{x_2}b_1^{y_1}b_2^{y_2}$ and sum over $x \in \bar{\mathtt{W}}_1$ and $y\in \bar{\mathtt{B}}_1$ instead. The rest of the computation proceeds in a similar fashion. For the case $x \in \bar{\mathtt{W}}_0$,~\eqref{eq:gfwhite} is no longer valid and instead, we have the equation
\begin{equation}
\begin{split}
&\sum_{b \sim x }K^*(x,b)\mathbb{I}_{b=y}=\Delta_a K^{-1} (x,y)= a \bigg( K^{-1}(x+2e_1,y) \mathbb{I}_{x_1<2R-1} \mathbb{I}_{x_2<2R} \\
&+K^{-1}(x+2e_2,y) \mathbb{I}_{x_1>-2R-1} \mathbb{I}_{x_2<2R}  +K^{-1}(x-2e_1,y) \mathbb{I}_{x_1>-2R-1} \mathbb{I}_{x_2>-2R} \\
&+K^{-1}(x-2e_2,y) \mathbb{I}_{x_1<2R-1} \mathbb{I}_{x_2>-2R}  \bigg)
	+2(1+a^2) K^{-1}(x,y) \hspace{7mm}\mbox{for }x \in \bar{\mathtt{W}}_0
\end{split}
\end{equation}
To this equation, we multiply by $w_1^{x_1}w_2^{x_2}b_1^{y_1}b_2^{y_2}$ and sum over $x \in \bar{\mathtt{W}}_0$ and $y\in \bar{\mathtt{B}}_0$ or $y\in \bar{\mathtt{B}}_1$ depending on the case. The main steps of the computation proceed as the case $x \in \bar{\mathtt{W}}_1$ and $y\in \bar{\mathtt{B}}_0$. Note that there are few additional terms due to the vertex $(-1-2R,-2R)$ not being present in $\bar{\mathtt{W}}_0$, but these term are negligible from the same reason behind Lemma~\ref{Claim2}.   \end{proof}

We next prove Lemma~\ref{Claim1}. 
\begin{proof}[Proof of Lemma~\ref{Claim1}] 
We expand out the integral in Lemma~\ref{Claim1} using the definition of $d_{10}^1(w_1,w_2,b_1,b_2)$ which gives
\begin{equation}
\begin{split}
&\frac{1}{(2\pi \mathrm{i})^4} \int_{\Gamma_{1-\varepsilon}}\frac{d w_1}{w_1} \int_{\Gamma_{1-\varepsilon}}\frac{d w_2}{w_2} \int_{\Gamma_{1-\varepsilon}}\frac{d b_1}{b_1} \int_{\Gamma_{1-\varepsilon}}\frac{d b_2}{b_2} \frac{( b_1 b_2 +a b_1^{-1} b_2^{-1})w_1b_1 w_2^2 b_2^2}{w_1^{2R+x_1} w_2^{2R+x_2} b_1^{2R+y_1} b_2^{2R+y_2}}
\\ 
&\times \frac{ 
 ( f_R(w_1^4 b_1^4)f_R(w_2^4 b_2^4)+ w_1^2b_1^2 w_2^2 b_2^2 f_{R-1}(w_1^4 b_1^4)f_{R-1}(w_2^4 b_2^4))}{ \tilde{c}(w_1^2,w_2^2)}.
\end{split}
\end{equation}
{ We take the change of variables $w_i = \sqrt{u_i}$ and $b_i =\sqrt{v_i}$ for $i \in \{1,2\}$ for the above integral, moving the contours of integration from $\Gamma_{(1-\eps)^2}$ to $\Gamma_{1-\eps}$ which does not pick up any additional contributions. This change of variables doubles the contour of integration for each integral but there is an extra factor of $1/2$ from each change of variables which means the above equation is equal to }
\begin{equation}
\begin{split}
&\frac{1}{(2\pi \mathrm{i})^4} \int_{\Gamma_{1-\varepsilon}}\frac{d u_1}{u_1} \int_{\Gamma_{1-\varepsilon}}\frac{d u_2}{u_2} \int_{\Gamma_{1-\varepsilon}}\frac{d v_1}{v_1} \int_{\Gamma_{1-\varepsilon}}\frac{d v_2}{v_2} \frac{( v_1 v_2 +a ) }{u_1^{\frac{2R+x_1-1}{2}} u_2^{\frac{2R+x_2-2}{2}} v_1^{\frac{2R+y_1}{2}} v_2^{\frac{2R+y_2-1}{2}}}
\\ 
&\times \frac{ 
 ( f_R(u_1^2 v_1^2) f_R(u_2^2 v_2^2)+ u_1u_2v_1v_2 f_{R-1}(u_1^2 v_1^2) f_{R-1}(u_2^2 v_2^2))}{ \tilde{c}(u_1,u_2)}
\end{split}
\end{equation}
In the above integral, we can compute the integrals with respect to $v_1$ and $v_2$. This amounts to extracting coefficients of $v_1^{R+\frac{y_1}{2}}$ and $v_2^{R+\frac{y_2-1}{2}}$ for $y\in \bar{\mathtt{B}}_0$ in the numerator of the integrand. Notice that we cannot get a contribution for this from both $f_R(u_1^2 v_1^2) f_R(u_2^2 v_2^2)$ and $ u_1u_2v_1v_2 f_{R-1}(u_1^2 v_1^2) f_{R-1}(u_2^2 v_2^2))$ because $y \in \bar{\mathtt{W}}_0$ (that is, one term gives a contribution for black vertices of the form $(4i+1,4j)$ while the other term gives a contribution for the vertices of the form $(4i+3,4j+2)$). Doing this extraction gives

\begin{equation}
\frac{1}{(2\pi \mathrm{i})^2} \int_{\Gamma_{1-\varepsilon}}\frac{d u_1}{u_1} \int_{\Gamma_{1-\varepsilon}}\frac{d u_2}{u_2}  \frac{( u_1^{-1} u_2^{-1} +a ) }{u_1^{\frac{2R+x_1-1}{2}} u_2^{\frac{2R+x_2-2}{2}} u_1^{-\frac{2R+y_1}{2}} u_2^{-\frac{2R+y_2-1}{2}}\tilde{c}(u_1,u_2)}
\end{equation}
and simplifying gives
\begin{equation}
\begin{split}
&\frac{1}{(2\pi \mathrm{i})^2} \int_{\Gamma_{1-\varepsilon}}\frac{d u_1}{u_1} \int_{\Gamma_{1-\varepsilon}}\frac{d u_2}{u_2}  \frac{( u_1^{-1} u_2^{-1} +a ) }{u_1^{\frac{x_1-y_1-1}{2}} u_2^{\frac{x_2-y_2-1}{2}}\tilde{c}(u_1,u_2)}\\
&=\frac{1}{(2\pi \mathrm{i})^2} \int_{\Gamma_{1-\varepsilon}}\frac{d u_1}{u_1} \int_{\Gamma_{1-\varepsilon}}\frac{d u_2}{u_2}  \frac{( 1 +au_1u_2 ) }{u_1^{\frac{x_1-y_1+1}{2}} u_2^{\frac{x_2-y_2+1}{2}}\tilde{c}(u_1,u_2)}.
\end{split}
\end{equation}
Since $\tilde{c}(u_1,u_2)$ contains no poles in $\{(u_1,u_2):1-\varepsilon \leq u_1,u_2 \leq 1 \}$, we deform both contours to $\Gamma_1$ and the above integral is exactly equal to $\mathbb{K}_{1,1}^{-1}(x,y)$.  

\end{proof}

We now prove Lemma~\ref{Claim2}.
\begin{proof}[Proof of Lemma~\ref{Claim2}]
We only show the bound for one generic term. The rest of the terms in $d_{10}^2(w_1,w_2,b_1,b_2)$ follow from similar computations, as explained after the bound on the generic term.

Consider the term
\begin{equation} \label{eq:integralboundd2}
	\frac{1}{(2\pi \mathrm{i})^4} \int_{\Gamma_{1-\varepsilon}}\frac{d w_1}{w_1} \int_{\Gamma_{1-\varepsilon}}\frac{d w_2}{w_2} \int_{\Gamma_{1-\varepsilon}}\frac{d b_1}{b_1} \int_{\Gamma_{1-\varepsilon}}\frac{d b_2}{b_2}  \frac{G^{1,0} |_{ x_1=1-2R}}{ \tilde{c}(w_1^2,w_2^2)w_1^{\tilde{x}_1} w_2^{\tilde{x}_2} b_1^{\tilde{y}_1}b_2^{\tilde{y}_2}}
\end{equation}
where $\tilde{x}=( \tilde{x}_1,\tilde{x}_2) \in\bar {\mathtt{W}}_1$ and $\tilde{y}=( \tilde{y}_1,\tilde{y}_2) \in\bar {\mathtt{B}}_0$ and we recall that 
\begin{equation}
	G^{1,0} \big|_{ x_1=1-2R}=\sum_{x \in \bar{\mathtt{W}}_1}\sum_{ y \in \bar{\mathtt{B}}_0} K^{-1}(x,y) w_1^{1-2R } w_2^{x_2} b_1^{y_1} b_2^{y_2}. 
\end{equation}
We take the change of variables $w_i = \sqrt{u_i}$ and $b_i =\sqrt{v_i}$ for $i \in \{1,2\}$ for the integral in~\eqref{eq:integralboundd2}, moving the contours of integration from $\Gamma_{(1-\eps)^2}$ to $\Gamma_{1-\eps}$ which does not pick up any additional contributions. This change of variables doubles the contour of integration for each integral but there is an extra factor of $1/2$ from each change of variables which means that~\eqref{eq:integralboundd2} is equal to
\begin{equation}\label{eq:integralboundd3}
\begin{split}
&\frac{1}{(2\pi \mathrm{i})^4} \int_{\Gamma_{1-\varepsilon}}\frac{d u_1}{u_1} \int_{\Gamma_{1-\varepsilon}}\frac{d u_2}{u_2} \int_{\Gamma_{1-\varepsilon}}\frac{d v_1}{v_1} \int_{\Gamma_{1-\varepsilon}}\frac{d v_2}{v_2}\frac{u_1^{\frac{1-2R-\tilde{x}_1}{2}}}{\tilde{c}(u_1,u_2)}  \\
&\times 
	\sum_{\substack{x \in \bar{\mathtt{W}}_1  \\x_1=1-2R}}\sum_{y \in \bar{\mathtt{B}}_0}  K^{-1}(x,y)  u_2^{\frac{x_2-\tilde{x}_2}{2}} v_1^{\frac{y_1-\tilde{y}_1}{2}} v_2^{\frac{y_2-\tilde{y}_2}{2}}, 
\end{split}
\end{equation}
{ where the above sum in $x=(x_1,x_2)$ is only summed over those pairs with $x_1=1-2R$.}
We perform the integrals in~\eqref{eq:integralboundd3} with respect to $v_1$ and $v_2$ which gives
\begin{equation}\label{eq:integralboundd4}
\begin{split}
&\frac{1}{(2\pi \mathrm{i})^2} \int_{\Gamma_{1-\varepsilon}}\frac{d u_1}{u_1} \int_{\Gamma_{1-\varepsilon}}\frac{d u_2}{u_2}\frac{u_1^{\frac{1-2R-\tilde{x}_1}{2}}}{\tilde{c}(u_1,u_2)}  
\sum_{\substack{x \in \bar{\mathtt{W}}_1 \\ x_1=1-2R}} K^{-1}(x,\tilde{y})  u_2^{\frac{x_2-\tilde{x}_2}{2}}
\end{split}
\end{equation}
where $\tilde{y}=(\tilde{y}_1,\tilde{y}_2)$.  For the integral with respect to $u_1$, we make the change of variables $u_1 \mapsto u_1^{-1}$. Notice that $\tilde{c}(u_1^{-1},u_2) =\tilde{c}(u_1,u_2)$ and that $\tilde{c}(u_1,u_2)$ contains no zeroes for $r<|u_1|<1/r$ for $0<r<1$ close to 1. We deform the contour of integration for the integral with respect to $u_1$ to $\Gamma_r$. We also split up the integral with respect to $u_2$ depending on whether  $x_2\leq \tilde{x}_2$ or $x_2 >\tilde{x}_2$ and in the latter case, deform the contour to $\Gamma_{1+\eps}$. Again, no additional contributions are picked up.  Under these steps,~\eqref{eq:integralboundd4} is equal to
\begin{equation}\label{eq:integralboundd5}
\begin{split}
&\frac{1}{(2\pi \mathrm{i})^2} \int_{\Gamma_{r}}\frac{d u_1}{u_1} \int_{\Gamma_{1-\varepsilon}}\frac{d u_2}{u_2} \frac{u_1^{\frac{2R-1+\tilde{x}_1}{2}}}{\tilde{c}(u_1,u_2)}  
\sum_{\substack{x \in \bar{\mathtt{W}}_1 \\ x_1=1-2R \\ x_2>\tilde{x}_2}} K^{-1}(x,\tilde{y})  u_2^{\frac{x_2-\tilde{x}_2}{2}}\\
&+\frac{1}{(2\pi \mathrm{i})^2} \int_{\Gamma_{r}}\frac{d u_1}{u_1} \int_{\Gamma_{1+\varepsilon}}\frac{d u_2}{u_2}\frac{u_1^{\frac{2R-1+\tilde{x}_1}{2}}}{\tilde{c}(u_1,u_2)}  
\sum_{\substack{x \in \bar{\mathtt{W}}_1 \\ x_1=1-2R\\x_2 \leq \tilde{x}_2}} K^{-1}(x,\tilde{y})  u_2^{\frac{x_2-\tilde{x}_2}{2}}.
\end{split}
\end{equation}
We now need the following claim which is proved after the conclusion of the proof of the lemma.
\begin{claim}\label{claim:bound}
For $x=(x_1,x_2) \in \bar{\mathtt{W}}_1$ with $x_1=1-2R$, $x_1=2R-3$, $x_2=2-2R$ or $x_2=2R-2$, and $\tilde{y}=(\tilde{y}_1,\tilde{y}_2) \in \bar{\mathtt{B}}$, we have
\begin{equation}
| K^{-1}(x,\tilde{y}) | \leq C
\end{equation}
for some $C>0$ constant. 
\end{claim}
We can now take absolute values of each of the terms in~\eqref{eq:integralboundd5}. Using the claim and that $\tilde{c}(u_1,u_2)>0$ for $u_1$ on $\Gamma_r$ and $u_2$ on $\Gamma_{1-\eps}$, we have that~\eqref{eq:integralboundd5} is bounded above by 
\begin{equation}\label{eq:integralboundd6}
\begin{split}
&C_1 \int_{\Gamma_{r}}\frac{d u_1}{u_1} \int_{\Gamma_{1-\varepsilon}}\frac{d u_2}{u_2} {|u_1|^{\frac{2R-1+\tilde{x}_1}{2}}}  
\sum_{\substack{x \in \bar{\mathtt{W}}_1 \\ x_1=1-2R \\ x_2>\tilde{x}_2}}  1 \\
&+C_2 \int_{\Gamma_{r}}\frac{d u_1}{u_1} \int_{\Gamma_{1+\varepsilon}}\frac{d u_2}{u_2}{|u_1|^{\frac{2R-1+\tilde{x}_1}{2}}}  
\sum_{\substack{x \in \bar{\mathtt{W}}_1 \\ x_1=1-2R\\x_2 \leq \tilde{x}_2}}   1,
\end{split} 
\end{equation}
where $C_1,C_2>0$ are constants. Since $r<1$ and that $\tilde{x}_1$ is fixed in terms of $R$, the above term is bounded by $C R e^{-c_0 R}$ as required.

To bound the rest of the terms in the integrals on the left side of the equation in Lemma~\ref{Claim2}, we apply the same procedure using either the variable $u_1$ or $u_2$ in bounding the integral, depending on which has a factor of $R$ or $-R$ in its exponent. Note that for integrals containing the term $u_1^{R}$ or $u_2^R$, we can immediately make the contour deformation to $\Gamma_r$ without the change of variables $u_1 \mapsto u_1^{-1}$. The analagous bound to the one given in Claim~\ref{claim:bound} also holds for $x \in \bar{\mathtt{W}}_0$; see Remark~\ref{remarkclaim}.   After applying these steps to all the terms in the integral on the left side of the equation in Lemma~\ref{Claim2}, we find that all terms are bounded by $CRe^{-c_0R}$ as required. Finally, we now give the proof of Claim~\ref{claim:bound} which completes the proof of the proposition. 
\end{proof}

Finally, we give the proof of Claim~\ref{claim:bound}.
\begin{proof}[Proof of Claim~\ref{claim:bound}]	
	For the purpose of this proof, denote $K_{\mathbb{L}_R^{\mathtt{w}}}$ (resp. $K_{\mathbb{L}_R^{\mathtt{f}}}$) and $K_{\mathbb{L}_R^{\mathtt{w}}}^{-1}$ (resp $K_{\mathbb{L}_R^{\mathtt{f}}}^{-1}$) to be the Kasteleyn and inverse Kasteleyn matrices on $\mathbb{L}_R^{\mathtt{w}}$ resp ($\mathbb{L}_R^{\mathtt{f}}$). We also let $C,C_1,C_2>0$ be arbitrary constants throughout the proof and also denote $\mathbb{L}_R^\star \backslash \{ x, \tilde{y} \}$ to be the graph $\mathbb{L}_R^\star$ with the vertices $x$ and $\tilde{y}$ removed from $\mathbb{L}_R^\star$ along with their incident edges, where $\star$ is either $\mathtt{w}$ or $\mathtt{f}$.  

	From the gauge transformation given in  Proposition~\ref{prop:gaugeequivalent}, we have
\begin{equation} \label{claimeq1}
	K^{-1}(x, \tilde{y}) = a^{\frac{1}{2}(x_2+2R-2)+1}a^{-\frac{1}{2}(\tilde{y}_2+2R-1)} K_{\mathbb{L}_R^{\mathtt{w}}}^{-1}(x, \tilde{y}) =a^{\frac{1}{2}(x_2-\tilde{y}_2+1)}K_{\mathbb{L}_R^{\mathtt{w}}}^{-1}(x, \tilde{y}).
\end{equation}
	The entry $K_{\mathbb{L}_R^{\mathtt{w}}}^{-1}(x,\tilde{y})$ when $x$ and $\tilde{y}$ are not on the same face, up to an overall sign, is a signed weighted count of dimer coverings on $\mathbb{L}_R^{\mathtt{w}}\backslash \{x,\tilde{y}\}$ divided by the partition function. The sign in the signed weighted count is from the fact that the original Kasteleyn orientation on $\mathbb{L}_R^{\mathtt{w}}$ is no longer a valid Kasteleyn orientation on $\mathbb{L}_R^{\mathtt{w}}\backslash \{x ,\tilde{y}\}$.   Nevertheless, this signed weighted count is bounded above by the partition function on $\mathbb{L}_R^{\mathtt{w}}\backslash \{x,\tilde{y}\}$.  Therefore, we have
\begin{equation}\label{claimeq2}
|K_{\mathbb{L}_R^{\mathtt{w}}}^{-1}(x, \tilde{y})| \leq \frac{Z_{\mathbb{L}_R^{\mathtt{w}}\backslash \{x,\tilde{y}\}}}{Z_{\mathbb{L}_R^{\mathtt{w}}}}
\end{equation}
where $Z_G$ denotes the partition function of the dimer model on the graph $G$.
	From the correspondence detailed in Section~\ref{subsec:spanning}, the dimer model on $\mathbb{L}_R^{\mathtt{w}}$ is equivalent to  directed random spanning tree on $\mathbb{T}_R^{\mathtt{w},p}$ with $x$ being a vertex on the graph of the primal tree.  For this directed (primal) spanning tree, there is no directed edge passing through the vertex $\tilde{y}$, there is an incoming edge into $x$ but no outgoing edge from $x$.  Each of these is a restriction of the total number of weighted spanning tree configurations (up to a constant) and we conclude that 
\begin{equation}
{Z_{\mathbb{L}_R^{\mathtt{w}}\backslash \{x,\tilde{y}\}}} \leq C_1{Z_{\mathbb{L}_R^{\mathtt{w}}}}. 
\end{equation}
Using the above equation and~\eqref{claimeq1} and~\eqref{claimeq2}, we find that 
\begin{equation}
K^{-1}(x, \tilde{y}) \leq C_1 a^{\frac{1}{2}(x_2-\tilde{y}_2-1)}.
\end{equation}

Proceeding as above which gave equations~\eqref{claimeq1} and~\eqref{claimeq2}, but instead using the gauge transformation between $\mathbb{L}_R$ and $\mathbb{L}_R^{\mathtt{f}}$ in Proposition~\ref{prop:gaugeequivalent}, we also have  
\begin{equation}\label{claimeq3}
K^{-1}(x, \tilde{y})=a^{\frac{1}{2}(\tilde{y}_2-x_2-1)} K_{\mathbb{L}_R^{\mathtt{f}}}^{-1} (x, \tilde{y})
\end{equation}
and 
\begin{equation}
|K_{\mathbb{L}_R^{\mathtt{f}}}^{-1} (x, \tilde{y})| \leq 
\frac{Z_{\mathbb{L}_R^{\mathtt{f}}\backslash \{x,\tilde{y}\}}}{Z_{\mathbb{L}_R^{\mathtt{f}}}}.
\end{equation}
	From the correspondence detailed in Section~\ref{subsec:spanning}, the dimer model on $\mathbb{L}_R^{\mathtt{f}}$ is equivalent to  directed random spanning tree on $\mathbb{T}_R^{\mathtt{f},p}$, but this time, $x$ is a vertex on the graph of the dual tree.  To put the restriction onto the primal tree, for simplicity we suppose that $x_1=1-2R$ (the other cases follow from a similar argument). For this choice of $x_1$, we have	
\begin{equation}
Z_{\mathbb{L}_R^{\mathtt{f}}\backslash \{x,\tilde{y}\}} = Z_{\mathbb{L}_R^{\mathtt{f}}\backslash \{x,\tilde{y}, x-e_1+e_2,x-e_1\}}+a^2Z_{\mathbb{L}_R^{\mathtt{f}}\backslash \{x,\tilde{y}, x-e_1+e_2,x+e_2\}}
\end{equation}
	which follows from just partitioning over dimers incident to $x-e_1+e_2$.   This split has removed the restriction on the dual tree. Each of the terms $Z_{\mathbb{L}_R^{\mathtt{f}}\backslash \{x,\tilde{y}, x-e_1+e_2,x-e_1\}}$ and $Z_{\mathbb{L}_R^{\mathtt{f}}\backslash \{x,\tilde{y}, x-e_1+e_2,x+e_2\}}$ can be bounded by $\frac{1}{2}C_2{Z_{\mathbb{L}_R^{\mathtt{f}}}}$ because each  of their directed spanning tree configurations are contained within the directed spanning tree on $\mathbb{T}_R^{\mathtt{f},p}$ for $C_2$ large enough. This and the two equations above give
\begin{equation}\label{claimeq4}
|K^{-1}(x, \tilde{y})|=C_2a^{\frac{1}{2}(\tilde{y}_2-x_2+1)}.
\end{equation}
Since both~\eqref{claimeq3} and~\eqref{claimeq4} hold, for large enough $C_1$ and $C_2$ we obtain the claim. 
 
\end{proof}

\begin{remark}\label{remarkclaim}
An analagous bound to Claim~\ref{claim:bound} holds for $x=(x_1,x_2) \in \bar{\mathtt{W}}_0$ with $x_1=-1-2R$, $x_1=2R-1$, $x_2=-2R$ or $x_2=2R$. The same proof holds, albeit with a simplification as now the removed vertices are on the boundary $\mathbb{L}_R$. We omit  this computation as it contains no additional technical information. 
\end{remark}

\subsection{Proof of Theorem~\ref{thm:onecorridor}} 
 Before proving Theorem~\ref{thm:onecorridor}, we need the following lemma.

\begin{lemma}\label{lem:oneended}
The directed spanning forests on  $\mathbb{T}^{\mathtt{w},p}$ and $\mathbb{T}^{\mathtt{f},p}$ are single trees almost surely.
\end{lemma}
\begin{remark}
	A similar result for a more general construction was proved in~\cite{Sun:16}, however that approach requires embedding spanning forests (cycle-rooted spanning forests) on the torus and taking the toroidal exhaustion. Our approach bypasses this. 
\end{remark}

\begin{proof}[Proof of Lemma~\ref{lem:oneended}]
Proposition~\ref{prop:fintetoinfinite} gives that, as $R \to \infty$, the dimer model on $\mathbb{L}_R$ converges weakly to the full-plane smooth phase.  Moreover, Proposition~\ref{prop:gaugeequivalent} shows that the full-plane smooth phase is equivalent to the directed spanning forest on both $\mathbb{T}^{\mathtt{w},p}$ and $\mathbb{T}^{\mathtt{f},p}$.  All edge probabilities for directed spanning forests on $\mathbb{T}^{\mathtt{w},p}$ and $\mathbb{T}^{\mathtt{f},p}$ can then be computed explicitly using the local statistics formula given in Theorem~\ref{localstatisticsthm} with the correlation kernel given in~\eqref{smoothphaseeqn}. Moreover, by symmetry of the full-plane smooth phase inverse Kasteleyn matrix, probabilities of all cylinder events of directed spanning forests on $\mathbb{T}^{\mathtt{w},p}$ are equivalent to those on $\mathbb{T}^{\mathtt{f},p}$ after rotating the configurations by $\pi$. This can be seen by rotating the dimer configuration on $\mathbb{T}^{\mathtt{w},p}$, followed by shifting  the configuration by the vector $e_1+e_2$ and computing the new cylinder events there.  From Proposition~\ref{prop:tree}, the directed spanning forest on $\mathbb{T}^{\mathtt{w},p}$ is a single tree almost surely and due to the equivalence, the directed spanning forest on $\mathbb{T}^{\mathtt{f},p}$ is also a single tree.  
\end{proof}

\begin{proof}[Proof of Theorem~\ref{thm:onecorridor}]

Recall that the $a$-dimers correspond to directed edges on both directed spanning trees on  $\mathbb{T}^{\mathtt{w},p}$ and $\mathbb{T}^{\mathtt{f},p}$. That is, for $w \in \mathtt{W}_0^*$, the $a$-dimers $(w,w+e_i)$ correspond to the directed edges $(w,w+2e_i)$ for $i \in \{1,2\}$ which is on the directed spanning tree $\mathbb{T}^{\mathtt{f},p}$. Conversely, for $w \in \mathtt{W}_1^*$, $a$-dimers $(w,w-e_i)$ correspond to the directed edges $(w,w-2e_i)$ for $i \in \{1,2\}$ which is on the directed spanning tree $\mathbb{T}^{\mathtt{w},p}$.

	Suppose that there is a biinfinite path in the smooth phase.  The biinfinite path in the smooth phase cannot be supported on only one tree as this contradicts Lemma~\ref{lem:oneended}, that is, the $a$-dimers on the biinfinite path belong to both $\mathbb{T}^{\mathtt{w},p}$ and $\mathbb{T}^{\mathtt{f},p}$.  Consider a dimer, $d_1$, incident to $\mathtt{W}_0^*$ (not necessarily an $a$-dimer)  and take the same type of dimer, $d_2$, on the other side of the biinfinite path, that is, there is a sequence of adjacent faces between $d_1$ and $d_2$ that crosses the biinfinite path an odd number of times.   Thanks to the bijection between trees, dimers and heights (up to height level) two dimers of a tree (of the same type) can only join the same branch if they separate the same height after winding~\cite{KPW:00}.  Since these two dimers separate different heights, then the branch passing through $d_1$  must unwind before joining the branch passing through $d_2$ (or vice versa).  However, this is true for all pairs of vertices on either side of the biinfinite path,  which is only possible if there are more than one tree for $\mathbb{T}^{\mathtt{w},p}$ and  $\mathbb{T}^{\mathtt{f},p}$, which is a contradiction.

\end{proof}

\section{Peierls argument for loops and double edges} \label{sec:Peierls}

In this section, we first give the proof of Lemma~\ref{lem:loopsalmostsure} which is based on Peierls argument.  It turns out that the same argument can be applied for double edges, which holds for all $a \in (0,1)$. This statement and proof is also given below.

\begin{proof}[Proof of Lemma~\ref{lem:loopsalmostsure}]
	We give the result for $\mathbb{P}_{\mathrm{Az}}$ and then explain the difference for $\mathbb{P}_{\mathrm{sm}}$.
	Let $\gamma$ be a loop in $D_m$ and let
$Z_{D_m\backslash \gamma}$ be the partition function for the dimer coverings on $D_m \backslash \gamma$. Then,
\begin{equation}\label{peierlsestimate}
	Z_{D_m} \geq  Z_{D_m \backslash \gamma} (1+a^{\ell(\gamma)})
\end{equation}
	where the coefficient of $Z_{D_m \backslash \gamma}$ comes from rotating the $a$-dimers along the loop $\gamma$ so that they are now $b$-dimers and noting that the product of edge weights when all dimers on $\gamma$ are $a$-edges is $a^{\ell(\gamma)}$ while when they are all $b$-dimers, the product of edge weights is equal to 1. From this, we have
	\begin{equation}\label{eq:Aztecbound}
\mathbb{P}_{\mathrm{Az}}[\mbox{All $a$ edges along $\gamma\in \mathcal{D}_l$}]\leq  a^{l(\gamma)}.
\end{equation}
Then, letting $v$ be a face in $S$, we obtain
	\begin{equation}
\begin{split}
&\mathbb{P}_{\mathrm{Az}}[\mbox{$\exists \gamma \in \mathcal{D}_l$ containing $v, l(\gamma) \geq d$, $a$ edges along $\gamma$}]\\
& \leq  \sum_{\substack{ \gamma \ni v}}a^{l(\gamma) } \leq  \sum_{k=d}^{\infty}(3a)^k=\frac{(3a)^{d}}{1-3a}
\end{split}
\end{equation}
	provided that $a<1/3$. The first inequality above comes from a counting argument: when tracing over the edges of the loop, there are two choices for $b$-edges at the endpoint of each $a$-edge at a $b$-face.  For one of these choices, the next $a$-edge in the sequence is determined, while the other choice has two choices for the next $a$-edge in  the sequence, which means three choices in total.  We now take a union bound over all faces $v$ in $S$ which gives the result for $\mathbb{P}_{\mathrm{Az}}$.

	The same argument holds for the full-plane smooth phase provided we show the analog of~\eqref{eq:Aztecbound} for the smooth phase, that is showing  
	\begin{equation}\label{eq:smoothbound}
\mathbb{P}_{\mathrm{sm}}[\mbox{All $a$ edges along $\gamma\in \mathcal{D}_l$}]\leq  a^{l(\gamma)}.
\end{equation}
	However, the above equation immediately follows since the smooth phase is a Gibbs measure~\cite{KOS:06}.
\end{proof}

Next we show that the same argument given in the proof of Lemma~\ref{lem:loopsalmostsure} holds for double edges.   For a dimer covering, let $\mathcal{D}_e$ be the set of all sequences of distinct edges $\gamma=(e_1,\dots, e_{2k})$  such that the following properties hold
\begin{enumerate}
\item $e_i$ shares endpoints with $e_{i-1}$ and $e_{i+1}$ for all $0 \leq i \leq 2k$ with $e_0=e_{2k}$ and $e_{2k+1}=e_1$;
\item $e_{2i+1}$ are $a$-edges while $e_{2i+2}$ are $b$-edges for all $0 \leq i \leq k-1$;
\item the pairs $(e_{2i+1},e_{2k-(2i+1)})$ form double edges after the squishing procedure for all $0 \leq i \leq k-1$; 
\item $\gamma$ is not incident to any other double edges.
\end{enumerate}
For $\gamma \in \mathcal{D}_{e}$, let $\ell_e(\gamma)$ be the number of $a$-dimers  in $\gamma$. 
\begin{lemma} \label{lem:doubleedges}
	Let $S$ be a set of $a$-edges in $D_m$ or in the full-plane. Then, for all $a \in (0,1)$, 	
$$
	\mathbb{P}[\exists \gamma \in \mathcal{D}_e\mbox{ that intersects $S$ and has length $\ell_e(\gamma)$ at least } d] \leq \frac{2|S|}{1-a} a^d
$$
	where $|S|$ is the size of $S$, and $\mathbb{P}$ is either $\mathbb{P}_{\mathrm{Az}}$ or $\mathbb{P}_{\mathrm{sm}}$.  
	
\end{lemma}

\begin{proof}
We give the proof of $D_m$ and the proof for the full-plane is analagous by arguing the same way for~\eqref{eq:smoothbound}.   Let $\gamma \in \mathcal{D}_e\cap \Lambda_L$  and  let $Z_{D_m\backslash \gamma}$ be the number of dimer coverings on $D_m \backslash \gamma$, where $D_m\backslash \gamma$ is the graph $D_m$ removing $\gamma$ and any incident edges to $\gamma$. If $\gamma= \emptyset$, then $Z_{D_m \backslash \emptyset}=Z_{D_m}$, which is the number of dimer coverings on $D_m$.  We can partition the set of dimer coverings on $D_m$ into the set of dimer coverings which is also a dimer covering of the smaller graph $D_m \backslash \gamma$ (with a dimer covering on $\gamma$) and those where there is no dimer covering on the smaller graph $D_m \backslash \gamma$.  This gives
\begin{equation}
	Z_{D_m} \geq Z_{D_m \backslash \gamma} \prod_{i=0}^{\ell_e(\gamma)/2}(1+a^{2})
\end{equation}
where the coefficient of $Z_{D_m \backslash \gamma}$ is 
	due to each double edge could be replaced  $b$ edges instead. This gives $Z_{D_m} \geq Z_{D_m \backslash \gamma}$ which means that 
\begin{equation}
	\mathbb{P}_{\mathrm{Az}}[\mbox{All double edges along $\gamma\in \mathcal{D}_e$}]=\frac{ a^{\ell_e(\gamma)}Z_{D_m \backslash \gamma} }{Z_{D_m}}\leq  a^{\ell_e(\gamma)}.
\end{equation}
Let $v$ be a vertex in $S$. Then, 
\begin{equation}
\begin{split}
&\mathbb{P}_{\mathrm{Az}}[\mbox{$\exists \gamma \in \mathcal{D}_e$ containing $v, \ell_e(\gamma) \geq d $, double edges along $\gamma$}] \\
&\leq  \sum_{\substack{ \gamma \ni v}}a^{\ell_e(\gamma) } =2 \sum_{k=d}^{\infty}a^k=\frac{2 a^{d}}{1-a}.
\end{split}
\end{equation}
for $a \in (0,1)$.
The factor $2$ above is due to there being only two choices for the direction of $\gamma$ and once that choice is made, there are no further choices.

\end{proof}

\begin{appendix}
	\section{Proof of Lemma~\ref{anothersmoothbound}}  \label{appendix}
	Here, we bring forward a result from~\cite{CJ:16} which allows us to prove Lemma~\ref{anothersmoothbound}.
	Introduce 
\begin{equation}\label{Ekl}
\mathtt{E}_{k,\ell}=\frac 1{(2\pi\mathrm{i})^2}\int_{\Gamma_1}\frac{du_1}{u_1}\int_{\Gamma_1}\frac{du_2}{u_2}\frac{u_1^{\ell}u_2^{k}}{\tilde{c}(u_1,u_2)}.
\end{equation}
	Then, from~\eqref{smoothphaseeqn} we have 
\begin{equation}\label{GasE}
\mathbb{K}^{-1}_{1,1}(x,y)=-\mathrm{i}^{1+h(\eps_1,\eps_2)}\left(a^{\eps_2}\mathtt{E}_{k_1,\ell_1}+a^{1-\eps_2}\mathtt{E}_{k_2,\ell_2}\right),
\end{equation}
where 
\begin{equation}\label{kldef}
	\begin{split}
		&k_1=\frac{x_2-y_2-1}{2}+h(\varepsilon_1,\varepsilon_2)\,,\,\ell_1=\frac{y_1-x_1-1}{2}\\
		&k_2=\frac{x_2-y_2+1}{2}-h(\varepsilon_1,\varepsilon_2)\,,\,\ell_2=\frac{y_1-x_1+1}{2}.
\end{split}
\end{equation}

	The following is given in~\cite[Lemma 4.7]{CJ:16}.
\begin{lemma}\label{lem:Eformula}
Let $A_m$, $B_m$, $m\ge 1$, be given and set $b_m=\max(|A_m|,|B_m|)$, and
\begin{equation}
a_m=\begin{cases} 
A_m   &\quad \text{if } b_m=|B_m| \\
B_m   &\quad \text{if } b_m=|A_m|. \\
\end{cases}
\end{equation}
Assume that $b_m >0$, $m\geq 2$.  There exists constants $C,d_1,d_2 >0$ so that
\begin{equation} \label{Easympeq2}
	|\mathtt{E}_{B_m+A_m,B_m-A_m} | \leq \frac{C}{\sqrt{b_m}}\mathcal{C}^{2 b_m} \left( e^{-d_1 \frac{a_m^2}{b_m}}+e^{-d_2 b_m} \right)
\end{equation}
	for all $m\geq 2$ and $\mathcal{C}$ is defined in~\eqref{G}.

\end{lemma}
	We now prove Lemma~\ref{anothersmoothbound}.
	\begin{proof}
		We set $B_m=(k_i+l_i)/2$ and $A_m=(k_i-l_i)/2$ with $1 \leq i \leq 2$ in Lemma~\ref{lem:Eformula} and notice that $b_m=|A_m|$ while $a_m=B_m$ for the conditions given in Lemma~\ref{anothersmoothbound}. From the restriction of $R$, we have that $L<\lambda_1 (\log m)^2$. This restriction on $L$ means that $\Lambda^1$ and $\Lambda^2$ do not overlap and are separated by a distance of at least $\lambda_1(2-\sqrt{2}) (\log m)^2$.  This means that  the smallest $b_m$ happens when $\Lambda^1$ and $\Lambda^2$ are closest, that is, $k_2=k_1+1$ in the definition of $\Lambda^1$ and $\Lambda^2$, and so $b_m$ is at least equal to $\lambda_1(2-\sqrt{2}) (\log m)^2$.  	We apply Lemma~\ref{lem:Eformula} and using that $b_m$ is at least of order $(\log m)^2$, we find that 
	$$
		|\mathtt{E}_{B_m+A_m,B_m-A_m} | \leq  \frac{2 C}{\log m } \mathcal{C}^{c_1 (\log m)^2} =\frac{2C}{\log m} e^{-c_0 (\log m)^2} 
		$$
		with $C,c_0,c_1>0$, since $\mathcal{C}<1$. 
	\end{proof}

	\section{Proof of Lemma~\ref{lem:subtlecomps}}  \label{Appendix:proofsubtlecomps}

	We only give the proof of the first equation in the lemma.  The proof of the second equation is analagous, but requires considering the particle process transversally (there is no additional technical complications here, just more notation).  The outline of the proof is to introduce particle process given~\cite{BCJ:16} and then use the determinantal structure to perform a cumulant expansion.  We can then use results from~\cite{BCJ:16}.  
	
	For $\eps\in\{0,1\}$, introduce	
	\begin{equation}
	\mathcal{L}_m^{\eps}(q,k) = \{(2t-\eps+\frac{1}{2} ) e_1 -\beta_m(q,k)e_2 ; t \in [0,4m]\cap \mathbb{Z} \}.
\end{equation}
and 
	\begin{equation}
\mathcal{L}_m^{\eps}= \bigcup_{q=1}^{L_1} \bigcup_{k=1}^M \mathcal{L}_m^{\eps}(q,k).
\end{equation}
	Then, $\mathcal{L}_m=\mathcal{L}_m^0 \cup \mathcal{L}_m^1$ defines discrete intervals on the Aztec diamond.   For $z \in \mathcal{L}_m$, write $\eps(z)=\eps $ if $z \in \mathcal{L}_m^{\eps}$ where $\eps \in \{0,1\}$.  For $z \in \mathcal{L}_m$, and since each $z$ is incident to an $a$-face, we let
\begin{equation}
\begin{split} \label{xyz}
	x(z)&= z-\frac{1}{2} (-1)^{\eps(z)} e_2 \in \mathtt{W}_{\eps}  \\
y(z)&= z+\frac{1}{2} (-1)^{\eps(z)} e_2 \in \mathtt{B}_{\eps}. \\
\end{split}
\end{equation}
for $\eps \in \{0,1\}$ which gives a relation between particles and dimers.  The determinantal point process on $\mathcal{L}_m$ is given by 
\begin{equation} \label{eq:Ksplit2}
	\tilde{\mathcal{K}}_m(z,z')  = a \mathrm{i} K_{a,1}^{-1}(x(z'),y(z))=\tilde{\mathcal{K}}_{m,0}(z,z') +\tilde{\mathcal{K}}_{m,1}(z,z') 
\end{equation}
where the second equality is due to~\eqref{eq:Ksplit}.
Let $I_{p,q,k}$ be the interval in $\mathcal{L}_m(q,k)$ between $J_{p,q,k,1}^r$ and $J_{p,q,k,k}^r$. Then, let 
\begin{equation}
	\mathbb{I}_{p,q,k}(z) = \left\{ \begin{array}{ll}
		1 & z \in I_{p,q,k} \\
	0 & \mbox{otherwise.} \end{array} \right. 
\end{equation}
We have that 
\begin{equation}
	h^a(J_{p,q,k,1}^r)-h^a(J_{p,q,k,k}^r) = \sum_i (-1)^{\eps(z_i)} \mathbb{I}_{p,q,k}(z_i) 
\end{equation}
where $\sum_i$ is the sum of all particles in the determinantal point process on $\mathcal{L}_m$.  Let 
\begin{equation}
	\psi (z) = \sum_{k=2}^M \sum_{p=1}^{L_2} \sum_{q=1}^{L_1} w_{p,q} (-1)^{\eps(z)} \mathbb{I}_{p,q,k}(z). 
\end{equation}
Then, we have the following
\begin{equation}
	\begin{split}
		&	\mathbb{E}_{\mathrm{Az}} \bigg[ \exp \bigg[ \frac{1}{M} \sum_{p=1}^{L_2} \sum_{q=1}^{L_1} \sum_{k=2}^M w_{p,q} ( h^a(J_{p,q,k,1}^r)-h^a(J_{p,q,k,k}^r))\bigg] \bigg]= \mathbb{E}\big[e^{\frac{1}{M} \sum_i \psi(z_i)}\big] \\
		 &=\mathbb{E}\bigg[\prod_ie^{\frac{1}{M}  \psi(z_i)}\bigg]
		= \det ( \mathbb{I} +(e^{\frac{1}{M} \psi} -1)\mathcal{K}_m)
	\end{split} 
\end{equation}
where the expectations in the second and third equality are over the determinantal point process on $\mathcal{L}_m$.  We can now take a cumulant expansion by taking logarithms of both sides, which gives 
\begin{equation} \label{traceexp}
	\log \det ( \mathbb{I} +(e^{\psi/M} -1)\mathcal{K}_m) = \sum_{s=1}^{\infty} \frac{1}{M^s} \sum_{r=1}^s \frac{ (-1)^{r+1}}{r} \sum_{\substack{ \ell_1+ \dots +\ell_r =s \\  \ell_1, \dots ,\ell_r \geq  s }} \mathrm{tr} \big[ \psi^{\ell_1} \mathcal{K}_m  \dots \psi^{\ell_r} \mathcal{K}_m  \big];
\end{equation}
see for example p450 in~\cite{BD:14}.  Below, we use the notation $[N]=\{1,\dots, N\}$, $\overline{p}=(p_1,\dots,p_r)$ and  $z_{r+1}=z_1$. Expanding the above trace gives 
\begin{equation}
	\begin{split} \label{trace2}
		&\mathrm{tr} \big[ \psi^{\ell_1} \mathcal{K}_m  \dots \psi^{\ell_r} \mathcal{K}_m  \big]= \sum_{\overline{z} \in \mathcal{L}_m^r} \sum_{\overline{k} \in [M]^r} \sum_{\overline{p} \in [L_2]^r}\sum_{\overline{q} \in [L_1]^r} \prod_{i=1}^r\\ & \,\times  \mathbb{I}_{p_i,q_i,k_i}(z_i) (-1)^{\ell_i\eps( z_i)} w_{p_i,q_i}^{r_i} \mathcal{K}_m(z_i,z_{i+1})\\
		&=\sum_{\overline{\delta} \in \{0,1\}} \sum_{\overline{z} \in \mathcal{L}_m^r} \sum_{\overline{k} \in [M]^r} \sum_{\overline{p} \in [L_2]^r}\sum_{\overline{q} \in [L_1]^r} \prod_{i=1}^r\\ & \,\times  \mathbb{I}_{p_i,q_i,k_i}(z_i) (-1)^{\ell_i \eps(z_i)} w_{p_i,q_i}^{r_i}  \mathcal{K}_{m,\delta_i}(z_i,z_{i+1})
		\end{split}
\end{equation}
by~\eqref{eq:Ksplit2}.  Following our previous approach in~\cite[Section 4]{BCJ:16} we split this trace into four parts. Let 
\begin{equation}
D_r=\{0,1\}^r \times [M]^r \times [L_2]^r \times [L_1]^r.
\end{equation}
Define
\begin{equation}
	D_{r,0} = \{(\overline{\delta},\overline{k},\overline{p},\overline{q}) \in D_r; \delta_i =0,k_i=k_{i+1},p_i=p_{i+1} \mbox{~and~} q_i=q_{i+1},1\leq i \leq r\},
\end{equation}
\begin{equation}
D_{r,1} = \{ (\overline{\delta},\overline{k},\overline{p},\overline{q}) \in D_r; \delta_i =0, q_i=q_{i+1} \mbox{~for~}1\leq i \leq r \mbox{~and~}p_i\not=p_{i+1} \mbox{~for some~} i \},  
\end{equation}
\begin{equation}
	\begin{split}
		D_{r,2} = \{ &(\overline{\delta},\overline{k},\overline{p},\overline{q}) \in D_r; \delta_i =0,q_i=q_{i+1},p_i=p_{i+1}\mbox{~for~}1\leq i\leq r \\ & \mbox{~and~} k_i\not=k_{i+1} \mbox{~for some~}i\},
	\end{split}
\end{equation}
and
\begin{equation}
D_{r,3} = \{ (\overline{\delta},\overline{k},\overline{p},\overline{q}) \in D_r; \delta_i =1 \mbox{~or~} q_i \not = q_{i+1} \mbox{~for some~} i \}.
\end{equation}
Then, we have $D_r=D_{r,0} \cup D_{r,1} \cup D_{r,2} \cup D_{r,3}$. Introduce
\begin{equation}  \label{Tjmrell}
	\begin{split}
		T_j(m,r,\overline{l}) &=  \sum_{\overline{z} \in( \mathcal{L}_m)^r } \sum_{(\overline{\delta},\overline{k},\overline{p},\overline{q}) \in D_{r,j}} \prod_{i=1}^r(-1)^{\ell_i \eps(z_i) }  w_{p_i,q_i}^{\ell_i}   \mathbbm{I}_{p_i,q_i,k_i} (z_i)
  \mathcal{K}_{m,\delta_i}(z_i,z_{i+1}),
	\end{split}
\end{equation}
for  $0\leq j \leq 3$.  Then, by~\eqref{traceexp} and~\eqref{trace2} we have
\begin{equation}
\log \det( \mathbbm{I} +(e^{\frac{1}{M} \psi}-1) \mathcal{K}_m) = \sum_{j=0}^3 U_j(m)
\end{equation}
where we define
\begin{equation} \label{Ujm}
U_j(m)=  \sum_{s=1}^\infty \frac{1}{M^s} \sum_{r=1}^s \frac{(-1)^{r+1}}{r} \sum_{\substack{\ell_1+\dots+\ell_r=s \\ \ell_1,\dots,\ell_r \geq 1}} \frac{T_j(m,r,\overline{\ell})}{\ell_1! \dots \ell_r!}.
\end{equation}
From Lemmas 4.1 and 4.2 in~\cite{BCJ:16}, we have that $U_0(m), U_1(m), U_2(m)$ tend to 0 uniformly as $m\to \infty$ for $|w_{p,q}| \leq R$.   We can trivially bound $T_3$ by using Lemma~\ref{lem:Eformula} and Theorem~\ref{Airyasymptotics} since the sum over each $z_i$ in~\eqref{Tjmrell} is over atmost $M(\log m)^2$ terms which gives the bound
\begin{equation}
	| T_3(m,r,\overline{\ell})| \leq \sum_{(\overline{\delta},\overline{k},\overline{p},\overline{q}) \in D_{r,3}}  \frac{C_1^r}{m^{r/3}} M^{r} (\log m)^{2r} \leq  \frac{C^r}{m^{r/3}} M^{2r} (\log m)^{2r} 
\end{equation}
where $C,C_1>0$ are  constants. This gives that 
\begin{equation}
	| U_3(m) | \leq \sum_{s=1}^\infty \frac{R^s}{M^s} \sum_{r=1}^s \frac{1}{r} \sum_{\substack{ \ell_1+\dots+\ell_r=s \\ \ell_1,\dots,\ell_r\geq s}} \frac{C^r  (\log m)^{2r}M^{2r}}{m^{r/3}}  
\end{equation}
which tends to zero as $m \to \infty$ for $R$ sufficiently small.  This concludes the proof of the first equation in the lemma.

\end{appendix}

\bibliographystyle{plain}

\bibliography{ref}

\begin{thebibliography}{10}

\bibitem{Agg:19}
Amol Aggarawal.
\newblock Limit shapes and local statistics for the stochastic six-vertex
  model.
\newblock arXiv:1902.10867, 2019.

\bibitem{ADPZ:20}
Kari Astala, Erik Duse, Istv\'{a}n Prause, and Xiao Zhong.
\newblock Dimer models and conformal structures.
\newblock arXiv:2004.02599, 2020.

\bibitem{BCJ:16}
Vincent Beffara, Sunil Chhita, and Kurt Johansson.
\newblock Airy point process at the liquid-gas boundary.
\newblock {\em Ann. Probab.}, 46(5):2973--3013, 2018.

\bibitem{BLPS:01}
Itai Benjamini, Russell Lyons, Yuval Peres, and Oded Schramm.
\newblock Uniform spanning forests.
\newblock {\em Ann. Probab.}, 29(1):1--65, 2001.

\bibitem{Be:19}
Tomas Berggren.
\newblock Domino tilings of the {A}ztec diamond with doubly periodic
  weightings.
\newblock arXiv:1911.01250, 2019.

\bibitem{BD:19}
Tomas Berggren and Maurice Duits.
\newblock Correlation functions for determinantal processes defined by infinite
  block {T}oeplitz minors.
\newblock arXiv:1901.10877, 2019.

\bibitem{BD:14}
Jonathan Breuer and Maurice Duits.
\newblock The {N}evai condition and a local law of large numbers for orthogonal
  polynomial ensembles.
\newblock {\em Adv. Math.}, 265:441--484, 2014.

\bibitem{CDKL:19}
Christophe Charlier, Maurice Duits, Arno Kuijlaars, and Jonatan Lenells.
\newblock A periodic hexagon tiling model and non-{H}ermitian orthogonal
  polynomials.
\newblock arXiv:1901.02460, 2019.

\bibitem{CJ:16}
Sunil Chhita and Kurt Johansson.
\newblock Domino statistics of the two-periodic {A}ztec diamond.
\newblock {\em Adv. Math.}, 294:37--149, 2016.

\bibitem{CY:13}
Sunil Chhita and Benjamin Young.
\newblock Coupling functions for domino tilings of {A}ztec diamonds.
\newblock {\em Adv. Math.}, 259:173--251, 2014.

\bibitem{CKP:01}
Henry Cohn, Richard Kenyon, and James Propp.
\newblock A variational principle for domino tilings.
\newblock {\em J. Amer. Math. Soc.}, 14(2):297--346 (electronic), 2001.

\bibitem{CS:16}
F.~Colomo and A.~Sportiello.
\newblock Arctic curves of the six-vertex model on generic domains: the tangent
  method.
\newblock {\em J. Stat. Phys.}, 164(6):1488--1523, 2016.

\bibitem{DK:17}
Maurice Duits and Arno Kuijlaars.
\newblock The two periodic {A}ztec diamond and matrix valued orthogonal
  polynomials.
\newblock {\em JEMS}, 2017.
\newblock To Appear.

\bibitem{Joh:03}
Kurt Johansson.
\newblock Discrete polynuclear growth and determinantal processes.
\newblock {\em Comm. Math. Phys.}, 242(1-2):277--329, 2003.

\bibitem{Joh17}
Kurt Johansson.
\newblock Edge fluctuations of limit shapes.
\newblock In {\em Current developments in mathematics 2016}, pages 47--110.
  Int. Press, Somerville, MA, 2018.

\bibitem{Ken:97}
Richard Kenyon.
\newblock Local statistics of lattice dimers.
\newblock {\em Ann. Inst. H. Poincar\'e Probab. Statist.}, 33(5):591--618,
  1997.

\bibitem{Ken:09}
Richard Kenyon.
\newblock Lectures on dimers.
\newblock In {\em Statistical mechanics}, volume~16 of {\em IAS/Park City Math.
  Ser.}, pages 191--230. Amer. Math. Soc., Providence, RI, 2009.

\bibitem{KO:07}
Richard Kenyon and Andrei Okounkov.
\newblock Limit shapes and the complex {B}urgers equation.
\newblock {\em Acta Math.}, 199(2):263--302, 2007.

\bibitem{KOS:06}
Richard Kenyon, Andrei Okounkov, and Scott Sheffield.
\newblock Dimers and amoebae.
\newblock {\em Ann. of Math. (2)}, 163(3):1019--1056, 2006.

\bibitem{KPW:00}
Richard~W. Kenyon, James~G. Propp, and David~B. Wilson.
\newblock Trees and matchings.
\newblock {\em Electron. J. Combin.}, 7:Research Paper 25, 34 pp. (electronic),
  2000.

\bibitem{LP:16}
Russell Lyons and Yuval Peres.
\newblock {\em Probability on trees and networks}, volume~42 of {\em Cambridge
  Series in Statistical and Probabilistic Mathematics}.
\newblock Cambridge University Press, New York, 2016.

\bibitem{LPS:03}
Russell Lyons, Yuval Peres, and Oded Schramm.
\newblock Markov chain intersections and the loop-erased walk.
\newblock {\em Ann. Inst. H. Poincar\'e Probab. Statist.}, 39(5):779--791,
  2003.

\bibitem{Pem:91}
Robin Pemantle.
\newblock Choosing a spanning tree for the integer lattice uniformly.
\newblock {\em Ann. Probab.}, 19(4):1559--1574, 1991.

\bibitem{SSSWX:18}
Zhan Shi, Vladas Sidoravicius, He~Song, Longmin Wang, and Kainan Xiang.
\newblock Uniform spanning forests associated with biased random walks on
  {E}uclidean lattices.
\newblock arXiv:1805.01615, 2018.

\bibitem{Sun:16}
Wangru Sun.
\newblock Toroidal {D}imer {M}odel and {T}emperley's {B}ijection.
\newblock arXiv:1603.00690, 2016.

\bibitem{Tem:74}
H.~N.~V. Temperley.
\newblock In {C}ombinatorics: {P}roceedings of the {B}ritish combinatorial
  conference.
\newblock {\em London Mathematical Society Lecture Notes Series}, pages
  202--204, 1974.

\bibitem{Thu:90}
William~P. Thurston.
\newblock Conway's tiling groups.
\newblock {\em Amer. Math. Monthly}, 97(8):757--773, 1990.

\bibitem{Wil:96}
David~Bruce Wilson.
\newblock Generating random spanning trees more quickly than the cover time.
\newblock In {\em Proceedings of the {T}wenty-eighth {A}nnual {ACM} {S}ymposium
  on the {T}heory of {C}omputing ({P}hiladelphia, {PA}, 1996)}, pages 296--303,
  New York, 1996. ACM.

\bibitem{You:09}
Ben Young.
\newblock Squishing dimers on the hexagon lattice.
\newblock {\em Electron. J. Combin.}, 16(1):Research Paper 86, 20, 2009.

\end{thebibliography}

\end{document}